\documentclass[a4paper, 11pt, oneside]{article}
\usepackage{comment}
\usepackage{lipsum}
\usepackage{fullpage}
\usepackage[margin=1in]{geometry} 
\usepackage{amsmath,amsthm,amssymb,amsfonts}
\usepackage[english]{babel}
\usepackage[utf8]{inputenc}
\usepackage{amsmath,amsfonts}
\usepackage{enumitem}
\usepackage{stackrel}
\usepackage{mathtools,bm}
\usepackage{mathrsfs}
\usepackage{comment}
\usepackage{fullpage}
\usepackage[margin=1in]{geometry} 
\usepackage{amsmath,amsthm,amssymb,amsfonts}
\usepackage{float}
\usepackage{amsmath,amsfonts}
\usepackage{enumitem}
\usepackage{stackrel}
\usepackage[unicode]{hyperref}
\usepackage{mathtools,bm}
\usepackage{graphicx}
\usepackage{dsfont}
\usepackage{quiver}

\usepackage[T1]{fontenc}

\usepackage{csquotes}
\usepackage[backend=biber]{biblatex}
\usepackage{graphicx}
\graphicspath{ {./graphics/} }

\newtheorem{theorem}{Theorem}[section]
\newtheorem{proposition}{Proposition}[section]
\newtheorem{corollary}{Corollary}[theorem]
\newtheorem{lemma}[theorem]{Lemma}
\newtheorem{definition}{Definition}[section]

\newtheorem{notation}{Notation}[section]

\newcommand\restr[2]{{% we make the whole thing an ordinary symbol
  \left.\kern-\nulldelimiterspace % automatically resize the bar with \right
  #1 % the function
  \vphantom{\big|} % pretend it's a little taller at normal size
  \right|_{#2} % this is the delimiter
  }}

\usepackage{tikz}
\usepackage{tikz-cd}
\usetikzlibrary{matrix, calc, arrows}

\usepackage{authblk}

\begin{document}
\title{Homotopy Bicategories of $2$-fold Segal Spaces}
\author{Jack Rom\"o\footnote{Supported by the Leeds Doctoral Scholarship from the Faculty of Engineering and Physical Sciences at the University of Leeds, 2020-2024.}}

\affil{\textit{Email:} \texttt{\href{mailto:jack.t.romo@gmail.com}{jack.t.romo@gmail.com}}}

\maketitle

\begin{abstract}
In this paper, we address the construction of homotopy bicategories of $(\infty,2)$-categories, which we take as being modeled by 2-fold Segal spaces. Our main result is the concrete construction of a functor $h_2$ from the category of Reedy fibrant $2$-fold Segal spaces (each decorated with chosen sections of the Segal maps) to the category of unbiased bicategories and pseudofunctors between them.
Though our construction depends on the choice of sections, we show that for a given $2$-fold Segal space, all possible choices yield the same unbiased bicategory up to an equivalence that acts as the identity on objects, morphisms and $2$-morphisms. We illustrate our results with the example of the fundamental bigroupoid of a topological space.
\end{abstract}
\section{Introduction}

The development of higher category theory has proven to follow a complete reversal of the usual timeline for mathematical research. One should usually expect to define the constructions of interest and later prove theorems they satisfy. Higher category theorists have instead found themselves on a different trajectory, developing definitions with the intent of satisfying certain theorems. In the face of such a conundrum, one may find themselves questioning what is theorem or definition. For instance, a desirable property of higher groupoids is the \emph{homotopy hypothesis}, posed by Grothendieck in \cite{grothendieckPursuingStacks}, which declares that $n$-groupoids should model homotopy $n$-types. Some models of higher category prove this as a theorem, like the groupoidal weakly globular $n$-fold categories of Paoli \cite{paoliWeaklyGlobularNfold2016}, while many instead take it as an axiom in their definitions, such as quasicategories \cite{joyalQuasicategoriesKanComplexes2002}, complicial sets \cite{verityWeakComplicialSets2008}, Segal $n$-categories \cite{hirschowitzSimpsonDescentePourNchamps2001}, complete $n$-fold Segal spaces \cite{barwickCatClosedModel2005}, $\Theta_n$-spaces \cite{bergnerRezkComparisonModelsCategories2013} and $n$-quasi-categories \cite{araHigherQuasicategoriesVs2014}.

Due in part to these variable foundations, it has become somewhat challenging to establish translations between all the diverse definitions of higher category. In exchange however, each correspondence constructed sheds new light on both models being considered, as one must devise means to extract definition from theorem or theorem from definition. One example of such a correspondence is that of a \emph{homotopy category}. Take an $\infty$-groupoid $X$. Should we believe the homotopy hypothesis, where the objects of $X$ may be seen as points, $1$-cells as paths and so on in some topological space that we may write as $\lvert X \rvert$, one may consider the set of path components $\pi_0(X) \cong \pi_0(\lvert X \rvert)$, which is just the set of equivalence classes of objects in $X$. A somewhat frivolous name for this could be the `homotopy set' of $X$. In a similar manner, one may consider an $(\infty, 1)$-category $Y$, where the $\infty$-category of morphisms $\textbf{Hom}_Y(x, y)$ for any objects $x, y \in Y$ is an $\infty$-groupoid. We may use $\pi_0$ to obtain a category from $Y$, called the homotopy category, whose objects are the same and whose hom-sets take the form $\pi_0(\textbf{Hom}_Y(x, y))$.

In general, there should be a means to construct from any $(\infty, n)$-category $X$ its \emph{homotopy $n$-category} $h_n(X)$, for $n \geq 0$, by taking the $k$-cells of $h_n(X)$ to be those of $X$ for $k < n$ and the $n$-cells in $h_n(X)$ to be equivalence classes of $n$-cells in $X$. The difficulty of completing such a construction revolves around the models of $(\infty, n)$-category and $n$-category chosen. We take special interest in comparing models which are `homotopy-theoretic', which assume the homotopy hypothesis and use it to control coherence conditions, and the `algebraic' models, which specify composition operations and coherence morphisms directly. A homotopy $n$-category functor from a model in the former camp to the latter would amount to, for a given $(\infty, n)$-category $X$ and some composable diagram of morphisms $P$ in $X$, taking the contractible space $X_P$ of composites available for $P$ in $X$ and choosing particular points in $X_P$ to be the chosen composites of $P$ and particular higher paths to be the coherence isomorphisms between them. For instance, if $P$ is a diagram $x \xrightarrow{f} y \xrightarrow{g} z$, the space $X_P$ would contain all potential composites of the morphisms $g$ and $f$. One would be tasked with choosing points in $X_P$ representing composites like $g \circ f, (g \circ 1_y) \circ f$ or $g \circ (1_y \circ (f \circ 1_x))$, together with chosen paths between these points representing composites of associators and unitors and higher paths giving higher coherences between these.

A case of particular interest is the comparison of \emph{$2$-fold Segal spaces}, as defined for instance in \cite{bergnerRezkComparisonModelsInfty2020} or \cite{johnson-freydScheimbauerOpLaxNatural2017}, with classical bicategories. General $n$-fold Segal spaces are `homotopy-theoretic' models of $(\infty, n)$-category defined inductively. One starts from $\infty$-groupoids $\textbf{SeSp}_0 := \infty\textbf{Grpd}$, assumed to be modeled by spaces of some form by the homotopy hypothesis. Given a suitable category $\textbf{SeSp}_{n-1}$ of $(n-1)$-fold Segal spaces equipped with a sensible notion of equivalences thereof, one then defines an $n$-fold Segal space to be a functor $X : \Delta^{op} \rightarrow \textbf{SeSp}_{n-1}$ sending $[n] \mapsto X_n$, such that the \emph{Segal maps}
$$
    X_n \rightarrow X_1 \times^h_{X_0} \cdots \times^h_{X_0} X_1
$$
are equivalences in $\textbf{SeSp}_{n-1}$. This yields a category $\textbf{SeSp}_n$, where the equivalences $X \rightarrow Y$ are the maps such that the levelwise maps $X_k \rightarrow Y_k$ for all $k \geq 0$ are equivalences in $\textbf{SeSp}_{n-1}$.

Advancements in constructing extended topological quantum field theories in full generality have proceeded in the model of symmetric monoidal $n$-fold Segal spaces  \cite{ calaqueScheimbauerNoteInftyCategory2019, gradyPavlovExtendedFieldTheories2022, lurieClassificationTopologicalField2009, schommerPriesInvertibleField2017}, while much work on once-extended such theories has been done using symmetric monoidal bicategories \cite{bartlettDouglasSchommerPriesVicaryExtended3DimensionalBordismTheory2014, martinsPorterCategorificationQuinnFinite2023, pstragowskiDualizableObjectsMonoidal2014, schommer-priesClassificationTwoDimensionalExtended2014}. Relating these two models by a homotopy bicategory construction, sending a $2$-fold Segal space $X$ to a suitable bicategory $h_2(X)$, paves the way to comparing results established in each model, such as applying the cobordism hypothesis proven for $n$-fold Segal spaces to known constructions of bicategorical topological quantum field theories.

In this paper, we explicitly construct a functor
$$
    h_2 : \textbf{SeSp}_2^{comp} \rightarrow \textbf{UBicat}
$$
from the category $\textbf{SeSp}_2^{comp}$ of (Reedy fibrant) $2$-fold Segal spaces, each decorated with chosen sections of the \emph{Segal maps}, to the category of \emph{unbiased bicategories} as defined in \cite{leinsterOperadsHigherDimensionalCategory2000}. Unbiased bicategories are equivalent to bicategories, though allow composition operations that compose chains of $n$ morphisms for all $n$ directly, rather than just allowing a single binary composition operation that must be applied many times to compose such a chain. For instance, in an unbiased bicategory $\mathscr{B}$, a chain of $1$-morphisms $X \xrightarrow{f} Y \xrightarrow{g} Y \xrightarrow{h} Z$ can be directly composed with a trinary composition operation into $h \circ g \circ f$, whereas in a classical bicategory we could only form the composites $h \circ (g \circ f)$ and $(h \circ g) \circ f$. Using unbiased bicategories highlights the $n$-ary horizontal composition operations implicitly available for every $n$ in a $2$-fold Segal space, obtained by weakly inverting all the Segal maps.

Another advantage to using unbiased bicategories is that they readily generalize to known \emph{algebraic} models of $n$-category for higher $n$, such as the Batanin-Leinster models of $\infty$-category \cite{leinsterHigherOperadsHigher2004}, which likewise support unbiased composition operations. Moreover, our unbiased approach, being fundamentally operadic by design, is similar in nature to known operadic models of algebraic $n$-category, like Trimble $n$-categories \cite{chengComparingOperadicTheories2011}. It is the author's hope that the techniques presented in this paper will lead to homotopy $n$-category models from $n$-fold Segal spaces or similar homotopy-theoretic models to algebraic models of $n$-category, leading to a better understanding of the relationship between these two camps of model for higher category theory.

\subsection{Previous Constructions}

There are several known constructions of homotopy bicategory in the literature relevant to the goals of this paper. One is the approach of Moser in \cite{moserDoubleCategoricalNerveDoubleCategories2021}, where a functor from bisimplicial spaces to strict $2$-categories is developed that restricts to a suitable homotopy $2$-category functor from Reedy fibrant \emph{complete} $2$-fold Segal spaces. While this construction defines a Quillen pair between strict $2$-categories and complete $2$-fold Segal spaces, it is designed to only produce strict $2$-categories.
The major developments in once-extended TQFTs using symmetric monoidal bicategories in the literature, such as the work of Schommer-Pries \cite{schommer-priesClassificationTwoDimensionalExtended2014}, define the relevant higher categories of manifolds, cobordisms and higher cobordisms with non-trivial associators and unitors. Thus, we would like our homotopy bicategory functor to explicitly produce weak bicategories from the incoming $(\infty, 2)$-categories, obtaining associators and unitors from the structure of a $2$-fold Segal space. This will help ensure that our homotopy bicategory construction will translate $(\infty, 2)$-categories of cobordisms constructed via $2$-fold Segal spaces, such as those of Calaque and Scheimbauer \cite{calaqueScheimbauerNoteInftyCategory2019} or Grady and Pavlov \cite{gradyPavlovExtendedFieldTheories2022}, directly into bicategorical such constructions that can be more easily compared with current developments in this domain. Indeed, an outline for such a comparison has been given in Scheimbauer's thesis \cite{scheimbauerThesis}, which we hope our work here may help to formalize.

One should also consider the approach of Campbell in \cite{campbellHomotopyCoherentCellular2020}, where a homotopy bicategory functor from \emph{$2$-quasicategories} to bicategories is defined. The construction itself is appealing in its simplicity, much like the approach of Moser, and admits several similar technical advantages, such as being participant in a Quillen pair. Applying this construction to our context would require a suitable functor $\textbf{SeSp}_2 \rightarrow \textbf{qCat}_2$ from $2$-fold Segal spaces to $2$-quasicategories. This should define an equivalence of the models, either a Quillen equivalence or similarly an equivalence of $(\infty, 1)$-categories. The potential means to obtain such a functor currently in the literature, as far as the author is aware, are given by the chains in the (non-commuting) diagram
\[\begin{tikzcd}
	{\textbf{SeSp}_2} & {\textbf{CSSP}_2} && {\Theta_2\textbf{Sp}} && {\textbf{qCat}_2}
	\arrow["C", from=1-1, to=1-2]
	\arrow[""{name=0, anchor=center, inner sep=0}, "{d_\ast}", curve={height=-18pt}, from=1-2, to=1-4]
	\arrow[""{name=1, anchor=center, inner sep=0}, "{d^\#}"', curve={height=18pt}, from=1-2, to=1-4]
	\arrow[""{name=2, anchor=center, inner sep=0}, "{d^\ast}"{description}, from=1-4, to=1-2]
	\arrow["{(-)_0}", curve={height=-12pt}, from=1-4, to=1-6]
	\arrow["{\textbf{Sing}_W}"', curve={height=12pt}, from=1-4, to=1-6]
	\arrow["\dashv"{anchor=center, rotate=85}, draw=none, from=1, to=2]
	\arrow["\dashv"{anchor=center, rotate=95}, draw=none, from=2, to=0]
\end{tikzcd}\]
The functor $C : \textbf{SeSp}_2 \rightarrow \textbf{CSSP}_2$ is a \emph{completion} functor, converting a $2$-fold Segal space to the special case of a \emph{complete} $2$-fold Segal space. An example of such a functor can be found for instance in \cite[Thm. 1.2.13]{lurieInfty2CategoriesGoodwillie2009}. The functor $d^\ast : \Theta_2\textbf{Sp} \rightarrow \textbf{CSSP}_2$ from \emph{$\Theta_2$-spaces} to complete $2$-fold Segal spaces is a simple pullback functor. However, the functor $d_\ast$, the pushforward, is less clear explicitly; it is given by a coend that may be complicated to compute directly. Moreover, $d_\ast$ is right Quillen on the injective model structures for the relevant categories of simplicial presheaves \cite{bergnerRezkComparisonModelsInfty2020}, so the overall composite of functors from $\textbf{SeSp}_2$ to bicategories may not be left Quillen, as one may expect from a homotopy bicategory functor and as is established for instance in \cite{campbellHomotopyCoherentCellular2020} and \cite{moserDoubleCategoricalNerveDoubleCategories2021}. Moreover, the alternative functor $d^\#$ is Quillen on the \emph{projective} model structure rather than the \emph{injective} structure \cite{bergnerRezkComparisonModelsInfty2020}, so an injective fibrant replacement or some similar adjustment may be required. The two functors from $\Theta_2$-spaces to $2$-quasicategories are defined in \cite{araHigherQuasicategoriesVs2014}, while the three functors between $\Theta_2$-spaces and complete $2$-fold Segal spaces are established and studied in \cite{bergnerRezkComparisonModelsInfty2020}. Studying this conversion seems highly nontrivial and its composition with Campbell's homotopy bicategory functor is left to future work.

Many other related constructions exist in the literature, whose relationship with the author's construction are left to future work. `Homotopy sets,' more commonly just referred to as sets of path components, are easily constructed from Kan complexes. Homotopy categories of quasicategories are classical \cite[pg. 24]{lurieHigherToposTheory2008}, while Rezk has established homotopy categories for complete Segal spaces \cite{rezkModelHomotopyTheory2000}. Campbell, beyond his construction of homotopy bicategory, notes in \cite{campbellHomotopyCoherentCellular2020} many `nerve' functors which should be right adjoint to suitable notions of homotopy bicategory, such as the functor from strict $2$-categories to quasicategory-enriched categories defined by enriching along the nerve functor of quasicategories, the nerve of a bicategory as a $2$-dimensional Postnikov complex due to Duskin \cite{duskinSimplicialMatricesNerves2001} and Gurski \cite{gurskiNervesBicategoriesStratified2009}, the nerve of strict $2$-category as a $2$-precompicial set due to Ozornova and Rovelli \cite{ozornovaRovelliNerves2categories2categorification2021} and the composition of Lack and Paoli's nerve for bicategories \cite{lackPaoli2nervesBicategories2008} with change of base by the quasicategory nerve functor to obtain Reedy fibrant quasicategory-enriched Segal categories. A number of such nerve constructions are compared by Moser, Ozornova and Rovelli in \cite{moserOzornovaRovelliModelIndependence2022}, though all of these nerves are, save the work of Campbell, defined for strict $2$-categories.

Johnson-Freyd, Calaque and Scheimbauer have suggested a notion of homotopy bicategory for projective fibrant complete $2$-fold Segal spaces in \cite{johnson-freydScheimbauerOpLaxNatural2017}, \cite{scheimbauerThesis} and \cite{calaqueScheimbauerNoteInftyCategory2019}. In fact, it is this last approach that underpins our method for defining homotopy bicategories; our construction is a Reedy fibrant implementation of their desired definition. We proceed in the Reedy fibrant case to make available certain necessary lifting problems for our construction to succeed. It would also be of interest to compare with the suggested means to obtain algebraic coherence data from a weakly enriched $\infty$-category by Gepner and Haugseng in \cite[Remark 2.2.19]{gepnerHaugsengEnrichedCategories2015}, following the equivalence between this model of higher category and $2$-fold Segal spaces discussed in \cite{haugsengRectificationEnriched2015}.

\textbf{Acknowledgements.} This paper was written under the careful and endlessly inspiring guidance of the author's PhD supervisors, Jo\~ao Faria Martins and Nicola Gambino, to who endless thanks are owed for making the results of this work possible. The author also wishes to thank Julie Bergner and Ross Street for helpful e-mail correspondences, as well as Udit Mavinkurve and Tim Porter for suggesting useful references. The author learned of many of the potential comparison functors forming pathways between 2-fold Segal spaces and 2-quasicategories from the \emph{$(\infty, 2)$-categories workshop} run by Martina Rovelli, Nima Rasekh and Lyne Moser; the author extends his gratitude to both these organizers and all participants. The author also thanks the developers of \url{https://q.uiver.app/}, with which the commutative diagrams in this paper were constructed. Much of this research was funded by the Leeds Doctoral Scholarship from the Faculty of Engineering and Physical Sciences at the University of Leeds, to whom the author gives his gratitude.

Significant portions of this paper are taken vertabim from the author's thesis \cite{romoTowardsAlgebraicNCategories}, which is largely an augmented and corrected version of an earlier version of this article. The reader seeking more background in model categories or curious about the \emph{projective fibrant} case of $2$-fold Segal spaces, which underpins the models of higher categories of cobordisms studied in \cite{lurieClassificationTopologicalField2009} and \cite{calaqueScheimbauerNoteInftyCategory2019}, is referred to the author's thesis for relevant supplementary material.

\tableofcontents
\section{2-fold Segal Spaces}

In producing a comparison between two models of higher category, our first port of call is to establish the fine-print of the models we have chosen. Our starting point is that of $(\infty, 2)$-categories, which for us, are \emph{complete $2$-fold Segal spaces}. In fact, we will work with the weaker notion of a \emph{$2$-fold Segal space}.

Our choice of complete $2$-fold Segal spaces as a model for $(\infty, 2)$-categories comes largely from their connections to fully extended topological quantum field theories \cite{johnson-freydScheimbauerOpLaxNatural2017} \cite{calaqueScheimbauerNoteInftyCategory2019} \cite{lurieClassificationTopologicalField2009}, along with their capacity to be inductively extended to a model of $(\infty, n)$-category \cite{barwickCatClosedModel2005} \cite{bergnerRezkComparisonModelsInfty2020} \cite{calaqueScheimbauerNoteInftyCategory2019} \cite{lurieInfinityCategoriesGoodwillie2009} \cite{barwickSchommerPriesUnicityHomotopyTheory2020}; extending the constructions in this paper to general $n$ will be future work. The form of homotopy bicategory presented here is also a natural extension of the construction discussed in \cite{calaqueScheimbauerNoteInftyCategory2019}, \cite{scheimbauerThesis} and \cite{johnson-freydScheimbauerOpLaxNatural2017}, though applied this time to Reedy fibrant $2$-fold Segal spaces so that certain lifting problems are made available.

\subsection{Model Structures on $n$-uple Simplicial Spaces}

Before we can proceed much further, we should take the time to establish some of the important infrastructure for our definitions. Our formalism is inspired by the discussion in \cite[pg. 5-7]{rezkModelHomotopyTheory2000}, though primarily applies results from \cite{hirschhornModelCategoriesTheir2009}. It should be noted thus that all our model categories are assumed to have functorial factorizations, as in \cite[Def. 7.1.3]{hirschhornModelCategoriesTheir2009}.

For notational purposes, we write $\textbf{sSet}$ for the category of simplicial sets. We choose to write $\langle f_0, \cdots, f_m \rangle : [m] \rightarrow [n]$ for the map in $\Delta$ such that $i \mapsto f_i$ for all $0 \leq i \leq m$. There are special maps
$$
d_i^n := \langle 0, \cdots, i-1, i+1, \cdots, n \rangle : [n] \rightarrow [n+1]
$$
which we dub the \emph{coface maps}, and
$$
s_i^n := \langle 0, \cdots, i, i, \cdots, n \rangle : [n] \rightarrow [n-1]
$$
which we call the \emph{codegeneracy maps}. We will often omit the superscript when it is evident.

In general, for any \emph{simplicial object} $X$ in a category $\mathscr{C}$, meaning a functor $X : \Delta^{op} \rightarrow \mathscr{C}$, we write $X_i := X([i])$ and $X_f : X_i \rightarrow X_j$ for $X(f)$, given a map $f : [j] \rightarrow [i]$ in $\Delta$. For instance, there are maps $X_{d_0}, X_{d_1}, X_{d_2} : X_2 \rightarrow X_1$ and $X_{\langle 0, 4, 7 \rangle} : X_9 \rightarrow X_2$. For a simplicial set $X$, we call the maps $X_{d_i^n}$ the \emph{face maps} and $X_{s_i^n}$ the \emph{degeneracy maps}.

Write $\Delta[n]$ for the representable simplicial set from $[n]$ and $\partial \Delta[n]$ for the maximal subobject of $\Delta[n]$ not containing the identity $1_{[n]} : [n] \to [n]$. Here, a `subobject' simply means a levelwise subset.

\begin{definition} \label{defn_sesp_sp_k}
    For $n \geq 1$ and $k \geq 1$, let
    $$
        Sp(n) := \Delta[1] \sqcup_{\Delta[0]} \cdots \sqcup_{\Delta[0]} \Delta[1] \in \textbf{sSet}
    $$
    and set $g_n : Sp(n) \hookrightarrow \Delta[n]$ to be the inclusion $\langle 0, 1 \rangle \sqcup_{\langle 1 \rangle} \cdots \sqcup_{\langle n-1 \rangle} \langle n-1, n \rangle$. For $n=0$, let $Sp(0) = \Delta[0]$ and $g_0 = 1_{\Delta[0]}$.
\end{definition}

\begin{definition}[{\cite[pg. 60]{heutsMoerdijkSimplicialDendroidalHomotopy2022}}]
    Let $k \geq 0$ and $X \in \textbf{sSet}$. Then the \emph{$k$-skeleton} of $X$, written $\textbf{sk}_k X$, is the minimal subobject of $X$ such that $(\textbf{sk}_k X)_n = X_n$ for all $0 \leq n \leq k$.
\end{definition}

It is known that $\textbf{sSet}$ is a simplicial model category and is Cartesian closed, where the exponent $Y^X$ defining the enrichment for $X, Y \in \textbf{sSet}$ is such that
$$
    (Y^X)_n \cong \textbf{Hom}_{\textbf{sSet}}(X \times \Delta[n], Y).
$$

We write $\iota_{\textbf{S}} : \textbf{Set} \hookrightarrow \textbf{sSet}$ for the functor sending a set $X$ to a levelwise constant simplicial set. We call simplicial sets in the (essential) image of $\iota_{\textbf{S}}$ \emph{discrete}.

There is an important classical comparison between simplicial sets and topological spaces we will need. Henceforth, let $\textbf{Top}$ be a convenient category of topological spaces, such as $k$-spaces.

\begin{definition} \label{def_ssets_geomreal}
    There is a functor $\bullet_t : \Delta \rightarrow \textbf{Top}$ sending $[n]$ to the space
    $$
        \Delta_t[n] := \{(x_0, \cdots, x_n) \in \mathbb{R}^{n+1} \mid \sum_i x_i = 1, x_i \geq 0\}
    $$
    and $f: [n] \to [m]$ to the unique affine map $\Delta_t[n] \to \Delta_t[m]$ sending each coordinate vector $e_i \in \mathbb{R}^{n+1}$ to $e_{f_i} \in \mathbb{R}^{m+1}$.
\end{definition}

\begin{definition} \label{def_ssets_sing}
    Define the functor $\textbf{Sing} : \textbf{Top} \rightarrow \textbf{sSet}$ to send $X \in \textbf{Top}$ to the simplicial set $\textbf{Sing}(X)$, defined levelwise as
    $$
        \textbf{Sing}(X)_n := \textbf{Hom}_{\textbf{Top}}(\Delta_t[n], X).
    $$
    For a map $f : X \rightarrow Y$ in $\textbf{Top}$, $\textbf{Sing}(f)$ is given by levelwise postcomposition, sending an $n$-simplex $\Delta_t[n] \rightarrow X$ to its composition with $f :X \rightarrow Y$. 
\end{definition}

There is a left adjoint $\lvert \cdot \rvert : \textbf{sSet} \rightarrow \textbf{Top}$ to $\textbf{Sing}$, sending $X \in \textbf{sSet}$ to its \emph{geometric realization} $\lvert X \rvert \in \textbf{Top}$, given by a certain coend. We will not discuss the details here.

We will need more substantial data structures than mere simplicial sets:

\begin{definition} \label{def_ssets_nuple_simp_spaces}
    The category of \emph{$n$-uple simplicial spaces}, for any $n \geq 0$, is defined to be $\textbf{sSpace}_n := \textbf{sSet}^{(\Delta^{op})^n}$.

    The category of \emph{simplicial spaces} is the category of $1$-uple simplicial spaces, which we write as $\textbf{sSpace} := \textbf{sSpace}_1$.

    The category of \emph{bisimplicial spaces} is the category of $2$-uple simplicial spaces, which we write as $\textbf{ssSpace} := \textbf{sSpace}_2$.
\end{definition}

Write $\Delta[i_0, \cdots, i_n]$ for the $n$-uple simplicial space represented by $([i_0], \cdots, [i_n]) \in \Delta^{n+1}$, noting $\textbf{sSet}^{(\Delta^{op})^n} \cong \textbf{Set}^{(\Delta^{op})^{n+1}}$. Our convention for ordering of indices follows from the evident isomorphism
$$
    \textbf{sSet}^{(\Delta^{op})^n} \cong (\cdots (\textbf{sSet}^{\Delta^{op}}) \cdots)^{\Delta^{op}}
$$
which does not permute the copies of $\Delta^{op}$. With this in mind we also write, for $X \in \textbf{sSpace}_n$, that
$$
    X_{i_1, \cdots, i_n} := ((X_{i_1}) \cdots )_{i_n}.
$$
Hence, by the Yoneda lemma,
$$
    \textbf{Hom}_{\textbf{sSpace}_n}(\Delta[i_0, \cdots, i_n], X) \cong (X_{i_n, i_{n-1}, \cdots, i_1})_{i_0}.
$$

\begin{definition} \label{def_ssets_f_n}
    Let $0 \leq k < n$. Define the map
    $$
        F_n^k : \textbf{sSpace}_k \rightarrow \textbf{sSpace}_n
    $$
    such that, for any $K \in \textbf{sSpace}_k$,
    $$
        (F_n^k(K)_{i_1, \cdots, i_n})_{j} = (K_{i_1, \cdots, i_k})_{i_{k+1}}.
    $$
    More formally, this is given by the precomposition map on presheaf categories induced by the functor $\Delta^{n+1} \rightarrow \Delta^{k+1}$ sending
    $$
        ([i_0], \cdots, [i_n]) \mapsto ([i_{n-k}], \cdots [i_n]).
    $$
\end{definition}

Note that $F_n^0(K)$ is levelwise a \emph{discrete} $(n-1)$-uple simplicial space, meaning a constant functor $(\Delta^{op})^{n-1} \rightarrow \textbf{sSet}$ whose image is discrete as a simplicial set.

We will also need to understand boundaries of our representable $n$-uple simplicial spaces, once we discuss model structures:

\begin{notation} \label{not_ssets_bdry}
    Write $\partial \Delta[i_n, \cdots, i_0]$ for the maximal sub-presheaf of $\Delta[i_n, \cdots, i_0]$ missing the element $(1_{[i_n]}, \cdots, 1_{[i_0]})$.
\end{notation}

\begin{notation} \label{not_ssets_f_rep}
    Write
    \begin{align*}
        F_n^k(i_k, \cdots, i_0) &= F_n^k(\Delta[i_k, \cdots, i_0]) \\
        \partial F_n^k(i_k, \cdots, i_0) &= F_n^k(\partial \Delta[i_k, \cdots, i_0]).
    \end{align*}
\end{notation}

\begin{lemma} \label{lemma_ssets_fnk_iso_zeros}
    There is an isomorphism
    $$
        F_n^k(i_k, \cdots, i_0) \cong \Delta[0, \cdots, 0, i_k, \cdots, i_0].
    $$
\end{lemma}

\begin{proposition} \label{prop_ssets_bdry_decomp}
    For all $n \geq 1$ and $i_0, \cdots, i_n \geq 0$, there is an isomorphism
    $$
        \partial \Delta[i_n, \cdots, i_0] \cong (\Delta[i_n, \cdots, i_1, 0] \times \partial F_n^0(i_0)) \sqcup_{\partial \Delta[i_n, \cdots, i_1, 0] \times \partial F_n^0(i_0)} (\partial \Delta[i_n, \cdots, i_1, 0] \times F_n^0(i_0)).
    $$
\end{proposition}

\begin{proof}
    The pushout contains all maps in $\Delta^n$ of the form $([j_n], \cdots, [j_0]) \rightarrow ([i_n], \cdots, [i_0])$ such that either $[j_0] \rightarrow [i_0]$ is not a surjection or such that the maps $[j_h] \rightarrow [i_h]$ are not all surjections for $1 \leq h \leq n$. Thus, all maps which are not surjections on every coordinate are contained, which is an alternative characterization of $\partial \Delta[i_n, \cdots, i_0]$.
\end{proof}

\begin{notation} \label{not_ssets_nuple_k_enrichment}
    For any $0 \leq k \leq n$ and some $X, Y \in \textbf{sSpace}_n$, write $\textbf{Map}_n^k(X, Y)$ for the induced $k$-uple mapping space between $X$ and $Y$. More concretely, this is
    $$
        \textbf{Map}_n^k(X, Y) = (Y^X)_{\underbrace{0, \cdots, 0}_{n-k}}.
    $$
    If $n = k$, this is just the inner hom.
\end{notation}

\begin{proposition} \label{prop_ssets_enrich_f_rep}
    For any $0 \leq k < n$ and $i_0, \cdots, i_k \geq 0$,
    $$
        \textbf{Map}_n^{n-k-1}(F_n^k(i_k, \cdots, i_0), X) \cong X_{i_0, \cdots, i_k}.
    $$
\end{proposition}

\begin{proof}
    We can compare the two sides levelwise. For $j_0, \cdots, j_{n - k - 1} \geq 0$:
    \begin{align*}
        &(\textbf{Map}_n^{n-k-1}(F_n^k(\Delta[i_k, \cdots, i_0]), X)_{j_0, \cdots, j_{n-k-2}})_{j_{n - k - 1}} = ((X^{F_n^k(\Delta[i_k, \cdots, i_0])})_{0, \cdots, 0, j_0, \cdots, j_{n-k-2}})_{j_{n - k - 1}} \\
        &\cong \textbf{Hom}_{\textbf{sSpace}_n}(\Delta[j_{n-k-1}, \cdots, j_0, 0, \cdots, 0], X^{F_n^k(\Delta[i_k, \cdots, i_0])}) \\
        &\cong \textbf{Hom}_{\textbf{sSpace}_n}(F_n^k(\Delta[i_k, \cdots, i_0]) \times \Delta[j_{n-k-1}, \cdots, j_0, 0, \cdots, 0], X) \\
        &\cong \textbf{Hom}_{\textbf{sSpace}_n}(\Delta[0, \cdots, 0, i_k, \cdots, i_0] \times \Delta[j_{n-k-1}, \cdots, j_0, 0, \cdots, 0], X) \\
        &\cong \textbf{Hom}_{\textbf{sSpace}_n}(\Delta[j_{n-k-1}, \cdots, j_0, i_k, \cdots, i_0], X) \cong (X_{i_0, \cdots, i_k, j_0, \cdots, j_{n-k-2}})_{j_{n - k - 1}}
    \end{align*}
    which completes the proof.
\end{proof}

The case of $k = 0$ and $n = 1$ is discussed in \cite[pg. 6]{rezkModelHomotopyTheory2000}. More generally, the case $k = 0$ and $n \geq 1$ is discussed in \cite[pg. 6]{rezkCartesianPresentationWeak2010}. Note that $\textbf{Map}_n^0$ defines an enrichment of $\textbf{sSpace}_n$ in simplicial sets.

Finally, we need another means to embed lower-dimensional simplicial spaces into higher-dimensional ones, by producing a space that is levelwise constant:

\begin{definition} \label{def_ssets_iota_n}
    Let $0 \leq k < n$. Define the map
    $$
        \iota_n^k : \textbf{sSpace}_k \rightarrow \textbf{sSpace}_n
    $$
    such that, for any $K \in \textbf{sSpace}_k$ and $i_0, \cdots, i_n \geq 0$,
    $$
        (\iota_n^k(K)_{i_0, \cdots, i_{n-1}})_{i_n} = (K_{i_{n-k}, \cdots, i_{n-1}})_{i_n}.
    $$
    More formally, this is given by the precomposition map on presheaf categories induced by the functor $\Delta^{n+1} \rightarrow \Delta^{k+1}$ sending
    $$
        ([i_n], \cdots, [i_0]) \mapsto ([i_n], \cdots [i_{n-k}]).
    $$
\end{definition}

\begin{notation} \label{not_ssets_tensor_simp_sp}
    For $0 \leq k < n$, write $- \square_n^k - : \textbf{sSpace}_n \times \textbf{sSpace}_k \rightarrow \textbf{sSpace}_n$ for the functor defined such that, for $X \in \textbf{sSpace}_n$ and $K \in \textbf{sSpace}_k$,
    $$
        X \square_n^k K := X \times \iota_n^k(K).
    $$
    If $k = n$, say that $\square_n^k := \times$ is the Cartesian product on $\textbf{sSpace}_n$.
\end{notation}

We now turn to model structures. In general, we will assume the injective model structure on $\textbf{sSpace}_n$, which we write as $\textbf{sSpace}_n^{inj}$. We have by \cite{bergnerRezkReedyCategoriesVarTheta2013} that $\Delta$ is an \emph{elegant Reedy category}, so this is in fact precisely the Reedy model structure. The weak equivalences and cofibrations are levelwise, meaning the cofibrations are simply levelwise injections. Hence, every object in $\textbf{sSpace}_n$ is Reedy cofibrant.

An inductive characterization of the fibrations, the only nontrivial part of this model structure, is as follows:

\begin{definition} \label{def_reedy_reedystr_sspace_n}
    Let $f : X \rightarrow Y$ be a map in $\textbf{sSpace}_n$. Then $f$ is a \emph{Reedy fibration of $n$-uple simplicial spaces} if and only if, for every $k \geq 0$, the map
    $$
        X_k \rightarrow Y_k \times_{\textbf{Map}_n^{n-1}(\partial F_n^0(k), Y)} \textbf{Map}_n^{n-1}(\partial F_n^0(k), X)
    $$
    is a fibration in the Reedy model structure on $\textbf{sSpace}_{n-1}$.
\end{definition}

A more elementary description not invoking Reedy fibrations of simplicial spaces is also possible, by instead taking the underlying model category to be $\textbf{sSet}$ and the Reedy category to be $(\Delta^{op})^n$. It is shown by \cite[Theorem 15.5.2]{hirschhornModelCategoriesTheir2009} that these choices will yield identical model structures on $\textbf{sSpace}_n$. We thus have:

\begin{proposition} \label{prop_reedy_levelwise_fib}
    Let $n \geq 1$. A map $f : X \rightarrow Y$ is a Reedy fibration in $\textbf{sSpace}_n$ if and only if, for every $i_1, \cdots, i_n \geq 0$, the map
    $$
        X_{i_1, \cdots, i_n} \rightarrow Y_{i_1, \cdots, i_n} \times_{\textbf{Map}_n^0(\partial F_n^{n-1}(i_n, \cdots, i_1), X)} \textbf{Map}_n^0(\partial F(i_n, \cdots, i_1), Y)
    $$
is a Kan fibration.
\end{proposition}

We may also show that the mapping spaces $\textbf{Map}_n^k(-, -)$ interact neatly with this model structure in a manner outlined in \cite{mayPontoMoreConciseAlgebraic2012}. Before we can show this, we must first show a useful fact about this enrichment:

\begin{definition}[{\cite[pg. 329]{mayPontoMoreConciseAlgebraic2012}}]
    Let $\mathscr{V}$ be a bicomplete closed symmetric monoidal category. A $\mathscr{V}$-enriched category $\mathscr{M}$ is \emph{$\mathscr{V}$-bicomplete} if and only if it is bicomplete and has all tensors and cotensors.
\end{definition}

\begin{proposition}\label{prop_enrich_sspace_adj_2_vars}
    $\textbf{sSpace}_n$ is $\textbf{sSpace}_k$-bicomplete for all $0 \leq k \leq n$.
\end{proposition}

\begin{proof}
    That $\textbf{sSpace}_n$ is bicomplete is immediate from it being a category of presheaves. For tensors and cotensors, consider setting the tensor $\odot := \square^k_n$ and the cotensor $\pitchfork := (-)^{\iota^k_n(-)}$. Note that $\iota_n^k$ is left adjoint to $\rho_n^k$, as proven in for instance \cite[Prop. 2.1.46]{romoTowardsAlgebraicNCategories}. Then we have the natural bijection induced by the adjunction between $\iota_n^k$ and $\rho_n^k$ of the form
    \begin{align*}
        \textbf{Hom}_{\textbf{sSpace}_n}(X \odot V, Y) &= \textbf{Hom}_{\textbf{sSpace}_n}(X \square_n^k V, Y) \\
        &= \textbf{Hom}_{\textbf{sSpace}_n}(X \times \iota_n^k(V), Y) \\
        &\cong \textbf{Hom}_{\textbf{sSpace}_n}(\iota_n^k(V), Y^X) \\
        &\cong \textbf{Hom}_{\textbf{sSpace}_k}(V, \textbf{Map}_n^k(X, Y)).
    \end{align*}
    Moreover, we have a natural bijection induced by the same adjunction of the form
    \begin{align*}
        \textbf{Hom}_{\textbf{sSpace}_k}(V, \textbf{Map}_n^k(X, Y)) &\cong \textbf{Hom}_{\textbf{sSpace}_n}(\iota_n^k(V), Y^X) \\
        &\cong \textbf{Hom}_{\textbf{sSpace}_n}(\iota_n^k(V) \times X, Y) \\
        &\cong \textbf{Hom}_{\textbf{sSpace}_n}(X, Y^{\iota_n^k(V)}).
    \end{align*}
    These bijections establish the conditions for tensors and cotensors as needed.
\end{proof}

In general, we will not use the notation of $\odot$ and $\pitchfork$ and simply write $\square_n^k$ and $(-)^{(-)}$. Note in the latter case that we omit the $\iota_n^k$ implicit in the exponent.

We then have the following classical result:

\begin{lemma}[{\cite[Lemma 16.4.5]{mayPontoMoreConciseAlgebraic2012}}] \label{lemma_enrich_quill_bifunc_equiv}
    Suppose $\mathscr{M}$ and $\mathscr{V}$ are model categories such that $\mathscr{V}$ is a bicomplete closed symmetric monoidal category and $\mathscr{M}$ is $\mathscr{V}$-bicomplete. Let $\odot$ and $\pitchfork$ be a tensor and cotensor on $\mathscr{M}$. Then the following conditions are equivalent:
    \begin{enumerate}
        \item Given a cofibration $f : X \rightarrow Y$ in $\mathscr{M}$ and a cofibration $p : U \rightarrow V$ in $\mathscr{V}$, the induced map
        $$
            f \square p : (Y \odot U) \sqcup_{X \odot U} (X \odot V) \rightarrow Y \odot V
        $$
        is a cofibration in $\mathscr{M}$, which is trivial if either $f$ or $p$ is.
        
        \item Given a cofibration $f : W \rightarrow X$ and a fibration $g : Y \rightarrow Z$ in $\mathscr{M}$, the induced map
        $$
            \mathscr{M}[f, g] : \mathscr{M}(X, Y) \rightarrow \mathscr{M}(W, Y) \times_{\mathscr{M}(W, Z)} \mathscr{M}(X, Z)
        $$
        is a fibration in $\mathscr{V}$, which is trivial if either $f$ or $g$ is.

        \item Given a cofibration $p : U \rightarrow V$ in $\mathscr{V}$ and a fibration $g : Y \rightarrow Z$ in $\mathscr{M}$, the induced map
        $$
            \pitchfork[p, g] : V \pitchfork Y \rightarrow U \pitchfork Y \times_{U \pitchfork Z} V \pitchfork Z
        $$
        is a fibration in $\mathscr{M}$, which is trivial if either $p$ or $g$ is.
    \end{enumerate}
\end{lemma}

\begin{proposition}[{\cite[Prop. 4.2.8]{hoveyModelCategories2007}}] \label{prop_enrich_sset_quillen}
    The category $\textbf{sSet}$, enriched over itself, with tensor given by Cartesian product and cotensor by the exponent, satisfies the equivalent conditions in Lemma \ref{lemma_enrich_quill_bifunc_equiv}.
\end{proposition}

\begin{proposition}
    The tensor and cotensor in Proposition \ref{prop_enrich_sspace_adj_2_vars} satisfy all the equivalent conditions in Lemma \ref{lemma_enrich_quill_bifunc_equiv}.
\end{proposition}

\begin{proof}
    We will prove the first condition holds. Suppose $f : U \rightarrow V$ is a cofibration in $\textbf{sSpace}_n$ and $g : W \rightarrow X$ is a cofibration in $\textbf{sSpace}_k$. As we are employing the injective model structures on these categories, these are both simply levelwise cofibrations in $\textbf{sSet}$, so checking that the pushout product $f \square g : (V \square_n^k W) \sqcup_{U \square_n^k W} (U \square_n^k X) \rightarrow (V \square_n^k X)$ is a cofibration amounts to checking this levelwise. This problem reduces to checking that for every $i_1, \cdots, i_n \geq 0$, the map
    \[\begin{tikzcd}
    	{(V_{i_1, \cdots, i_n} \times W_{i_{n-k+1}, \cdots, i_n}) \sqcup_{U_{i, \cdots, i_n} \times W_{i_{n-k+1}, \cdots, i_n}} (U_{i_1, \cdots, i_n} \times X_{i_{n-k+1}, \cdots, i_n})} \\
    	{V_{i_1, \cdots, i_n} \times X_{i_{n-k+1}, \cdots, i_n}}
    	\arrow[from=1-1, to=2-1]
    \end{tikzcd}\]
    is a cofibration in $\textbf{sSet}$. By Proposition \ref{prop_enrich_sset_quillen}, $f \square g$ is then immediately a cofibration as desired. The above reasoning extends to the case where $f$ or $g$ is a trivial cofibration, as these are again just levelwise trivial cofibrations in $\textbf{sSet}$.
\end{proof}

We thus have the following useful results, allowing us to manipulate fibrations in $\textbf{sSpace}_n^{inj}$.

\begin{corollary} \label{corr_enrich_map_n_k_cofib_fib}
    Suppose $f : U \rightarrow V$ is a Reedy cofibration and $p : Y \rightarrow Z$ a Reedy fibration in $\textbf{sSpace}_n$. Then the induced map
    $$
        \textbf{Map}_n^k(f, p) : \textbf{Map}_n^k(V, Y) \rightarrow \textbf{Map}_n^k(U, Y) \times_{\textbf{Map}_n^k(U, Z)} \textbf{Map}_n^k(V, Z)
    $$
    is a Reedy fibration in $\textbf{sSpace}_k$, which is trivial if either $f$ or $p$ is.
\end{corollary}

\begin{corollary} \label{corr_enrich_map_n_k_fibrant}
    If $Y$ is Reedy fibrant in $\textbf{sSpace}_n$ and $f : U \rightarrow V$ is a (trivial) Reedy cofibration in $\textbf{sSpace}_n$, then the map
    $$
        \textbf{Map}_n^k(f, Y) : \textbf{Map}_n^k(V, Y) \rightarrow \textbf{Map}_n^k(U, Y)
    $$
is a (trivial) Reedy fibration in $\textbf{sSpace}_k$.
\end{corollary}

We note again that this fact is not necessarily new in all cases. For instance, the case $k = 0$ and $n = 1$ is established in \cite[Prop. 2.4]{joyalTierneyQuasicategoriesVsSegal2007}. More generally, the case $k = 0$ for all $n$ is already known:

\begin{proposition} \label{prop_reedy_map_n_0_goerss}
    The mapping spaces $\textbf{Map}_n^0$, tensors $\square_n^0$ and cotensor $(-)^{\iota_n^0(-)}$ are precisely those obtained from the natural simplicial model structure on $\textbf{sSpace}_n$ described in \cite[pg. 370]{goerssJardineReedyModelCategories2009}.
    
    Moreover, it is precisely the enrichment from the natural simplicial \emph{projective} model structure on $\textbf{sSpace}_n$ obtained inductively via \cite[Def. 11.7.2]{hirschhornModelCategoriesTheir2009} and from the enrichment of $\textbf{sSet}$ by itself.
\end{proposition}

We will have use for the following notation henceforth, as $\textbf{Map}_n^k(-, -)$ will prove cumbersome to use in equations and diagrams:

\begin{notation} \label{not_cd_x_k}
    Suppose $X \in \textbf{sSpace}_k$. Let $0 \leq m < k$. Write 
    $$
        X_\bullet : (\textbf{sSpace}_m)^{op} \rightarrow \textbf{sSpace}_{k-m-1}
    $$
    for the functor $\textbf{Map}_k^{k-m-1}(F_k^m(\bullet), X)$, sending
    $$
        K \mapsto X_K := \textbf{Map}_k^{k-m-1}(F_k^m(K), X).
    $$
\end{notation}

For instance, we have for $X \in \textbf{sSpace}_k$ that
\begin{align*}
    X_{\Delta[n]} &\cong X_n = X_{n, \bullet, \cdots, \bullet} \\
    X_{Sp(n)} &\cong X_1 \times_{X_0} \cdots \times_{X_0} X_1 \\&= X_{1, \bullet, \cdots, \bullet} \times_{X_{0, \bullet, \cdots, \bullet}} \cdots \times_{X_{0, \bullet, \cdots, \bullet}} X_{1, \bullet, \cdots, \bullet} \\
    X_{\partial \Delta[n]} &\cong M_n X \\
    X_{\partial \Delta[i_k, \cdots, i_1]} &\cong M_{([i_k], \cdots, [i_1])} X
\end{align*}
where we write $M_{([i_k], \cdots, [i_1])} X$ for the matching object of $X$ at $\Delta[i_k, \cdots, i_1]$.

We have by Corollary \ref{corr_enrich_map_n_k_fibrant} that if $j : K \rightarrow L$ is a Reedy cofibration of $m$-uple simplicial spaces, the map $j^* = X_j : X_L \rightarrow X_K$ is a Reedy fibration of $(k-m-1)$-uple simplicial spaces.

In particular, setting $k = 2$ and $m = 0$ will give us, for $X \in \textbf{sSpace}_2$ and $K_1, K_2 \in \textbf{sSet}$, a simplicial space $X_{K_1}$ and a simplicial set $X_{K_1, K_2} := (X_{K_1})_{K_2}$. Moreover, if $Q \in \textbf{sSpace}$, we get a simplicial set $X_Q$.
\subsection{Segal Spaces}

Complete Segal spaces were first developed by Rezk as a model for $(\infty, 1)$-categories in \cite{rezkModelHomotopyTheory2000}. Though his intention was moreso to use complete Segal spaces as models for homotopy theories akin to model categories, it soon became apparent that his construction could be iterated to obtain a definition of $(\infty, n)$-category, called \emph{complete $n$-fold Segal spaces}. These objects are studied in \cite{lurieInfinityCategoriesGoodwillie2009}, \cite{bergnerRezkComparisonModelsInfty2020}, \cite{barwickCatClosedModel2005} and \cite{johnson-freydScheimbauerOpLaxNatural2017}. For us, the case $n=2$ will suffice. Moreover, we will not need the full notion of a complete $n$-fold Segal space; more general $n$-fold Segal spaces will have enough information for our construction to apply. Using complete $n$-fold Segal spaces is then merely a matter of restricting to these more constrained objects.

We draw from the intuitions given in \cite{lurieClassificationTopologicalField2009}. Suppose we wished to construct a model of an $(\infty, 1)$-category $X$. Where do we start? A reasonable place would be the \emph{homotopy hypothesis}, which tells us that Kan complexes are a suitable model for $\infty$-groupoids. Hence, we can at least represent the `underlying $\infty$-groupoid' $X_0 \in \textbf{sSet}$ of $X$, the result of stripping away all non-invertible higher morphisms.

How do we encode the non-invertible morphisms in $X$? Consider that a morphism should be represented by an `$\infty$-functor' $[1] \rightarrow X$, from the poset category $[1] = \{0 \rightarrow 1\}$ to $X$. Hence, we might imagine the `$\infty$-groupoid of such functors' as another simplicial set $X_1 \in \textbf{sSet}$. One might note that $X_0$ could similarly be interpreted as the $\infty$-groupoid of functors $[0] \rightarrow X$, where $[0]$ is the discrete category with one object.

We immediately consider there to be maps $s, t : X_1 \rightarrow X_0$ obtained by `precomposing' with the maps $[0] \rightarrow [1]$. These maps extract the source and target of a $1$-morphism. Moreover, there is a map $i : X_0 \rightarrow X_1$ from the projection $[1] \rightarrow [0]$, which we interpret as giving the identity map of an object.

How do we compose morphisms? This will be answered by considering a new $\infty$-groupoid $X_2 \in \textbf{sSet}$ of `functors' $[2] \rightarrow X$, where in general $[n]$ is the poset category $\{0 < \cdots < n\}$ for $n \geq 0$. Such a functor identifies a chain of two morphisms $x \overset{f}{\rightarrow} y \overset{g}{\rightarrow} z$ and a third morphism $x \overset{g \circ f}{\rightarrow} z$ such that said maps commute up to some higher equivalence - indeed, a functor between higher categories need not strictly respect composition of morphisms. There are clearly two maps $X_1 \rightarrow X_2$ given by inserting identities on the left or right, along with three maps $X_2 \rightarrow X_1$ given by extracting $f, g$ or $g \circ f$.

Being able to take compositions now reduces to demanding that the induced \emph{Segal map}
$$
    \gamma_2 : X_2 \rightarrow X_1 \times_{t, X_0, s} X_1
$$
is invertible. Then, we have a path $X_1 \times_{t, X_0, s} X_1 \rightarrow X_2 \rightarrow X_1$ sending $(f, g)$ to $g \circ f$.

Is invertibility such a good idea? Similarly to quasicategories, we will abstain from choosing a particular composite and instead ask that there is a `contractible space of options' for composites of two maps. More formally, we will ask that $\gamma_2$ is a \emph{weak equivalence} of simplicial sets. Indeed, we will go further and assert it is a trivial fibration - if everything is indeed a Kan complex, one would have that for each pair of maps $f, g \in X_1$ such that $t(f) = s(g)$, the preimage $\gamma_2^{-1}((f, g))$ would be a contractible Kan complex as desired.

For weak associativity and such, we will include spaces $X_n$ of `functors' $[n] \rightarrow X$, representing chains of length $n$ and all possible unbiased composites. Maps between these spaces are given by functors $[m] \rightarrow [n]$. We will then demand that all the remaining Segal maps
$$
    \gamma_n : X_n \rightarrow X_1 \times_{X_0} \cdots \times_{X_0} X_1
$$
are trivial fibrations. The result of this discussion is a functor $X : \Delta^{op} \rightarrow \textbf{sSet}$ known as a \emph{Segal space}.

\begin{definition}[{\cite[pg. 11]{rezkModelHomotopyTheory2000}}] \label{defn_sesp_segal_space}
    A \emph{Segal space} is a Reedy fibrant simplicial space $X : \Delta^{op} \rightarrow \textbf{sSet}$ such that the Segal maps,
    $$
        \gamma_n : X_n \rightarrow X_1 \times_{X_0} \cdots \times_{X_0} X_1,
    $$
    are weak equivalences for all $n \geq 2$.
\end{definition}

\begin{definition} \label{defn_sesp_cat_sesp}
    Let $\textbf{SeSp}$ be the full subcategory of $\textbf{sSpace}$ whose objects are the Segal spaces.
\end{definition}

Note that, because a Segal space $X$ is Reedy fibrant, the spaces $X_0$ and $X_1$ are Kan complexes and the maps $s := X_{\langle 0 \rangle} = X_{d^1_1}$ and $t := X_{\langle 1 \rangle} = X_{d^1_0}$ must both be Kan fibrations. Hence, as noted in \cite[pg. 11]{rezkModelHomotopyTheory2000}, the space $X_1 \times_{X_0} \cdots \times_{X_0} X_1$ is actually a homotopy limit, ie.
$$
    X_1 \times_{X_0} \cdots \times_{X_0} X_1 \cong X_1 \times^h_{X_0} \cdots \times^h_{X_0} X_1.
$$
Note also that the maps $\gamma_n$ are necessarily trivial fibrations \cite[pg. 11]{rezkModelHomotopyTheory2000}. Indeed, referring to Definition \ref{defn_sesp_sp_k}, we can clearly see that $\gamma_n$ is just $\textbf{Map}_1^0(F_1^0(g_n), X)$. In the case that $n < 2$, $\gamma_n$ is just an identity. This is clearly a cofibration, so $\gamma_n$ is a Reedy fibration by Corollary \ref{corr_enrich_map_n_k_fibrant} as needed.

An example is now in order. One of the classic litmus tests for any new definition of higher category is whether it supports a notion of \emph{fundamental higher groupoid} of a topological space. This structure should encode all the weak homotopy theory of the space in question. The means we present to accomplish this in the $(\infty, 1)$ and later $(\infty, 2)$ cases are by no means novel, but an explicit reference in the literature remains elusive.

\begin{definition} \label{defn_sesp_sing_ss}
    Let $\textbf{Sing}_{\textbf{sS}} : \textbf{Top} \rightarrow \textbf{sSpace}$ be the functor sending any $X \in \textbf{Top}$ to the space $\textbf{Sing}_{\textbf{sS}}(X)$, defined levelwise such that
    $$
        \textbf{Sing}_{\textbf{sS}}(X)_n := \textbf{Sing}(X^{\Delta_t[n]})
    $$
    with simplicial maps induced by precomposition.
\end{definition}

We find that
$$
    \textbf{Sing}_{\textbf{sS}}(X)_{n, m} \cong \textbf{Hom}_{\textbf{Top}}(\Delta_t[n] \times \Delta_t[m], X).
$$
Again, by \cite[pg. 41, Theorem 2]{maclaneMoerdijkSheavesGeometryLogic1994}, this has a left adjoint $\lvert \bullet \rvert_{\textbf{sS}}$ given by a certain coend. We do not take a great deal of interest in this adjoint, so we will not discuss it further.

It is prudent that we evaluate how the intuitions that led us to defining Segal spaces apply to this example. Given a space $X$, the Kan complex $\textbf{Sing}_{\textbf{sS}}(X)_0$ is just $\textbf{Sing}(X)$, which is indeed the prototypical example of an $\infty$-groupoid, namely the fundamental $\infty$-groupoid of a topological space. This will be the underlying $\infty$-groupoid of our $(\infty, 1)$-category.

The space $\textbf{Sing}_{\textbf{sS}}(X)_1$ is then the fundamental $\infty$-groupoid of the space $X^{\Delta_t[1]}$ of paths in $X$. Indeed, the $1$-morphisms in our $(\infty, 1)$-category will be paths in our topological space. Looking at $\textbf{Sing}_{\textbf{sS}}(X)_n$ reveals in general that the $n$-simplices will simply be the topological $n$-simplices in $X$.

Now, the Segal map on $\textbf{Sing}_{\textbf{sS}}(X)$ is one which sends a $2$-simplex $\Delta_t[2] \rightarrow X$ to the restriction $\lvert Sp(2) \rvert \rightarrow X$. That the Segal map is a weak equivalence is immediate: there is a deformation retract $\Delta_t[2] \rightarrow \lvert Sp(2) \rvert$ of the spine inclusion, which induces a deformation retract of the Segal map itself by precomposition. This means that, given a chain of two $1$-morphisms in $\textbf{Sing}_{\textbf{sS}}(X)_1$, they may be composed by applying this deformation retract to obtain a new $1$-morphism in $\textbf{Sing}_{\textbf{sS}}(X)$. One may in general compose a chain of length $n$ by doing the same with $\Delta_t[n]$ and $\lvert Sp(n) \rvert$.

One might expect some notion of a `mapping space' in a definition of $\infty$-category, namely an $\infty$-groupoid of higher morphisms between two fixed objects. This is easily obtained with Segal spaces:

\begin{definition}[{\cite[pg. 10]{bergnerSurveyModelsInfty2018}}] \label{defn_sesp_map_sp}
    Let $X$ be a Segal space and $x, y \in X_{0, 0}$. The \emph{mapping space} between $x$ and $y$, written here as $X(x, y)$, is defined to be the pullback
    \[\begin{tikzcd}
    	{X(x, y)} & {X_1} \\
    	{\{(x, y)\}} & {X_0 \times X_0}
    	\arrow[from=2-1, to=2-2]
    	\arrow[from=1-2, to=2-2]
    	\arrow[from=1-1, to=1-2]
    	\arrow[from=1-1, to=2-1]
    	\arrow["\lrcorner"{anchor=center, pos=0.125}, draw=none, from=1-1, to=2-2]
    \end{tikzcd}\]
    in $\textbf{sSet}$.
\end{definition}

Note that this is once again a homotopy pullback \cite[pg. 16]{bergnerSurveyModelsInfty2018}. Moreover, it is clearly fibrant, as fibrations are preserved under pullback and the discrete space $\{(x, y)\}$ is necessarily fibrant.

From now on, we introduce the following notation:

\begin{notation} \label{not_sesp_supscript_fiber}
    Let $X \in \textbf{sSpace}_k$. Let $x_1, \cdots, x_n \in (X_{0, \cdots, 0})_0$. Then write $(-)^{x_1, \cdots, x_n} : (\textbf{sSpace}_{k-1})_{/ (X_0)^n} \rightarrow \textbf{sSpace}_{k-1}$ be the functor sending $A \rightarrow (X_0)^n$ to the pullback
    $$
        A^{x_1, x_2, \cdots, x_n} := \{(x_1, \cdots, x_n)\} \times_{(X_0)^n} A.
    $$
\end{notation}

For instance, $X(a, b) = X_1^{a, b}$. We will also start writing $X(a_0, \cdots, a_n) := X_{Sp(n)}^{a_0, \cdots, a_n}$. Moreover, we have Segal maps on fibers
$$
\gamma_n^{a_0, \cdots, a_n} : X_n^{a_0, \cdots, a_n} \rightarrow X(a_0, \cdots, a_n).
$$
We would like a convenient term for when this is possible:

\begin{definition} \label{defn_sesp_object_fibered}
    Consider a cospan  $A \rightarrow (X_0)^m \leftarrow B$ in $\textbf{sSpace}_k$. A map $f : A \rightarrow B$ is \emph{object-fibered} with respect to this cospan if $f$ commutes with the cospan. If the cospan is evident, simply say $f$ is object-fibered.
\end{definition}

\begin{proposition} \label{prop_sesp_fiber_pres_trivfib}
    $(-)^{x_1, \cdots, x_n}$ preserves fibrations and trivial fibrations for the injective model structures on $(\textbf{sSpace}_{k-1})_{/ (X_0)^n}$ and $\textbf{sSpace}_{k-1}$.
\end{proposition}

\begin{proof}
    Pullbacks preserve fibrations and trivial fibrations in general.
\end{proof}

For instance, the maps $\gamma_n^{a_0, \cdots, a_n}$ are pullbacks of trivial fibrations of the form
\[\begin{tikzcd}
	{X_n^{a_0, \cdots, a_n}} & {X_n} \\
	{X(a_0, \cdots, a_n)} & {X_1 \times_{X_0} \cdots \times_{X_0} X_1}
	\arrow["{\gamma_n}", from=1-2, to=2-2]
	\arrow["{\gamma_n^{a_0, \cdots, a_n}}"', from=1-1, to=2-1]
	\arrow[from=1-1, to=1-2]
	\arrow[hook, from=2-1, to=2-2]
	\arrow["\lrcorner"{anchor=center, pos=0.125}, draw=none, from=1-1, to=2-2]
\end{tikzcd}\]

Another useful construction with regards to $\infty$-categories, and indeed one which underpins the core result of this paper, is the capacity to `collapse' such a higher category down to a mere category, with the same objects but with just path components of mapping spaces as its hom-sets. This is not at all difficult to achieve:

\begin{definition}[{\cite[Def. 1.9]{calaqueScheimbauerNoteInftyCategory2019}}] \label{defn_sesp_h_1}
    Let $X$ be a Segal space. The \emph{homotopy category} $h_1(X) \in \textbf{Cat}$ is the category whose objects are the elements of the set $(X_0)_0$, whose hom-sets are of the form
    $$
        \textbf{Hom}_{h_1(X)}(x, y) := \pi_0(X(x, y))
    $$
    with identities given by the degeneracies and composition by applying $\pi_0$ to the zig-zag diagram
    \[\begin{tikzcd}
    	{X(x, y) \times X(y, z)} & {X(x, y, z)} & {X(x,z)} & {X_2^{x,z}} & {X(x, z)}
    	\arrow["\cong", from=1-1, to=1-2]
    	\arrow["{X_{d^2_1}^{x,z}}", from=1-4, to=1-5]
    	\arrow[hook, from=1-2, to=1-3]
    	\arrow["{\gamma_2^{x,z}}"', from=1-4, to=1-3]
    \end{tikzcd}\]
\end{definition}

Because $\gamma_2^{x, z}$ is a weak equivalence of simplicial sets, the map $\pi_0(\gamma_2^{x, z})$ is a bijection. Hence, after applying $\pi_0$, we will have a single function from the leftmost object of the diagram to the rightmost.

Note that since the Segal maps $\gamma_2^{x, z}$ are trivial fibrations, we have a lifting problem of the form
\[\begin{tikzcd}
	&& {X_2^{x, z}} & {X(x, z)} \\
	{X(x, y) \times X(y, z)} & {(X_1 \times_{X_0} X_1)^{x, z}} & {(X_1 \times_{X_0} X_1)^{x, z}}
	\arrow[from=2-2, to=2-3]
	\arrow["{\gamma_2^{x, z}}"', from=1-3, to=2-3]
	\arrow["{\mu_2^{x, z}}", dashed, from=2-2, to=1-3]
	\arrow[from=1-3, to=1-4]
	\arrow[hook, from=2-1, to=2-2]
\end{tikzcd}\]
which admits the solution $\mu_2^{x, z}$. Taking $\pi_0$ of this chain of morphisms induces the same composition map as above. Rezk and Rasekh make a similar observation in  \cite{rezkModelHomotopyTheory2000} and \cite{rasekhIntroductionCompleteSegal2018} respectively, though instead note only that since $\gamma_2^{x, z}$ is a trivial fibration, any element of $(X_1 \times_{X_0} X_1)^{x, z}$ can be lifted to an element of $X_2^{x, z}$, with any two such liftings being in the same path component. All of these approaches yield the same category. We will make explicit use of the approach of solving the entire lifting problem when we come to discuss homotopy bicategories.

\begin{proposition}[{\cite[Prop. 5.4]{rezkModelHomotopyTheory2000}}] \label{prop_sesp_h1_is_cat}
    Let $X$ be a Segal space. Then $h_1(X)$ is a category.
\end{proposition}

One may also prove that any map $X \rightarrow Y$ between Segal spaces induces a functor $h_1(X) \rightarrow h_1(Y)$ in a functorial manner, due to commutativity with Segal maps and degeneracy maps. Hence, one obtains a functor
$$
    h_1 : \textbf{SeSp} \rightarrow \textbf{Cat}.
$$
It is perhaps useful to consider what happens to our running example of a Segal space. Here, $\Pi_1 : \textbf{Top} \rightarrow \textbf{Cat}$ is the \emph{fundamental groupoid} functor, sending a topological space to the groupoid whose objects are points in $X$ and whose morphisms are homotopy classes of paths relative to start and end points \cite[ch. 2]{may1999AConciseCourse}.

\begin{proposition} \label{prop_sesp_h_1_pi_1}
    There is a natural isomorphism
    $$
        h_1 \circ \textbf{Sing}_{\textbf{sS}} \cong \Pi_1.
    $$
\end{proposition}

\begin{proof}
    Consider a topological space $X$ and take the category $C := h_1(\textbf{Sing}_{\textbf{sS}}(X))$. The objects of this category are given by the underlying set of $\textbf{Sing}_{\textbf{sS}}(X)_0$, which is just the underlying set of $X$. Hence, there is a clear bijection from the objects of $C$ to those of $\Pi_1(X)$.

    Now, let $x, y \in X$. We wish to find a natural bijection
    $$
        \textbf{Hom}_C(x, y) := \pi_0(\{(x, y)\} \times_{X^2} X^{\Delta_t[1]}) \rightarrow \textbf{Hom}_{\Pi_1(X)}(x, y).
    $$
    Note that two paths $\Delta_t[1] \cong [0, 1] \rightarrow X$ from $x$ to $y$ are identified in $\Pi_1(X)$ if and only if there is a path $[0, 1] \rightarrow X^{\Delta_t[1]}$ between these paths that is constant on source and target. This happens if and only if they are in the same path component of $\pi_0(\{(x,y)\} \times_{X^2} X^{\Delta_t[1]})$, so there is a natural bijection between these sets sending $[\gamma]$ to $[\gamma]$ for any path $\gamma$ from $x$ to $y$ in $X$.

    That these bijections induce a functor is evident. Naturality is also clear.
\end{proof}

One may construct a model structure for Segal spaces, though we will not explore how this is done in depth here:

\begin{theorem}[{\cite[Theorem 5.2]{bergnerRezkComparisonModelsInfty2020}}] \label{thm_sesp_modstr_sesp}
    There is a model structure obtained as a left Bousfield localization of the Reedy model structure on $\textbf{sSpace}$, which we will write as $SeSp$, whose fibrant objects are the Segal spaces.
\end{theorem}

In short, one obtains this structure by localizing with respect to the spine inclusions $F_1^0(g_n)$ for all $n \geq 2$.

If we sought a fully correct model of $(\infty, 1)$-category, however, we would find that we are not quite done. One final addition is necessary to ensure that the `homotopical' data of $\infty$-groupoids at each level of a Segal space and the `categorical' data, displayed for instance in the homotopy category, coincide neatly. This is the notion of a \emph{complete Segal space} {\cite[pg. 14]{rezkModelHomotopyTheory2000}}, which we will not need and thus omit from this paper.

\subsection{2-fold Segal Spaces}

We are now ready to begin our foray into the world of $(\infty, 2)$-categories. Our model of choice will be \emph{(complete) $2$-fold Segal spaces}, which will be built out of (complete) Segal spaces.

In order to facilitate the structure of an $(\infty, 2)$-category $X$, we should consider both horizontal and vertical composition. Hence, for each $n, m \geq 0$, we will now have an $\infty$-groupoid of `grids' of $2$-morphisms of horizontal length $n$ and vertical length $m$, which we will write as $X_{n, m}$. We might see this as the $\infty$-groupoid of $\infty$-functors $[n] \times [m] \rightarrow X$.

This should, for a fixed $n \geq 0$, induce a (complete) Segal space $X_{n, \bullet}$ where composition is vertical. In the special case $n = 1$, we have what we might call the $(\infty, 1)$-category of $1$-morphisms $X_{1, \bullet}$. We then have $(\infty, 1)$-categories $X_{\bullet, n}$ whose composition is instead horizontal.

A starting point to formalize this intuition is as follows:

\begin{definition}[{\cite[pg. 15]{bergnerSurveyModelsInfty2018}}] \label{defn_sesp2_double_sesp}
    A \emph{double Segal space} is a Reedy fibrant bisimplicial space such that the Segal maps,
    \begin{align*}
        \gamma_{n, \bullet}: X_{n, \bullet} &\rightarrow X_{1, \bullet} \times_{X_{0, \bullet}} \cdots \times_{X_{0, \bullet}} X_{1, \bullet}, \\
        \gamma_{\bullet, n} : X_{\bullet, n} &\rightarrow X_{\bullet, 1} \times_{X_{\bullet, 0}} \cdots \times_{X_{\bullet, 0}} X_{\bullet, 1}, 
    \end{align*}
    are weak equivalences for all $n \geq 2$, so $X_{k, \bullet}$ and $X_{\bullet, k}$ are Segal spaces.
\end{definition}

This is in fact quite analogous to how we previously built Segal spaces from $\infty$-groupoids. At each level, a double Segal space $X$ will be a Segal space. We could indeed choose to see $X_0$ as the `underlying $(\infty, 1)$-category' of $X$, only containing invertible $2$-morphisms\footnote{This intuition is only truly precise in the case of \emph{completeness} of $X$; without this property, $X_0$ may not contain all of the `invertible $2$-morphisms' in $X$. This flexibility is necessary to construct many models of higher categories of cobordisms relevant to TQFTs, such as in \cite{lurieClassificationTopologicalField2009}, so we will not discuss completeness further.}. More generally, $X_n$ will be a Segal space, which we could choose to interpret as the $(\infty, 1)$-category of $(\infty, 2)$-functors $[n] \rightarrow X$. Precomposition yields a functor from $\Delta^{op}$ to $\textbf{SeSp}$, which is a Reedy fibrant bisimplicial space satisfying a Segal condition valued in $(\infty, 1)$-categories instead of $\infty$-groupoids. This can be made more precise:

\begin{proposition} \label{prop_sesp2_doubl_interp}
    A bisimplicial space $X$ is a double Segal space if and only if it is Reedy fibrant, levelwise a Segal space and is such that the Segal maps for all $n \geq 2$
    $$
        \gamma_n : X_n \rightarrow X_1 \times_{X_0} \cdots \times_{X_0} X_1
    $$
    are weak equivalences in $SeSp$.
\end{proposition}

\begin{proof}
    The Segal map $\gamma_n$ is a Reedy fibration. Thus, since Reedy trivial fibrations and trivial fibrations in $SeSp$ are the same \cite[Prop. 3.3.3]{hirschhornModelCategoriesTheir2009}, the maps $\gamma_n$ are weak equivalences in $SeSp$ if and only if they are weak equivalences levelwise. This means that for each $m \geq 0$, the map
    $$
        \gamma_{n, m} : X_{n, m} \rightarrow X_{1, m} \times_{X_{0, m}} \cdots \times_{X_{0, m}} X_{1, m}
    $$
    is a weak equivalence, which is precisely the first case of Segal maps for a double Segal space. The second case is given if and only if each $X_n$ is a Segal space.
\end{proof}

We had to separately assert that $X$ was levelwise a Segal space, since Reedy fibrancy alone will not guarantee it. One may note that this is precisely the notion of $2$-fold Segal space in \cite[Def. 3.4]{bergnerModelsInftyCategories2012} minus `essential constancy', a property we will return to later. This result is by no means original; there is a similar statement to be found for instance in \cite[Prop. 6.3]{bergnerSurveyModelsInfty2018}. It is written here only to secure intuitions, so will not study it further in this paper.

A useful construction with double Segal spaces will be an analogue of the mapping spaces we saw with Segal spaces:

\begin{definition}[{\cite[pg. 16]{bergnerSurveyModelsInfty2018}}] \label{defn_sesp2_map_sp}
    Let $X$ be a double Segal space. Let $x, y \in X_{0, 0, 0}$. Then the \emph{mapping space} $X(x, y)$ is defined to be the pullback in $\textbf{sSpace}$ of the form
    \[\begin{tikzcd}
    	{X(x, y)} & {X_1} \\
    	{\{(x, y)\}} & {X_0 \times X_0.}
    	\arrow[from=1-2, to=2-2]
    	\arrow[from=1-1, to=1-2]
    	\arrow[from=1-1, to=2-1]
    	\arrow[from=2-1, to=2-2]
    	\arrow["\lrcorner"{anchor=center, pos=0.125}, draw=none, from=1-1, to=2-2]
    \end{tikzcd}\]
\end{definition}

As noted in \cite[pg. 15]{bergnerSurveyModelsInfty2018}, double Segal spaces do not quite model $(\infty, 2)$-category theory as we may wish them to. To see the issue, as our prior intuitions have demonstrated, we should consider $X_{1, 1}$ to be the $\infty$-groupoid of $1$-morphisms with $2$-morphisms between them. We might consider the source and target maps $X_{1, 1} \rightarrow X_{1, 0}$ to give the source and target $1$-morphisms. However, there are now four possible source and target objects, given by the four maps $X_{1, 1} \rightarrow X_{0, 0}$. These do not have to equate to each other to give a `globular picture' of $2$-morphisms:
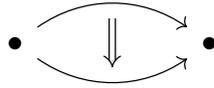
\[\begin{tikzcd}
	\bullet && \bullet
	\arrow[""{name=0, anchor=center, inner sep=0}, curve={height=-18pt}, from=1-1, to=1-3]
	\arrow[""{name=1, anchor=center, inner sep=0}, curve={height=18pt}, from=1-1, to=1-3]
	\arrow[shorten <=5pt, shorten >=5pt, Rightarrow, from=0, to=1]
\end{tikzcd}\]
Instead, we have to contend with the `vertical $1$-morphisms' in $X_{0, 1}$ that we have neglected thus far. We instead have a `cubical' picture of $2$-morphisms:
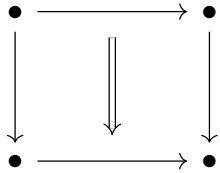
\[\begin{tikzcd}
	\bullet && \bullet \\
	\\
	\bullet && \bullet
	\arrow[""{name=0, anchor=center, inner sep=0}, from=1-1, to=1-3]
	\arrow[from=1-1, to=3-1]
	\arrow[""{name=1, anchor=center, inner sep=0}, from=3-1, to=3-3]
	\arrow[from=1-3, to=3-3]
	\arrow[shorten <=9pt, shorten >=9pt, Rightarrow, from=0, to=1]
\end{tikzcd}\]
More explicitly, given some $f \in (X_{1, 1})_0$, we could identify the parts of this diagram as
\[\begin{tikzcd}
	{X_{(d^1_1, d^1_1)}(f) \in X_{0,0}} && {X_{(d^1_1, d^1_0)}(f) \in X_{0,0}} \\
	\\
	{X_{(d^1_0, d^1_1)}(f) \in X_{0,0}} && {X_{(d^1_0, d^1_0)}(f) \in X_{0,0}}
	\arrow[""{name=0, anchor=center, inner sep=0}, "{X_{(d^1_1, id)}(f) \in X_{1,0}}", from=1-1, to=1-3]
	\arrow["{X_{(id, d^1_1)}(f) \in X_{0,1}}"', from=1-1, to=3-1]
	\arrow[""{name=1, anchor=center, inner sep=0}, "{X_{(d^1_1, id)}(f) \in X_{1,0}}"', from=3-1, to=3-3]
	\arrow["{X_{(id, d^1_0)}(f) \in X_{0,1}}", from=1-3, to=3-3]
	\arrow["{f \in X_{1,1}}", shorten <=9pt, shorten >=9pt, Rightarrow, from=0, to=1]
\end{tikzcd}\]
We take a greater interest in the globular picture here. Hence, the vertical spaces $X_{0, \bullet}$ will have to be brushed under the rug somehow, which we will accomplish by demanding that they are \emph{essentially constant}:

\begin{definition}[{\cite[pg. 12]{calaqueScheimbauerNoteInftyCategory2019}}] \label{defn_sesp2_ess_const}
    A Segal space $X$ is \emph{essentially constant} if and only if the natural map $q : \iota_1^0(X_0) \rightarrow X$, defined levelwise such that $q_n : X_0 \rightarrow X_n$ is induced by the degeneracy map $\Delta[n] \rightarrow \Delta[0]$, is a weak equivalence in $SeSp$.
\end{definition}

It will be beneficial for us to establish a running example of an $(\infty, 2)$-category. This will be done in much the same way as with complete Segal spaces:

\begin{definition} \label{defn_sesp2_sing_sss}
    Suppose $X \in \textbf{Top}$. Then define $\textbf{Sing}_{\textbf{ssS}}(X)$ to be the bisimplicial space such that, for all $n \geq 0$,
    $$
        \textbf{Sing}_{\textbf{ssS}}(X)_n := \textbf{Sing}_{\textbf{sS}}(X^{\Delta_t[n]}).
    $$
\end{definition}

Note then that
$$
    (\textbf{Sing}_{\textbf{ssS}}(X)_{a,b})_c \cong \textbf{Hom}_{\textbf{Top}}(\Delta_t[a] \times \Delta_t[b] \times \Delta_t[c], X).
$$
It is evident that this is a double Segal space. Indeed, we have that $\textbf{Sing}_{\textbf{ssS}}(X)_{\bullet, k} \cong \textbf{Sing}_{\textbf{ssS}}(X)_{k, \bullet} = \textbf{Sing}_{\textbf{sS}}(X^{\Delta_t[k]})$.

As for our intuitions, the idea that $\textbf{Sing}_{\textbf{ssS}}(X)_0$ should be seen as the underlying $(\infty, 1)$-category of $\textbf{Sing}_{\textbf{ssS}}(X)$ is somewhat immediate by definition. Moreover, the interpretation of $\textbf{Sing}_{\textbf{ssS}}(X)_1$ as the $(\infty, 1)$-category of morphisms also applies. The intuition carries on for higher levels.

Now, consider the cubical picture for $2$-morphisms from before. Indeed, the space $\textbf{Sing}_{\textbf{ssS}}(X)_{1, 1}$ is quite literally $\textbf{Sing}(X^{\Delta_t[1] \times \Delta_t[1]})$, the space of squares inside $X$. The face maps identify the edges and corners of this square. With the Segal maps, we can compose these squares either horizontally or vertically by aligning edges.

Note also however that the vertical maps are `essentially constant'; they are paths, which can be contracted so the square is essentially `pinched' into a globular picture, up to homotopy. Hence, the vertical data is rather negligible. Of course, in this example the horizontal data is also as such since $\textbf{Sing}_{\textbf{ssS}}(X)$ is really an $\infty$-groupoid in disguise, but more general examples of $(\infty, 2)$-categories are not so simple.

We henceforth employ the notation $A^{a_1, \cdots, a_n}$ and $X(a_1, \cdots, a_n)$ similarly to for simplicial spaces. Again, note that the fibers of the Segal maps $\gamma_n^{a_0, \cdots, a_n}$ are trivial fibrations in $\textbf{sSpace}^{inj}$ if $X$ is a double Segal space.
\subsection{Composition and Composite Diagrams}

Before we conclude on $2$-fold Segal spaces, it will serve us well to think further on how composition in such an $(\infty, 2)$-category really works. Our intuitions for $2$-fold Segal spaces and discussion of $h_1$ show that, for a $2$-fold Segal space $X$, constructing a binary composition operation amounts to solving the lifting problem
\[\begin{tikzcd}
	& {X_2} & {X_1} \\
	{X_1 \times_{X_0} X_1} & {X_1 \times_{X_0} X_1}
	\arrow["id"', from=2-1, to=2-2]
	\arrow["{\gamma_2}", from=1-2, to=2-2]
	\arrow["{\mu_2}", dashed, from=2-1, to=1-2]
	\arrow["{X_{\langle 0, 2 \rangle}}", from=1-2, to=1-3]
\end{tikzcd}\]
to obtain a map with respect to some $x, y, z \in (X_{0, 0})_0$
$$
    \circ^{x, y, z} : X(x, y) \times X(y, z) \hookrightarrow (X_1 \times_{X_0} X_1 \times_{x_0} X_1)^{x, z} \xrightarrow{\mu_2^{x, z}} X_2^{x, z} \xrightarrow{X_{\langle 0, 2 \rangle}^{x, z}} X(x, z).
$$
We should think of this as binary composition on the hom-spaces $X(x, y)$ and $X(y, z)$. In general, we can take any sequence of objects $x_0, \cdots, x_n \in (X_{0, 0})_0$ and solve a lifting problem
\[\begin{tikzcd}
	& {X_n} & {X_1} \\
	{X_1 \times_{X_0} \cdots \times_{X_0} X_1} & {X_1 \times_{X_0} \cdots \times_{X_0} X_1}
	\arrow["id"', from=2-1, to=2-2]
	\arrow["{\gamma_n}", from=1-2, to=2-2]
	\arrow["{\mu_n}", dashed, from=2-1, to=1-2]
	\arrow["{X_{\langle 0, n \rangle}}", from=1-2, to=1-3]
\end{tikzcd}\]
to obtain a composition map
$$
    \circ^{x_0, \cdots, x_n} : \prod_{i = 1}^n X(x_{i-1}, x_i) \hookrightarrow (X_1 \times_{X_0} \cdots \times_{X_0} X_1)^{x_0, x_n} \xrightarrow{\mu_n^{x_0, x_n}} X_n^{x_0, x_n} \xrightarrow{X_{\langle 0, n \rangle}^{x_0, x_n}} X(x_0, x_n).
$$
Note that $\mu_n$ must be object-fibered with respect to the cospan
$$
    X_n \xrightarrow{X_{\langle 0 \rangle} \times \cdots \times X_{\langle n \rangle}} (X_0)^{n+1} \xleftarrow{(X_{\langle 0 \rangle} \times X_{\langle 1 \rangle}) \times_{1_{X_0}} \cdots \times_{1_{X_0}} (X_{\langle 0 \rangle} \times X_{\langle 1 \rangle})} X_1 \times_{X_0} \cdots \times_{X_0} X_1
$$
as defined in Definition~\ref{defn_sesp_object_fibered}.

\begin{proposition} \label{prop_cd_objfibered_liftprob}
    Let $X$ be a $2$-fold Segal space. Consider a solution $g$ to a lifting problem in $\textbf{sSpace}^{inj}$
    \[\begin{tikzcd}
    	A & B \\
    	C & D
    	\arrow["p", from=1-2, to=2-2]
    	\arrow["t"', from=2-1, to=2-2]
    	\arrow["g"{description}, dashed, from=2-1, to=1-2]
    	\arrow["i"', from=1-1, to=2-1]
    	\arrow["s", from=1-1, to=1-2]
    \end{tikzcd}\]
    where every morphism besides $g$ is object-fibered over some maps to $(X_0)^{n+1}$. Then $g$ is object-fibered and for any $x_0, \cdots, x_n \in (X_{0, 0})_0$, the diagram\footnote{Note that this may no longer be a lifting problem, as $i^{x_0, \cdots, x_n}$ may not be a (trivial) cofibration. For us, usually $A = \emptyset$ and all objects are cofibrant when we want a lifting problem, so this doesn't matter.}
    \[\begin{tikzcd}
    	{A^{x_0, \cdots, x_n}} && {B^{x_0, \cdots, x_n}} \\
    	{C^{x_0, \cdots, x_n}} && {D^{x_0, \cdots, x_n}}
    	\arrow["{p^{x_0, \cdots, x_n}}", from=1-3, to=2-3]
    	\arrow["{t^{x_0, \cdots, x_n}}"', from=2-1, to=2-3]
    	\arrow["{g^{x_0, \cdots, x_n}}"{description}, dashed, from=2-1, to=1-3]
    	\arrow["{i^{x_0, \cdots, x_n}}"', from=1-1, to=2-1]
    	\arrow["{s^{x_0, \cdots, x_n}}", from=1-1, to=1-3]
    \end{tikzcd}\]
    commutes.
\end{proposition}

\begin{proof}
    Consider the chosen maps $f_Z : Z \rightarrow (X_0)^{n+1}$ for $Z \in \{A, B, C, D\}$. Then
    $$
        f_C = f_D \circ t = f_D \circ (p \circ g) = (f_D \circ p) \circ g = f_B \circ g
    $$
    and so $g$ is object-fibered. Commutativity then follows from functoriality of $(-)^{x_0, \cdots, x_n}$.
\end{proof}

The map $\circ^{x_0, \cdots, x_n}$ is then `unbiased composition,' taking a sequence of $1$-morphisms in $X$
$$
    x_0 \xrightarrow{f_1} x_1 \xrightarrow{f_2} \cdots \xrightarrow{f_{n-1}} x_{n-1} \xrightarrow{f_n} x_n
$$
and composing them into a single map $f_n \circ \cdots \circ f_0$ at once. These operations amount to horizontal composition in our $(\infty, 2)$-category, while the Segal maps in each Segal space $X_n$ induce vertical composition. We package this into a definition:

\begin{definition} \label{defn_cd_hor_comps}
    Let $X$ be a $2$-fold Segal space. Then a \emph{choice of horizontal compositions} is a sequence of maps $\mu_n : X_1 \times_{X_0} \cdots \times_{X_0} X_1 \rightarrow X_n$ for $n \geq 2$ solving the lifting problems
    \[\begin{tikzcd}
    	& {X_n} \\
    	{X_1 \times_{X_0} \cdots \times_{X_0} X_1} & {X_1 \times_{X_0} \cdots \times_{X_0} X_1}
    	\arrow["id"', from=2-1, to=2-2]
    	\arrow["{\gamma_n}", from=1-2, to=2-2]
    	\arrow["{\mu_n}", dashed, from=2-1, to=1-2]
    \end{tikzcd}\]
\end{definition}

\begin{notation} \label{not_cd_mu1_mu0}
    Given a choice of horizontal composition maps $\mu_n$, write $\mu_1 = 1_{X_1}$ and $\mu_0 = 1_{X_0}$.
\end{notation}

Recall that since $\gamma_1$ and $\gamma_0$ are just identities, the maps $\mu_1$ and $\mu_0$ are solutions of similar lifting problems. In fact, they are the only solutions of their respective lifting problems.

\begin{notation} \label{not_cd_circ_mu_n}
    For a $2$-fold Segal space $X$, given $n > 0$ and $x_0, \cdots, x_n \in (X_{0, 0})_0$, along with a choice of horizontal compositions $(\mu_n)_{n \geq 0}$, write $\circ^{x_0, \cdots, x_n} : \prod_{i = 1}^n X(x_{i-1}, x_i) \rightarrow X(x_0, x_n)$ for the map described above. If $n = 0$, write $\circ^{x_0} = X^{x_0, x_0}_{s^1_0} : \ast \cong \{x_0\} \hookrightarrow X(x_0, x_0)$.
\end{notation}

Note we did not need to separately specify $\circ^{x_0}$, as if $n = 0$ then $\prod_{i = 1}^n X(x_{i-1}, x_i) \cong \ast$. We do so simply for clarity; in future, the $n = 0$ case will be implicitly covered in this manner.

Note also that there is almost never a unique choice of horizontal compositions - there are a multitude of possible solutions to the lifting problems in general. We will see in time that these choices, while not equal, are in fact always equivalent in a way we will make precise later.

One might wonder how more complex `nested' composition operations look, for instance composing a chain of $1$-morphisms in $X$ of the form $v \xrightarrow{f} w \xrightarrow{g} x \xrightarrow{h} y \xrightarrow{k} z$ into the composite $(k \circ h) \circ (g \circ f)$. This is of course merely a matter of applying our operations $\circ^{x_0, \cdots, x_n}$ in sequence, yet a more direct interpretation is also possible, one which will be of immense importance in developing our homotopy bicategories.

Note that everything we have written can be rephrased in terms of mapping spaces:
\[\begin{tikzcd}
	& {X_n} & {X_1} \\
	{X_{Sp(n)}} & {X_{Sp(n)}}
	\arrow["id"', from=2-1, to=2-2]
	\arrow["{\gamma_n}", from=1-2, to=2-2]
	\arrow["{\mu_n}", dashed, from=2-1, to=1-2]
	\arrow["{X_{\langle 0, n \rangle}}", from=1-2, to=1-3]
\end{tikzcd}\]
The composition operation $(- \circ -) \circ (- \circ -)$ can now be constructed from the solutions to the sequence of two lifting problems
\[\begin{tikzcd}
	&& {X_2} & {X_1} \\
	&& {X_{Sp(2)}} \\
	&& {X_{\Delta[2] \sqcup_{\Delta[0]} \Delta[2]}} \\
	{X_{Sp(4)}} && {X_{Sp(4)}}
	\arrow["id"', from=4-1, to=4-3]
	\arrow["{\gamma_2 \times_{1_{X_0}} \gamma_2}", from=3-3, to=4-3]
	\arrow["{\mu_2 \times_{1_{X_0}} \mu_2}"{description}, dashed, from=4-1, to=3-3]
	\arrow["{X_{\langle 0, 2 \rangle} \times_{1_{X_0}} X_{\langle 0, 2 \rangle}}"', from=3-3, to=2-3]
	\arrow["{\gamma_2}", from=1-3, to=2-3]
	\arrow["{X_{\langle 0, 2 \rangle}}", from=1-3, to=1-4]
	\arrow["{\mu_2 \circ (X_{\langle 0, 2 \rangle} \times_{1_{X_0}} X_{\langle 0, 2 \rangle}) \circ (\mu_2 \times_{1_{X_0}} \mu_2)}", curve={height=-12pt}, dashed, from=4-1, to=1-3]
\end{tikzcd}\]
The overall map is then, for $v, w, x, y, z \in (X_{0, 0})_0$, the path
\[\begin{tikzcd}
	{X(v, w) \times X(w, x) \times X(x, y) \times X(y, z)} & {X_{Sp(4)}^{v, z}} \\
	& {X(v, z)}
	\arrow[hook, from=1-1, to=1-2]
	\arrow["{X_{\langle 0, 2 \rangle}^{v, z} \circ (\mu_2 \circ (X_{\langle 0, 2 \rangle} \times_{1_{X_0}} X_{\langle 0, 2 \rangle}) \circ (\mu_2 \times_{1_{X_0}} \mu_2))^{v, z}}"', from=1-2, to=2-2]
\end{tikzcd}\]
We can reinterpret this algebraic monstrosity in a more geometric light. Consider how the cospan
$$
    X_2 \xrightarrow{\gamma_2} X_{Sp(2)} \leftarrow X_{\Delta[2] \sqcup_{\Delta[0]} \Delta[2]}
$$
appears in the above lifting problem. The two lifts can be combined into a single lift to the pullback of the cospan
$$
    X_{\Delta[2] \sqcup_{\Delta[0]} \Delta[2]} \times_{X_{Sp(2)}} X_2
$$
which, since $F_2^0$ preserves colimits, we can rephrase as
$$
    X_{(\Delta[2] \sqcup_{\Delta[0]} \Delta[2]) \sqcup_{Sp(2)} \Delta[2]}
$$
to obtain a lifting problem
\[\begin{tikzcd}
	& {X_{(\Delta[2] \sqcup_{\Delta[0]} \Delta[2]) \sqcup_{Sp(2)} \Delta[2]}} & {X_1} \\
	{X_{Sp(4)}} & {X_{Sp(4)}}
	\arrow["id"', from=2-1, to=2-2]
	\arrow["{\iota^*}", from=1-2, to=2-2]
	\arrow["\mu"{description}, dashed, from=2-1, to=1-2]
	\arrow["{\tau^*}"', from=1-2, to=1-3]
\end{tikzcd}\]
where $\iota^*$ and $\tau^*$ are induced by the maps $\iota$ and $\tau$ in the cospan in $\textbf{sSet}$
$$
    Sp(4) \xhookrightarrow{\iota} (\Delta[2] \sqcup_{\Delta[0]} \Delta[2]) \sqcup_{Sp(2)} \Delta[2] \xhookleftarrow{\tau} \Delta[1].
$$
The domain of the map $\mu$ is well-understood to us by this point: it is simply the $\infty$-category of chains of length four in $X$. The codomain, however, is new. We should interpret it as the $\infty$-category of diagrams in $X$ of the form
\[\begin{tikzcd}
	& w && y \\
	&& x \\
	v &&&& z
	\arrow["f", from=3-1, to=1-2]
	\arrow["g", from=1-2, to=2-3]
	\arrow["h", from=2-3, to=1-4]
	\arrow["k", from=1-4, to=3-5]
	\arrow[""{name=0, anchor=center, inner sep=0}, "q"{description}, from=3-1, to=2-3]
	\arrow[""{name=1, anchor=center, inner sep=0}, "r"{description}, from=2-3, to=3-5]
	\arrow[""{name=2, anchor=center, inner sep=0}, "p"', from=3-1, to=3-5]
	\arrow["\alpha", shorten >=5pt, Rightarrow, from=1-2, to=0]
	\arrow["\beta"', shorten >=5pt, Rightarrow, from=1-4, to=1]
	\arrow["\gamma"', shorten >=3pt, Rightarrow, from=2-3, to=2]
\end{tikzcd}\]
where $\alpha, \beta, \gamma \in (X_{2, 0})_0$ are compositions, $f, g, h, k, q, r, p \in (X_{1, 0})_0$ are morphisms and $v, w, x, y, z \in (X_{0, 0})_0$ are objects in our $(\infty, 2)$-category.

The morphism $\mu$ can now be seen as taking a chain $v \xrightarrow{f} w \xrightarrow{g} x \xrightarrow{h} y \xrightarrow{k} z$ and extending it to a diagram as above, choosing $\alpha, \beta, \gamma, q, r$ and $p$. The shape of this diagram can in fact tell us directly what the composition operation in question is, as $p$ is a composite of $q$ with $r$, which are in turn composites of $f$ with $g$ and $h$ with $k$ respectively.

The general shape of such `composition diagrams' is evident by an induction. In the below definition, write
$$
    s, t: \Delta[0] \rightarrow Sp(n)
$$
for the maps in $\textbf{sSet}$ identifying the first and last objects.

\begin{definition} \label{defn_cd_simp_comp_diag}
    A \emph{simplicial composition diagram} is a cospan $Sp(n) \xrightarrow{\iota} K \xleftarrow{\tau} \Delta[1]$ defined inductively such that either:
    \begin{enumerate}
        \item $K = \Delta[n]$, $\iota = g_n : Sp(n) \hookrightarrow \Delta[n]$ and $\tau = \langle 0, n \rangle : \Delta[1] \hookrightarrow \Delta[n]$ for some $n \geq 0$;
        \item $K = (K_1 \sqcup_{\Delta[0]} \cdots \sqcup_{\Delta[0]} K_d) \sqcup_{Sp(d)} \Delta[d]$ for some $d > 1$, where $Sp(k_i) \xrightarrow{\iota_i} K_i \xleftarrow{\tau_i} \Delta[1]$ are simplicial composition diagrams for $1 \leq i \leq d$ and the pushouts are taken over the diagrams
        $$
            K_1 \xleftarrow{\iota_1 \circ t} \Delta[0] \xrightarrow{\iota_2 \circ s} \cdots \xleftarrow{\iota_{d-1} \circ t} \Delta[0] \xrightarrow{\iota_d \circ s} K_d
        $$
        and
        $$
            (K_1 \sqcup_{\Delta[0]} \cdots \sqcup_{\Delta[0]} K_d) \xleftarrow{\tau_1 \sqcup_{1_{\Delta[0]}} \cdots \sqcup_{1_{\Delta[0]}} \tau_d} Sp(d) \hookrightarrow \Delta[d].
        $$
        The maps are then $\iota$ as in the diagram
        \[\begin{tikzcd}
        	{Sp(n)} & {Sp(k_1 + \cdots + k_d)} & {Sp(k_1) \sqcup_{\Delta[0]} \cdots \sqcup_{\Delta[0]} Sp(k_d)} \\
        	&& {K_1 \sqcup_{\Delta[0]} \cdots \sqcup_{\Delta[0]} K_d} \\
        	&& K
        	\arrow["\cong"{marking, allow upside down}, draw=none, from=1-2, to=1-3]
        	\arrow["{\iota_1 \sqcup_{1_{\Delta[0]}} \cdots \sqcup_{1_{\Delta[0]}} \iota_d}", from=1-3, to=2-3]
        	\arrow[hook', from=2-3, to=3-3]
        	\arrow["\iota"', from=1-2, to=3-3]
        	\arrow["{=}"{description}, draw=none, from=1-1, to=1-2]
        \end{tikzcd}\]
        and $\tau : \Delta[1] \xrightarrow{\langle 0, d \rangle} \Delta[d] \hookrightarrow K$.
    \end{enumerate}
    We will say the \emph{arity} of the composition diagram $K$ is $n$. Its \emph{depth} is either $1$ if $K = \Delta[n]$ or $1 + \max_{1 \leq i \leq d} d_i$ otherwise, where $d_i$ is the depth of $K_i$.
\end{definition}

Note the inclusion of $\Delta[0] \xrightarrow{id} \Delta[0] \xleftarrow{X_{\langle 0, 0 \rangle}} \Delta[1]$, which amounts to the problem of taking identities. Note also that it is entirely unambiguous to refer to a simplicial composition diagram solely by the space $K$; there is only one possibility for the maps $\iota$ and $\tau$.

A similar notion to these diagrams is to be found in the membranes obtained from polygonal decompositions of convex $(n+1)$-gons used to define $2$-Segal spaces in \cite{dyckerhoffKapranovHigherSegalSpaces2019}. They also strongly resemble the opetopes of \cite{baezDolanHigherDimensionalAlgebraIII1998}. Making these connections precise and exploring their full influence is left to future work.

The space of composition diagrams in question will be $X_K$ for a $2$-fold Segal space $X$ and a simplicial composition diagram $K$. It is the case that these are always Segal spaces:

\begin{proposition} \label{prop_cd_sp_of_cd_cssp}
    Let $X$ be a $2$-fold Segal space. Suppose $K$ is a simplicial composition diagram. Then $X_K$ is a Segal space.
\end{proposition}

\begin{proof}
    We proceed by induction on the depth of $K$. The result is clear if $K = \Delta[n]$. Now, suppose the depth is greater than $1$, so $K = (K_1 \sqcup_{\Delta[0]} \cdots \sqcup_{\Delta[0]} K_n) \sqcup_{Sp(n)} \Delta[n]$. This implies that $X_K$ is isomorphic to the pullback
    $$
         (X_{K_1} \times_{X_0} \cdots \times_{X_0} X_{K_n}) \times_{X_{Sp(n)}} X_n.
    $$
    The outermost pullback is in fact along a Segal map. Hence, the induced map by the pullback 
    $$
        X_K \rightarrow X_{K_1} \times_{X_0} \cdots \times_{X_0} X_{K_n}
    $$
    is a trivial Reedy fibration, meaning it is one in $SeSp$ as any left Bousfield localization has the same trivial fibrations as the original model structure \cite[Prop. 3.3.3]{hirschhornModelCategoriesTheir2009}. By induction, the codomain is fibrant in this model structure, so the space $X_K$ is also fibrant as desired.
\end{proof}

Moreover, the map $\iota$ always induces a trivial fibration:

\begin{proposition} \label{prop_cd_iota_ast_trivfib}
    Let $X$ be a $2$-fold Segal space. Suppose $Sp(n) \xrightarrow{\iota} K \xleftarrow{\tau} \Delta[1]$ is a simplicial composition diagram. Then the map $\iota^* : X_K \rightarrow X_{Sp(n)}$ induced by $\iota$ is a trivial Reedy fibration.
\end{proposition}

\begin{proof}
    By induction on the definition of $K$, the map $\iota^*$ is a pullback of Segal maps and identities and is thus a trivial fibration by the Segal condition on $X$.
\end{proof}

Note also that the maps $\iota^* : X_K \rightarrow X_{Sp(n)}$ are tautologically object-fibered with respect to the maps induced by $\Delta[0] \hookrightarrow Sp(n) \xrightarrow{\iota} K$. The same can be said for $\tau^* : X_K \rightarrow X_1$ with respect to the maps $\Delta[0] \hookrightarrow \Delta[1] \xrightarrow{\tau} K$. These two maps also agree with the two endpoint inclusions of the spine composed with $\iota$.

Simplicial composition diagrams can be used to nest existing composition operations in an evident way, a natural extension of our example:

\begin{definition} \label{defn_cd_induced_comp_k}
    Let $X$ be a $2$-fold Segal space with a choice of horizontal compositions $(\mu_n)_{n \geq 0}$. Let $Sp(n) \xrightarrow{\iota} K \xleftarrow{\tau} \Delta[1]$ be a simplicial composition diagram. Then the \emph{induced composition on $K$} is the map $\mu_K$ solving the lifting problem
    \[\begin{tikzcd}
    	& {X_K} & {X_1} \\
    	{X_{Sp(n)}} & {X_{Sp(n)}}
    	\arrow["id"', from=2-1, to=2-2]
    	\arrow["{\iota^*}", from=1-2, to=2-2]
    	\arrow["{\mu_K}"{description}, dashed, from=2-1, to=1-2]
    	\arrow["{\tau^*}", from=1-2, to=1-3]
    \end{tikzcd}\]
    where $\iota^*$ and $\tau^*$ are induced by $\iota$ and $\tau$ respectively, such that:
    \begin{enumerate}
        \item If $K = \Delta[n]$ then $\mu_K = \mu_n$;
        \item If $K = (K_1 \sqcup_{\Delta[0]} \cdots \sqcup_{\Delta[0]} K_r) \sqcup_{Sp(r)} \Delta[r]$ where each $K_i$ has arity $k_i$, then $\mu_K$ is the pullback map induced by $(\mu_{K_1} \times_{1_{X_0}} \cdots \times_{1_{X_0}} \mu_{K_r})$ and
        $$
            \mu_r \circ (\tau_1^* \times_{1_{X_0}} \cdots \times_{1_{X_0}} \tau_r^*) \circ (\mu_{K_1} \times_{1_{X_0}} \cdots \times_{1_{X_0}} \mu_{K_r})
        $$
        where $\tau_i^*$ is induced by the map $\tau_i : \Delta[1] \rightarrow K_i$ for $1 \leq i \leq r$.
    \end{enumerate}
\end{definition}

It is not hard to see how $\mu_K$, by induction, defines a nesting of composition operations described in the structure of $K$. In particular, if $K = (K_1 \sqcup_{\Delta[0]} \cdots \sqcup_{\Delta[0]} K_r) \sqcup_{Sp(r)} \Delta[r]$, then the composition operation in question composes the first $k_1$ morphisms by $\mu_{K_1}$, the second $k_2$ morphisms by $\mu_{K_2}$ and so on up to $\mu_{K_r}$, then composes the resulting chain of $r$ morphisms by $\mu_r$. This can be generalized:

\begin{proposition} \label{prop_cd_decomp_scd}
    Suppose $K = (K_1 \sqcup_{\Delta[0]} \cdots \sqcup_{\Delta[0]} K_r) \sqcup_{Sp(r)} K_0$ is a simplicial composition diagram of arity $n$ for diagrams $\Delta[k_i] \xrightarrow{\iota_i} K_i \xleftarrow{\tau_i} \Delta[1]$, where $k_0 = r$. Then $\mu_K$ is the pullback map induced by $(\mu_{K_1} \times_{1_{X_0}} \cdots \times_{1_{X_0}} \mu_{K_r})$ and
    $$
        Q := \mu_{K_0} \circ (\tau_1^* \times_{1_{X_0}} \cdots \times_{1_{X_0}} \tau_r^*) \circ (\mu_{K_1} \times_{1_{X_0}} \cdots \times_{1_{X_0}} \mu_{K_r}).
    $$
\end{proposition}

\begin{proof}
    The proof is by induction on the depth of $K_0$. Either $K_0 = \Delta[n]$, for which the result is simply by definition, or $K_0 = (K^1_0 \sqcup_{\Delta[0]} \cdots \sqcup_{\Delta[0]} K^p_0) \sqcup_{Sp(p)} \Delta[p]$ for simplicial composition diagrams $\Delta[h_j] \xrightarrow{\iota^0_j} K^j_0 \xleftarrow{\tau^0_j} \Delta[1]$.
    Suppose we are in the latter case. We have that
    $$
        K = (K_1' \sqcup_{\Delta[0]} \cdots \sqcup_{\Delta[0]} K_p') \sqcup_{Sp(p)} \Delta[p]
    $$
    for simplicial composition diagrams $\Delta[q_j] \xrightarrow{\iota_j'} K_j' \xleftarrow{\tau_j'} \Delta[1]$. We have moreover for each $1 \leq j \leq p$ that $K_j' = (H^j_1 \sqcup_{\Delta[0]} \cdots \sqcup_{\Delta[0]} H^j_{h_j}) \sqcup_{Sp(h_j)} K^j_0$ for simplicial composition diagrams $\Delta[s_{j, r}] \xrightarrow{\iota_{j, r}} H^j_r \xleftarrow{\tau_{j, r}} \Delta[1]$. The spaces $H^j_r$ are clearly just the spaces $K_i$.

    We wish to show that the induced pullback map
    \[\begin{tikzcd}
    	{X_{Sp(n)}} \\
    	& {X_K} & {X_{K_1 \sqcup_{\Delta[0]} \cdots \sqcup_{\Delta[0]} K_r}} \\
    	& {X_{K_0}} & {X_{Sp(r)}}
    	\arrow["{\iota_0^*}"', from=3-2, to=3-3]
    	\arrow["{\tau_1^* \times_{1_{X_0}} \cdots \times_{1_{X_0}} \tau_r^*}", from=2-3, to=3-3]
    	\arrow[from=2-2, to=3-2]
    	\arrow[from=2-2, to=2-3]
    	\arrow["\ulcorner"{anchor=center, pos=0.125}, draw=none, from=2-2, to=3-3]
    	\arrow[dashed, from=1-1, to=2-2]
    	\arrow["{\mu_{K_0} \circ (\tau_1^* \times_{1_{X_0}} \cdots \times_{1_{X_0}} \tau_r^*) \circ (\mu_{K_1} \times_{1_{X_0}} \cdots \times_{1_{X_0}} \mu_{K_r})}"', curve={height=12pt}, from=1-1, to=3-2]
    	\arrow["{\mu_{K_1} \times_{1_{X_0}} \cdots \times_{1_{X_0}} \mu_{K_r}}", curve={height=-12pt}, from=1-1, to=2-3]
    \end{tikzcd}\]
    is $\mu_K$. Our strategy is to consider the diagram
    \[\begin{tikzcd}
    	{X_{Sp(n)}} \\
    	& {X_K} & {X_{K_1' \sqcup_{\Delta[0]} \cdots \sqcup_{\Delta[0]} K_p'}} & {X_{K_1 \sqcup_{\Delta[0]} \cdots \sqcup_{\Delta[0]} K_r}} \\
    	& {X_p} & {X_{Sp(p)}} \\
    	& {X_{K_0}} && {X_{Sp(r)}}
    	\arrow["{\iota_0^*}"', from=4-2, to=4-4]
    	\arrow[from=2-4, to=4-4]
    	\arrow["{((\tau^0_1)^* \times_{1_{X_0}} \cdots \times_{1_{X_0}} (\tau^0_p)^*) \times_{1_{Sp(p)}} 1_{X_p}}"', from=4-2, to=3-2]
    	\arrow["Q"', from=1-1, to=4-2]
    	\arrow[dashed, from=1-1, to=2-2]
    	\arrow[from=2-2, to=3-2]
    	\arrow["{\gamma_p}"', from=3-2, to=3-3]
    	\arrow["\ulcorner"{anchor=center, pos=0.125}, draw=none, from=2-2, to=3-3]
    	\arrow["R", from=2-4, to=2-3]
    	\arrow[from=2-3, to=3-3]
    	\arrow[from=2-2, to=2-3]
    	\arrow["{\mu_{K_1} \times_{1_{X_0}} \cdots \times_{1_{X_0}} \mu_{K_r}}", from=1-1, to=2-4]
    \end{tikzcd}\]
    where
    $$
        R := 1_{X_{K_1 \sqcup_{\Delta[0]} \cdots \sqcup_{\Delta[0]} K_r}} \times_{1_{X_{Sp(r)}}} (\mu_{K_0^1} \times_{1_{X_0}} \cdots \times_{1_{X_0}} \mu_{K_0^p})
    $$
    and show that the induced morphisms $X_{Sp(n)} \rightarrow X_{K_1' \sqcup_{\Delta[0]} \cdots \sqcup_{\Delta[0]} K_p'}$ and $X_{Sp(n)} \rightarrow X_p$ are $(\mu_{K_1'} \times_{1_{X_0}} \cdots \times_{1_{X_0}} \mu_{K_p'})$ and
    $$
        \mu_p \circ ((\tau_1')^* \times_{1_{X_0}} \cdots \times_{1_{X_0}} (\tau_p')^*) \circ (\mu_{K_1'} \times_{1_{X_0}} \cdots \times_{1_{X_0}} \mu_{K_p'})
    $$
    respectively. This will imply the induced pullback map is $\mu_K$ by the inner pullback square.
    
    Note first that, for each $1 \leq j \leq p$, by induction on depth we have that the diagram
    \[\begin{tikzcd}
    	{X_{Sp(q_j)}} \\
    	{X_{H^j_1 \sqcup_{\Delta[0]} \cdots \sqcup_{\Delta[0]} H^j_{h_j}}} &&&&& {X_{K_j'}}
    	\arrow["{1_{X_{H^j_1 \sqcup_{\Delta[0]} \cdots \sqcup_{\Delta[0]} H^j_{h_j}}} \times_{1_{X_{Sp(h_j)}}} \mu_{K_0^j}}"', from=2-1, to=2-6]
    	\arrow["{\mu_{H^j_1} \times_{1_{X_0}} \cdots \times_{1_{X_0}} \mu_{H^j_{h_j}}}"', from=1-1, to=2-1]
    	\arrow["{\mu_{K_j'}}", from=1-1, to=2-6]
    \end{tikzcd}\]
    commutes. Thus, by taking a pullback of these maps, the first morphism is as required. For $Q$, it will suffice to show that the diagram
    \[\begin{tikzcd}
    	{X_{Sp(n)}} \\
    	& {X_{K_1' \sqcup_{\Delta[0]} \cdots \sqcup_{\Delta[0]} K_p'}} && {X_{K_1 \sqcup_{\Delta[0]} \cdots \sqcup_{\Delta[0]} K_r}} \\
    	& {X_{K^1_0 \sqcup_{\Delta[0]} \cdots \sqcup_{\Delta[0]} K^p_0}} && {X_{Sp(r)}} \\
    	& {X_{Sp(p)}} \\
    	{X_p} &&& {X_{K_0}}
    	\arrow["{\mu_p}"', from=4-2, to=5-1]
    	\arrow["{\tau_1^* \times_{1_{X_0}} \cdots \times_{1_{X_0}} \tau_r^*}", from=2-4, to=3-4]
    	\arrow["{(\tau_1^0)^* \times_{1_{X_0}} \cdots \times_{1_{X_0}} (\tau_p^0)^*}", from=3-2, to=4-2]
    	\arrow["{\mu_{K^1_0} \times_{1_{X_0}} \cdots \times_{1_{X_0}} \mu_{K^p_0}}"', from=3-4, to=3-2]
    	\arrow[from=1-1, to=5-1]
    	\arrow["{((\tau^0_1)^* \times_{1_{X_0}} \cdots \times_{1_{X_0}} (\tau^0_p)^*) \times_{1_{Sp(p)}} 1_{X_p}}", from=5-4, to=5-1]
    	\arrow["{\mu_{K_0}}", from=3-4, to=5-4]
    	\arrow[from=2-2, to=3-2]
    	\arrow[from=1-1, to=2-2]
    	\arrow["{\mu_{K_1} \times_{1_{X_0}} \cdots \times_{1_{X_0}} \mu_{K_r}}", from=1-1, to=2-4]
    \end{tikzcd}\]
    commutes, where the map $X_{Sp(n)} \rightarrow X_p$ is induced by the other morphisms. The leftmost pentagon commutes by definition, while the bottom right pentagon commutes by the definition of $\mu_{K_0}$ and the fact that the lowest horizontal map is simply the projection map from the pullback defining $X_{K_0}$ in terms of $X_p$ and the spaces $X_{K^j_0}$.

    For the uppermost pentagon, we consider each $K_j'$ and repective $K_j^0$ in turn for varying $j$. The respective diagram will commute in each case by induction. Thus, the whole pentagon commutes and the leftmost vertical map is indeed the map we seek.
\end{proof}

We can package the nested nature of $\mu_K$ into a genuine statement about composition operations analogous to the maps $\circ^{x_0, \cdots, x_n}$:

\begin{notation} \label{not_cd_simpcomp_op}
    Let $X$ be a $2$-fold Segal space and $Sp(n) \xrightarrow{\iota} K \xleftarrow{\tau} \Delta[1]$ a simplicial composition diagram. Let $x_0, \cdots, x_n \in (X_{0, 0})_0$. Then define
    \[\begin{tikzcd}
    	{\prod_{i = 1}^n X(x_{i-1}, x_i)} & {X_{Sp(n)}^{x_0, x_n}} & {X_K^{x_0, x_n}} \\
    	&& {X(x_0, x_n)}
    	\arrow[hook, from=1-1, to=1-2]
    	\arrow["{\mu_K^{x_0, x_n}}", from=1-2, to=1-3]
    	\arrow["{(\tau^*)^{x_0, x_n}}", from=1-3, to=2-3]
    	\arrow["{\circ^{x_0, \cdots, x_n}_K}"', from=1-1, to=2-3]
    \end{tikzcd}\]
\end{notation}

Note again that if $n = 0$, the domain of $\circ^x_K$ is the terminal object $\ast \cong \{x\}$.

\begin{theorem} \label{thm_cd_diag_is_comp_of_circs}
    Suppose $X$ is a $2$-fold Segal space and $Sp(k_i) \xrightarrow{\iota_i} K_i \xleftarrow{\tau_i} \Delta[1]$ are simplicial composition diagrams for $0 \leq i \leq n$, with $k_0 = n \geq 1$. Let $r = \sum_i k_i$. Suppose
    $$
        Y = ((x^1_0, \cdots, x^1_{k_1}), \cdots, (x^n_0, \cdots, x^n_{k_n}))
    $$
    is a nested list of elements of $(X_{0, 0})_0$ such that $x^i_{k_i} = x^{i+1}_0$ for $i < n$. Let $(x_0, \cdots, x_r)$ be the flattened version of this list where all $x^i_{k_i}$ have been removed for $i < n$.
    
    Then, setting $K := (K_1 \sqcup_{\Delta[0]} \cdots \sqcup_{\Delta[0]} K_n) \sqcup_{Sp(n)} K_0$, we have that
    $$
        \circ^{x_0, \cdots, x_r}_K = \circ^{x^1_0, x^2_0, \cdots, x^n_0, x^n_{k_n}}_{K_0} \circ (\circ^{x^1_0, \cdots, x^1_{k_1}}_{K_1} \times \cdots \times \circ^{x^n_0, \cdots, x^n_{k_n}}_{K_n}).
    $$
\end{theorem}

Before we can prove this theorem, we need some notation for compactness:

\begin{notation} \label{not_cd_prod_ai}
    For a series of spans $X_0 \leftarrow A_i \rightarrow X_0$ with $1 \leq i \leq n$ and $n > 0$, write
    $$
        \prod_{X_0}^{1 \leq i \leq n} A_i := A_1 \times_{X_0} \cdots \times_{X_0} A_n.
    $$
    If $n = 0$, set this to be the terminal object $\ast$.
\end{notation}

\begin{proof}
    By Proposition \ref{prop_cd_decomp_scd}, it will suffice to prove that the diagram
    \[\begin{tikzcd}
    	{\prod_{i = 1}^n \prod_{j = 1}^{k_i} X(x^i_{j - 1}, x^i_j)} & {\prod_{i = 1}^r X(x_{i-1}, x_i)} \\
    	{\prod_{i = 1}^{n}(\prod_{X_0}^{1 \leq i \leq k_i} X_1)^{x^i_0, x^i_{k_i}}} & {(\prod_{X_0}^{1 \leq i \leq r} X_1)^{x_0, x_r}} \\
    	{\prod_{i = 1}^{n}X_{K_i}^{x^i_0, x^i_{k_i}}} & {X_{K_1 \sqcup_{\Delta[0]} \cdots \sqcup_{\Delta[0]} K_n}^{x_0, x_r}} \\
    	{\prod_{i = 1}^{n}X(x^i_0, x^i_{k_i})} & {X_{Sp(n)}^{x_0, x_r}} \\
    	& {X_{K_0}^{x_0, x_r}} \\
    	& {X(x_0, x_r)}
    	\arrow[hook, from=1-2, to=2-2]
    	\arrow[hook, from=1-1, to=2-1]
    	\arrow[hook, from=2-1, to=2-2]
    	\arrow["{\prod_{i=1}^n \mu_{K_i}^{x^i_0, x^i_{k_i}}}"', from=2-1, to=3-1]
    	\arrow["{(\mu_{K_1} \times_{1_{X_0}} \cdots \times_{1_{X_0}} \mu_{K_n})^{x_0, x_r}}", from=2-2, to=3-2]
    	\arrow[hook, from=3-1, to=3-2]
    	\arrow["\cong", hook, from=1-1, to=1-2]
    	\arrow["{\prod_{i = 1}^n (\tau_i^*)^{x^i_0, x^i_{k_i}}}"', from=3-1, to=4-1]
    	\arrow["{(\tau_1^* \times_{1_{X_0}} \cdots \times_{1_{X_0}} \tau_n^*)^{x_0, x_r}}", from=3-2, to=4-2]
    	\arrow[hook, from=4-1, to=4-2]
    	\arrow["{\mu_{K_0}^{x_0, x_r}}"', from=4-2, to=5-2]
    	\arrow["{(\tau_0^*)^{x_0, x_r}}"', from=5-2, to=6-2]
    \end{tikzcd}\]
    commutes, where the inclusions are in general of the form $A \times_{X_0} \{x\} \times_{X_0} B \hookrightarrow A \times_{X_0} B$, where $A \rightarrow X_0 \leftarrow B$ is a cospan of simplicial spaces and $x \in (X_{0, 0})_0$.

    One can reduce the problem to checking each of the three squares in the diagram commutes. The first square is immediate, as the two paths are simply a matter of removing objects $x_i$ in different orders. The second two squares commute trivially; it is a matter of applying the inclusion before or after the vertical morphisms, which has no effect. Hence, the result is immediate, as the two maps being equated are both paths in this diagram.
\end{proof}

When we come to define homotopy bicategories, our hom-categories will be $h_1(X(x, y))$ with composition functors $h_1(\circ^{x_0, \cdots, x_n})$. To obtain coherence isomorphisms like associators, we need an essential technical property of these maps which respect \emph{maps between simplicial composition diagrams}:

\begin{definition} \label{defn_cd_map_scds}
    A \emph{map of simplicial composition diagrams} is a map $f : K_1 \rightarrow K_2$ in $\textbf{sSet}$ for simplicial composition diagrams $\Delta[n] \xrightarrow{\iota_i} K_i \xleftarrow{\tau_i} \Delta[1]$ for $i = \{1, 2\}$, such that $f \circ \iota_1 = \iota_2$ and $f \circ \tau_1 = \tau_2$.
\end{definition}

\begin{definition} \label{defn_cd_category_scdn}
    Let $\textbf{SCD}_n$ be the category whose objects are simplicial composition diagrams of arity $n$ and whose morphisms are maps of simplicial composition diagrams.
\end{definition}

\begin{proposition} \label{prop_cd_objfib_liftprob}
    Let $Sp(n) \xrightarrow{\iota_i} K_i \xleftarrow{\tau_i} \Delta[1]$ for $i \in \{1, 2\}$ be two simplicial composition diagrams of arity $n \geq 0$ with $f : K_1 \rightarrow K_2$ a map in $\textbf{SCD}_n$ between them. Let $X$ be a $2$-fold Segal space.

    Then for any $x_0, \cdots, x_n \in (X_{0, 0})_0$, there exists a map $\mu_{K_2}'$ solving the same lifting problem as $\mu_{K_1}$, namely giving a commutative diagram
    \[\begin{tikzcd}
    	& {X_{K_1}} \\
    	{X_{Sp(n)}} & {X_{Sp(n)}}
    	\arrow["id"', from=2-1, to=2-2]
    	\arrow["{\iota_1^*}", from=1-2, to=2-2]
    	\arrow["{\mu_{K_2}'}", dashed, from=2-1, to=1-2]
    \end{tikzcd}\]
    such that the diagram
    \[\begin{tikzcd}
    	{\prod_{i = 1}^n X(x_{i-1}, x_i)} & {X_{Sp(n)}^{x_0, x_n}} & {X_{K_1}^{x_0, x_n}} \\
    	&& {X(x_0, x_n)}
    	\arrow[hook, from=1-1, to=1-2]
    	\arrow["{(\mu_{K_2}')^{x_0, x_n}}", from=1-2, to=1-3]
    	\arrow["{(\tau_1^*)^{x_0, x_n}}", from=1-3, to=2-3]
    	\arrow["{\circ^{x_0, \cdots, x_n}_{K_2}}"', from=1-1, to=2-3]
    \end{tikzcd}\]
    commutes.
\end{proposition}

\begin{proof}
    We define $\mu_{K_2}' : X_{Sp(n)} \xrightarrow{\mu_{K_2}} X_{K_2} \xrightarrow{f^*} X_{K_1}$. We have that $\mu_{K_2}'$ is object-fibered, as $f^*$'s commutativity with the maps $\iota_i^*$ implies it is object-fibered. It is then clear that this is a solution to the lifting problem, since $\iota_1^* \circ \mu_{K_2}' = (f \circ \iota_1)^* \circ \mu_{K_2} = \iota_2^* \circ \mu_{K_2} = id$. Moreover, the second diagram commutes, since $\tau_1^* \circ f^* = (f \circ \tau_1)^* = \tau_2^*$ and we obtain the original definition of $\circ_{K_2}^{x_0, \cdots, x_n}$.
\end{proof}

The consequence of this result we need is that, for any $f : K_1 \rightarrow K_2$ as above, the maps $\circ^{x_0, \cdots, x_n}_{K_i}$ are of the form $\beta \circ \mu^{(i)} \circ \alpha$, where the two $\mu^{(i)}$ for $i \in \{1, 2\}$ are solutions to the same lifting problem. We will find that this implies there is a homotopy between the two $\circ^{x_0, \cdots, x_n}_{K_i}$, which will descend under $h_1$ to our associator.

A final object we should define is the following functor:

\begin{definition} \label{defn_cd_mathcal_g}
    Let $n > 0$ and $k_1, \cdots, k_n \geq 0$. Define the functor
    $$
        \mathcal{G}^n_{k_1, \cdots, k_n} : \textbf{SCD}_n \times \prod_{i = 1}^n \textbf{SCD}_{k_i} \rightarrow \textbf{SCD}_{k_1 + \cdots + k_n}
    $$
    defined by sending
    $$
        (K, K_1, \cdots, K_n) \mapsto (K_1 \sqcup_{\Delta[0]} \cdots \sqcup_{\Delta[0]} K_n) \sqcup_{Sp(n)} K
    $$
    with similar behavior for maps.
\end{definition}

One might note that this creates a (nonsymmetric) operad $\textbf{SCD}_\bullet$ valued in $\textbf{Cat}$. We will not need this fact particularly so we omit a proof, but this same structure will reappear later.
\section{Bicategories of Varying Biases}

Armed with a reasonable understanding of our domain, we are now ready to consider the codomain of our construction. Our goal is to extend the ideas of homotopy categories of Segal spaces to $2$-fold Segal spaces. It would seem sensible then that the result of our construction should be a bicategory of some description.

We believe that a more natural target may be \emph{unbiased bicategories}. Indeed, consider the nature of composition in a $2$-fold Segal space $X$; given a choice of horizontal compositions $\mu_n$, we obtain composition maps $\circ^{x_0, \cdots, x_n}$ for all $n \geq 2$ rather than merely binary composition. Such operations would be most faithfully represented by an `unbiased' construction, where there are indeed specified operations to compose $n$ morphisms at once.

In the below definitions and indeed throughout this paper, we take the convention that for natural isomorphisms $\phi : F \Rightarrow G$, $\psi : G \Rightarrow H$ and $\theta : Q \Rightarrow R$, vertical composition is written $\psi \phi : F \Rightarrow H$ and horizontal is $\theta \circ \phi : Q \circ F \Rightarrow R \circ G$.

\begin{definition}[{\cite[Def. 1.2.1]{leinsterOperadsHigherDimensionalCategory2000}}] \label{defn_ubicat_ubicat}
    An \emph{unbiased bicategory} $\mathscr{B}$ is a collection of the following data:
    \begin{enumerate}
        \item A collection of objects $\textbf{ob}(\mathscr{B})$;
        \item For each pair of objects $x$ and $y$, a category of $1$-morphisms $x \to y$ and $2$-morphisms $f \Rightarrow g$ between them, $\textbf{Hom}_\mathscr{B}(x, y)$;
        \item For each tuple of objects $X = (x_0, \cdots, x_n)$ for $n \in \mathbb{Z}_{> 0}$, a functor
        $$
        \circ^X : \textbf{Hom}_\mathscr{B}(x_0, x_1) \times \cdots \times \textbf{Hom}_\mathscr{B}(x_{n-1}, x_n) \rightarrow \textbf{Hom}_\mathscr{B}(x_0, x_n)$$
        which we write as
        $$
        (f_1, \cdots, f_n) \mapsto (f_n \circ \cdots \circ f_1) = \bigcirc_{i = 1}^n f_i
        $$
        for $f_i \in \textbf{Hom}_{\mathscr{B}}(x_{i-1}, x_i)$ all $1$-morphisms (or all $2$-morphisms);
        \item For each object $x$, a functor $\circ^x : \star \rightarrow \textbf{Hom}_\mathscr{B}(x, x)$ from the discrete singleton category, identifying an element we denote as $()$;
        
        \item For each $n \in \mathbb{Z}_{> 0}$ and $k_1, \cdots, k_n \in \mathbb{Z}_{\geq 0}$, for each sequence of tuples $X_i = (x^i_0, \cdots, x^i_{k_i})$ such that $x^i_{k_i} = x^{i+1}_0$ for every $i < n$, setting
        $$
        X = (X_1, \cdots, X_n)
        $$
        and
        $$
        Y = (x^1_0, \cdots, x^1_{k_1}, x^2_{1}, \cdots, x^2_{k_2}, \cdots, x^n_{1}, \cdots, x^n_{k_n})
        $$
        a natural isomorphism
        $$
        \gamma_X :  \circ^{(x^1_0, x^1_{k_1}, x^2_{k_2}, \cdots, x^n_{k_n})} \circ (\circ^{X_1} \times \circ^{X_2} \times \cdots \times \circ^{X_n}) \Rightarrow \circ^Y
        $$
        sending $((f^n_{k_n} \circ \cdots \circ f^n_1) \circ \cdots \circ (f^1_{k_1} \circ \cdots \circ f^1_1))$ to $(f^n_{k_n} \circ \cdots \circ f^1_1)$, written levelwise as $\gamma_{((f^1_1, \cdots, f^1_{k_1}), \cdots, (f^n_1, \cdots, f^n_{k_n}))} : \bigcirc_{i = 1}^n (\bigcirc_{j = 1}^{k_i} f_j^i) \rightarrow \bigcirc_{i = 1}^n \bigcirc_{j = 1}^{k_i} f_j^i$;
        \item For each pair of objects $x$ and $y$, a natural isomorphism $\iota_{x, y} : 1_{\textbf{Hom}_\mathscr{B}(x, y)} \Rightarrow \circ^{x, y}$ sending $f$ to $(f)$,
    \end{enumerate}
    such that:
    
    \begin{enumerate}
        \item (associativity) for any $n, m_1, \cdots, m_n \in \mathbb{Z}_{> 0}$, integers $k_1^1, \cdots, k_{m_n}^n \in \mathbb{Z}_{\geq 0}$ and thrice-nested sequence of objects $(((x_{p, q, r})_{r=0}^{k^p_q})_{q=1}^{m_p})_{p=1}^n$ such that $x_{k^p_q, q, p} = x_{0, q+1, p}$ for $q < m_p$ and $x_{k^p_{m_p}, m_p, p} = x_{0, 1, p+1}$ for $p < n$, together with any sequence of $1$-morphisms $f_{p, q, r} \in \textbf{Hom}_{\mathscr{B}}(x_{p,q,r-1}, x_{p, q, r})$ for $p, q, r \geq 1$, the diagram
        \[\begin{tikzcd}
        	& {\bigcirc_{p=1}^n (\bigcirc_{q=1}^{m_p} ( \bigcirc_{r=1}^{k^p_q} f_{p, q, r}))} \\
        	{\bigcirc_{p=1}^n (\bigcirc_{q=1}^{m_p} \bigcirc_{r=1}^{k^p_q} f_{p, q, r})} && {\bigcirc_{p=1}^n \bigcirc_{q=1}^{m_p} ( \bigcirc_{r=1}^{k^p_q} f_{p, q, r})} \\
        	& {\bigcirc_{p=1}^n \bigcirc_{q=1}^{m_p} \bigcirc_{r=1}^{k^p_q} f_{p, q, r}}
        	\arrow["{\bigcirc_{p = 1}^n \gamma_{D_p}}"'{pos=0.7}, from=1-2, to=2-1]
        	\arrow["{\gamma_E}"'{pos=0.4}, from=2-1, to=3-2]
        	\arrow["{\gamma_{D}}"{pos=0.6}, from=1-2, to=2-3]
        	\arrow["{\gamma_F}"{pos=0.4}, from=2-3, to=3-2]
        \end{tikzcd}\]
        commutes, where
        \begin{enumerate}
            \item $D_p = ((f_{p, q, r})_{r=1}^{k_q^p})_{q = 1}^{m_p}$;
            \item $D = ((\bigcirc_{r = 1}^{k_q^p} f_{p, q, r})_{q=1}^{m_q})_{p=1}^n$;
            \item $E = ((f_{p, q, r})_{r = 1, q = 1}^{k_q^p, m_p})_{p = 1}^n$;
            \item $F = ((f_{p, q, r})_{r=1}^{k_q^p})_{q = 1, p = 1}^{m_q, n}$;
        \end{enumerate}
    
        \item (unitality) for any sequence of objects $x_0, \cdots, x_n \in \textbf{ob}(\mathscr{B})$ and $1$-morphisms $f_i \in \textbf{Hom}_{\mathscr{B}}(x_{i-1}, x_i)$, the diagram
        \[\begin{tikzcd}
        	& {\bigcirc_{i=1}^n f_i} \\
        	{(\bigcirc_{i = 1}^n f_i)} && {\bigcirc_{i = 1}^n(f_i)} \\
        	& {\bigcirc_{i = 1}^n f_i}
        	\arrow["{\iota_{\bigcirc_{i = 1}^n f_i}}"', from=1-2, to=2-1]
        	\arrow["{\gamma_{((f_1, \cdots, f_n))}}"'{pos=0.1}, from=2-1, to=3-2]
        	\arrow["{\bigcirc_{i=1}^n \iota_{f_i}}", from=1-2, to=2-3]
        	\arrow["{\gamma_{((f_1), \cdots, (f_n))}}"{pos=0.1}, from=2-3, to=3-2]
        	\arrow["{1_{\bigcirc_{i=1}^n f_i}}"{description}, from=1-2, to=3-2]
        \end{tikzcd}\]
        commutes.
    \end{enumerate}
\end{definition}

Note the existence of object identities given by the elements $()$. The unitality condition on these is now implicit in the associators, rather than the unitors.

We must also define pseudofunctors between these unbiased bicategories:

\begin{definition}[{\cite[Def. 1.2.3]{leinsterOperadsHigherDimensionalCategory2000}}] \label{defn_ubicat_pseudofunc}
    An \emph{unbiased pseudofunctor}, or \emph{unbiased weak functor} $P : \mathscr{B} \rightarrow \mathscr{C}$ between unbiased bicategories $\mathscr{B}$ and $\mathscr{C}$ is the following collection of data:
    \begin{enumerate}
        \item A mapping of objects $\textbf{ob}(\mathscr{B}) \rightarrow \textbf{ob}(\mathscr{C})$, mapping $x \mapsto P(x)$;
        \item For each pair of objects $x$ and $y$ in $\mathscr{B}$, a functor
        $$
            P_{x, y} : \textbf{Hom}_{\mathscr{B}}(x, y) \rightarrow \textbf{Hom}_{\mathscr{C}}(P(x), P(y));
        $$
        \item For each $n \in \mathbb{Z}_{> 0}$ and objects $x_0, \cdots, x_n$, a natural isomorphism
        $$
        \pi_{x_0, \cdots, x_n} : \circ_\mathscr{C}^{P(x_0), \cdots, P(x_n)} \circ (P_{x_0, x_1} \times \cdots \times P_{x_{n-1}, x_n}) \Rightarrow P_{x_0, x_n} \circ_\mathscr{B}^{x_0, \cdots, x_n}
        $$
        which is levelwise a morphism $\pi_{(f_1, \cdots, f_n)} : \bigcirc_{i = 1}^n P_{x_{i-1}, x_i}(f_i) \rightarrow P_{x_0, x_n}(\bigcirc_{i = 1}^n f_i)$ for $f_i \in \textbf{Hom}_{\mathscr{B}}(x_{i-1}, x_i)$;

        \item A natural isomorphism $\pi_{x} : \circ_{\mathscr{C}}^{P(x)} \Rightarrow P_{x, x} \circ \circ_{\mathscr{B}}^x$ of the form $() \rightarrow P(())$
    \end{enumerate}
    such that:
    
    \begin{enumerate}
        \item For any $n > 0$, integers $k_1, \cdots, k_n \in \mathbb{Z}_{\geq 0}$ and nested sequence of objects $((x^i_j)_{j = 0}^{k_i})_{i = 1}^n$ in $\mathscr{B}$ such that $x^i_{k_i} = x^{i+1}_0$ for $i < n$, along with morphisms $f^i_j \in \textbf{Hom}_{\mathscr{B}}(x^i_{j-1}, x^i_j)$ for $1 \leq i \leq n$ and $1 \leq j \leq k_i$, the diagram
        \[\begin{tikzcd}
        	{\bigcirc_{i = 1}^n (\bigcirc_{j = 1}^{k_i} P_{x_{j-1}^i, x_j^i}(f^i_j))} && {\bigcirc_{i = 1}^n \bigcirc_{j = 1}^{k_i} P_{x_{j-1}^i, x_j^i}(f^i_j)} \\
        	{\bigcirc_{i = 1}^n P_{x_0^i, x_{k_i}^i} (\bigcirc_{j = 1}^{k_i} f^i_j)} \\
        	{P_{x^1_0, x^n_{k_n}} (\bigcirc_{i = 1}^n (\bigcirc_{j = 1}^{k_i} f^i_j))} && {P_{x^1_0, x^n_{k_n}} (\bigcirc_{i = 1}^n \bigcirc_{j = 1}^{k_i} f^i_j)}
        	\arrow["{\bigcirc_{i=1}^n \pi_{(f^i_1, \cdots, f^i_{k_i})}}"', from=1-1, to=2-1]
        	\arrow["{\pi_{(\bigcirc_{j = 1}^{k_1} f^1_j, \cdots, \bigcirc_{j=1}^{k_n} f^n_j)}}"', from=2-1, to=3-1]
        	\arrow["{\pi_{(f^1_1, \cdots, f^n_{k_n})}}", from=1-3, to=3-3]
        	\arrow["{\gamma_D}", from=1-1, to=1-3]
        	\arrow["{P_{x^1_0, x^n_{k_n}}(\gamma_{D'})}"', from=3-1, to=3-3]
        \end{tikzcd}\]
        commutes, where:
        \begin{enumerate}
            \item $D = ((P_{x^i_{j-1}, x^i_j}(f_j^i))_{j = 1}^{k_i})_{i = 1}^n$;
            \item $D' = ((f_j^i)_{j = 1}^{k_i})_{i = 1}^n$;
        \end{enumerate}

        \item For each $x, y \in \textbf{ob}(\mathscr{B})$ and $f \in \textbf{Hom}_{\mathscr{B}}(x, y)$, the diagram
        \[\begin{tikzcd}
        	{P_{x, y}(f)} & {(P_{x, y}(f))} \\
        	& {P_{x, y}((f))}
        	\arrow["{P_{x, y} \iota_f}"', from=1-1, to=2-2]
        	\arrow["{\iota_{P_{x, y}(f)}}", from=1-1, to=1-2]
        	\arrow["{\pi_{(f)}}", from=1-2, to=2-2]
        \end{tikzcd}\]
        commutes.
    \end{enumerate}
\end{definition}

Note again the subtle insertion of $()$ here, with $\pi_{x_0} : () \rightarrow P_{x_0, x_0}(())$.

We wish to show that these form a category $\textbf{UBicat}$ of small unbiased bicategories and unbiased pseudofunctors. We are thus obliged to define the composite of two unbiased pseudofunctors:

\begin{definition}[{\cite[pg. 7-8]{leinsterOperadsHigherDimensionalCategory2000}}] \label{defn_ubicat_func_comp}
    Let $\mathscr{B}, \mathscr{C}$ and $\mathscr{D}$ be unbiased bicategories. Let $P : \mathscr{B} \rightarrow \mathscr{C}$ and $Q : \mathscr{C} \rightarrow \mathscr{D}$ be unbiased pseudofunctors, with natural isomorphisms $\pi$ and $\theta$, respectively. Then define $Q \circ P : \mathscr{B} \rightarrow \mathscr{D}$ to be the unbiased pseudofunctor with:

    \begin{enumerate}
        \item Objects mapping as $x \mapsto Q(P(x))$;
        \item $(Q \circ P)_{x, y} := Q_{P(x), P(y)} \circ P_{x, y}$ for $x, y \in \textbf{ob}(\mathscr{B})$;
        \item Natural isomorphisms
        $$
            \psi_{x_0, \cdots, x_n} := \big( Q_{P(x_0), P(x_n)}(\pi_{x_0, \cdots, x_n}) \big) \big( \theta_{P(x_0), \cdots, P(x_n)} \circ (P_{x_0, x_1} \times \cdots \times P_{x_{n-1}, x_n}) \big)
        $$
        which are thus levelwise of the form
        $$
            \bigcirc_{i = 1}^n Q_{P(x_{i-1}), P(x_i)}(P_{x_{i-1}, x_i}(f_i)) \rightarrow Q_{P(x_0), P(x_n)}(P_{x_0, x_n}(\bigcirc_{i = 1}^n f_i));
        $$
        \item Natural isomorphisms $\psi_x := \big( Q_{x, x}(\pi_x) \big) \theta_{P(x)}$.
    \end{enumerate}
\end{definition}

That this defines a genuine pseudofunctor is routine. One may also check that this composition is associative and unital, with identity pseudofunctors clearly being identities on objects and hom-categories with trivial natural isomorphisms. Thus, we have the following:

\begin{proposition}[{\cite[pg. 8]{leinsterOperadsHigherDimensionalCategory2000}}] \label{prop_ubicat_cat}
    There is a category $\textbf{UBicat}$, whose objects are small unbiased bicategories, morphisms are unbiased pseudofunctors with composition defined as in Definition \ref{defn_ubicat_func_comp}.
\end{proposition}

We thus might imagine that we seek a functor
$$
    h_2 : \textbf{SeSp}_2 \rightarrow \textbf{UBicat}.
$$
However, complications will arise over the domain for the functor $h_2$. Indeed, somehow we will need to choose the horizontal composites $\mu_n$ for every $2$-fold Segal space such that they result in a genuine functor. We choose to have a domain as follows:

\begin{definition} \label{defn_ubicat_cssp2comp}
    Let $\textbf{SeSp}_2^{comp}$ be the category whose objects are pairs $(X, (\mu_n)_{n \geq 0})$ of $2$-fold Segal spaces with a choice of horizontal composition and morphisms $(X, (\mu_n)_{n \geq 0}) \rightarrow (Y, (\nu_n)_{n \geq 0})$ are maps $X \rightarrow Y$ in $\textbf{SeSp}_2$.
\end{definition}

Note how the morphisms do not account for the choice of horizontal compositions in either the domain or codomain.

We hence seek a functor
$$
    h_2 : \textbf{SeSp}_2^{comp} \rightarrow \textbf{UBicat}.
$$

\subsection{Classical Versus Unbiased Bicategories}

One might balk at the definition of an unbiased bicategory. Why make use of such a definition when classical bicategories are better understood and more widely used? They have existed since B{\'e}nabou's definition in \cite{benabouIntroductionBicategories1967} and remain one of the oldest and best-understood instances of a higher category.

Such questions can be answered with the insight that nothing has truly been lost. Unbiased bicategories can be converted into bicategories easily enough:

\begin{definition}[{\cite[pg. 10]{leinsterOperadsHigherDimensionalCategory2000}}] \label{defn_ubicat_underlying_bicat}
    Let $\mathscr{B}$ be an unbiased bicategory. The \emph{underlying bicategory} $\Tilde{\mathscr{B}}$ is the classical bicategory given by the following data:
    \begin{enumerate}
        \item Objects and hom-categories are as in $\mathscr{B}$;
        
        \item The composition functors $\circ^{x, y, z} : \textbf{Hom}_{\Tilde{\mathscr{B}}}(x, y) \times \textbf{Hom}_{\Tilde{\mathscr{B}}}(y, z) \rightarrow \textbf{Hom}_{\Tilde{\mathscr{B}}}(x, z)$ for $x, y, z \in \textbf{ob}(\Tilde{\mathscr{B}})$ are the binary composition functors in $\mathscr{B}$;

        \item Identity functors $1_x : \ast \rightarrow \textbf{Hom}_{\Tilde{\mathscr{B}}}(x, x)$ for all $x \in \textbf{ob}(\Tilde{\mathscr{B}})$ are given by the functors $\circ^x$ in $\mathscr{B}$;
        
        \item The associators $\gamma_{w, x, y, z}$ for $w, x, y, z \in \Tilde{\mathscr{B}}$ are given by the tuples $X_1 := ((w, x), (x, y, z))$ and $X_2 := ((w, x, y), (y, z))$, along with the composites
        $$
            \alpha_{w, x, y, z} := \big( 1_{\circ^{w, y, z}} \circ (1_{\circ^{x, y, z}} \times \iota_{y, z}^{-1}) \big) \gamma_{X_2}^{-1} \gamma_{X_1} \big( 1_{\circ^{w, x, z}} \circ (\iota_{w, x} \times 1_{\circ^{x, y, z}}) \big)
        $$
        of domain $\circ^{w, x, z} \circ (1_{\textbf{Hom}_{\Tilde{\mathscr{B}}}(w, x)} \times \circ^{x, y, z})$ and codomain $\circ^{w, y, z} \circ (\circ^{w, x, y} \times 1_{\textbf{Hom}_{\Tilde{\mathscr{B}}}(y, z)})$;

        \item The unitors $\lambda_{x, y}$ and $\rho_{x, y}$ for $x, y \in \textbf{ob}(\Tilde{\mathscr{B}})$ are given by the natural isomorphisms
        $$
            \lambda_{x, y} := \iota^{-1}_{x, y} \gamma_{(x), (x, y)} \big( 1_{\circ^{x, x, y}} \circ (1_{\circ^x} \times \iota_{x, y}) \big) : \circ^{x, x, y} \circ (\circ^x \times 1_{\textbf{Hom}_{\Tilde{\mathscr{B}}}(x, y)}) \Rightarrow 1_{\textbf{Hom}_{\Tilde{\mathscr{B}}}(x, y)}
        $$
        and
        $$
            \rho_{x, y} := \iota^{-1}_{x, y} \gamma_{(x, y), (y)} \big( 1_{\circ^{x, y, y}} (\iota_{x, y} \times 1_{\circ^y}) \big) : \circ^{x, y, y} \circ (1_{\textbf{Hom}_{\Tilde{\mathscr{B}}}(x, y)} \times \circ^{y}) \Rightarrow 1_{\textbf{Hom}_{\Tilde{\mathscr{B}}}(x, y)}.
        $$
    \end{enumerate}
\end{definition}

The pentagon and triangle axioms may be checked to hold from the axioms of an unbiased bicategory. Indeed, the pentagon axiom quickly follows from pondering the following diagram:

\[\begin{tikzcd}
	&& {((fg)h)k} \\
	{(f(gh))k} & {(fgh)k} & fghk & {(fg)hk} & {(fg)(hk)} \\
	& {f(gh)k} & {f(ghk)} & {fg(hk)} \\
	& {f((gh)k)} && {f(g(hk))}
	\arrow[dotted, from=1-3, to=2-1]
	\arrow[dotted, from=2-1, to=4-2]
	\arrow[dotted, from=4-2, to=4-4]
	\arrow[dotted, from=2-5, to=4-4]
	\arrow[dotted, from=1-3, to=2-5]
	\arrow[dashed, from=1-3, to=2-2]
	\arrow[dashed, from=2-1, to=2-2]
	\arrow[dashed, from=2-1, to=3-2]
	\arrow[dashed, from=4-2, to=3-2]
	\arrow[dashed, from=4-2, to=3-3]
	\arrow[dashed, from=4-4, to=3-3]
	\arrow[dashed, from=4-4, to=3-4]
	\arrow[dashed, from=2-5, to=3-4]
	\arrow[dashed, from=2-5, to=2-4]
	\arrow[dashed, from=1-3, to=2-4]
	\arrow[from=2-2, to=2-3]
	\arrow[from=3-2, to=2-3]
	\arrow[from=3-3, to=2-3]
	\arrow[from=3-4, to=2-3]
	\arrow[from=2-4, to=2-3]
\end{tikzcd}\]

The dotted arrows are the newly constructed associators in the biased bicategory, while the dashed and solid arrows are associators in the original unbiased bicategory. Commutativity now follows from all the five triangles and quadrilaterals in the pentagon's interior commuting. The triangles commute by definition, while the quadrilaterals are each an instance of the associativity condition of an unbiased bicategory. Thus, the result holds.

One may also construct a pseudofunctor between classical bicategories from an unbiased pseudofunctor, just by restricting to binary composition. The laws of a pseudofunctor are then induced. This conversion induces a functor $\Tilde{\bullet} : \textbf{UBicat} \rightarrow \textbf{Bicat}$, to the category of bicategories and pseudofunctors:

\begin{theorem} [{\cite[Thm. 1.3.1]{leinsterOperadsHigherDimensionalCategory2000}}] \label{thm_ubicat_equiv}
    For any bicategory $\mathscr{B}$, $\Tilde{\mathscr{B}}$ is a bicategory. Moreover, $\Tilde{\bullet}$ defines a genuine functor from $\textbf{UBicat}$ to $\textbf{Bicat}$. This functor is fully faithful and essentially surjective.
\end{theorem}

It is clear however that there is no `canonical' inverse to the functor $\Tilde{\bullet}$. There are many ways to obtain an unbiased bicategory from a biased one, with a countably infinite number of choices resulting just from building the unbiased compositions out of different nestings of binary composition operations. We will not explore this further here.
\section{Homotopies and Globular $2$-Homotopies}

It is now time for us to address the question of coherence isomorphisms. These will be achieved by a certain notion of \emph{homotopy} between maps $f, g : X \rightarrow Y$ in $\textbf{SeSp}^{inj}$. In general, it is our aim to show that, for a given $2$-fold Segal space $X$ with some choice of horizontal compositions $(\mu_n)_{n \geq 0}$, maps $f : K_1 \rightarrow K_2$ of simplicial composition diagrams induce homotopies from $\mu_{K_1}$ to $f^\ast \circ \mu_{K_2}$, which will in turn induce homotopies between $\circ^{x_0, \cdots, x_n}_{K_1}$ and $\circ^{x_0, \cdots, x_n}_{K_2}$ for all $x_0, \cdots, x_n \in (X_{0, 0})_0$.

The payoff of this result is in the images of these homotopies under $h_1$. We will design our notion of homotopy to be such that the image of any homotopy under $h_1$ is a natural isomorphism between functors. Moreover, we will show that, under certain conditions, any two homotopies between the same two functors will be related by a higher `homotopy between homotopies', which we will refer to as a \emph{globular $2$-homotopy}. The image of such a $2$-homotopy under $h_1$ will be an equality between natural isomorphisms. Our coherence isomorphisms will then be the images of certain homotopies between composition operations $\circ^{x_0, \cdots, x_n}_K$ under $h_1$, which will satisfy the coherence conditions because of the existence of certain globular $2$-homotopies.

The main concrete result of this section is an enrichment of $h_1$ by Kan complexes, whose domain $\textbf{SeSp}^{inj}$ will be such that the hom-space from $A$ to $B$ will be the Kan complex of `higher homotopies' between maps $A \rightarrow B$. The $1$-simplices will be homotopies, while the $2$-simplices will be some generalization of globular $2$-homotopies to a simplicial context. Moreover, the codomain $\textbf{Cat}$ will be enriched by natural isomorphisms between functors. We will also give a more abstract justification for globular $2$-homotopies, proving they exist in a general model category and that their presence between certain homotopies is not specific to simplicial spaces or Segal spaces.

\subsection{Homotopy Categories of Segal Spaces}

Before we can proceed further, we will need a natural extension of the fact that $\pi_0$ preserves products to $h_1$. We do not claim originality for this elementary observation as it is noted in \cite[pg. 31]{rezkModelHomotopyTheory2000}. However, we have not yet found a proof in the literature, so we provide one here:

\begin{lemma} \label{lemma_h1_prods}
    $h_1$ preserves products up to natural isomorphism of categories.
\end{lemma}

\begin{proof}
    Let $X$ and $Y$ be Segal spaces. We seek to specify an isomorphism in $\textbf{Cat}$ of the form $p : h_1(X \times Y) \rightarrow h_1(X) \times h_1(Y)$, which is natural in $X$ and $Y$.

    On objects, there is an evident natural bijection $\textbf{ob}(h_1(X \times Y)) \rightarrow \textbf{ob}(h_1(X) \times h_1(Y))$. For morphisms, note that
    \begin{align*}
        (X \times Y)((x_1, y_1), (x_2, y_2)) &:= \{(x_1, y_1)\} \times_{X_0 \times Y_0} (X_1 \times Y_1) \times_{X_0 \times Y_0} \{(x_2, y_2)\} \\
        &\cong (\{x_1\} \times_{X_0} X_1 \times_{X_0} \{x_2\}) \times (\{y_1\} \times_{Y_0} Y_1 \times_{Y_1} \{y_2\}) \\
        &= X(x_1, x_2) \times Y(x_2, y_2).
    \end{align*}
    Hence, there are natural bijections
    $$
        \textbf{Hom}_{h_1(X \times Y)}((x_1, y_1), (x_2, y_2)) \rightarrow \textbf{Hom}_X(x_1, x_2) \times \textbf{Hom}_Y(y_1, y_2)
    $$
    induced by $\pi_0$ preserving products. That these respect identities and composition is easily verified.
\end{proof}

With this out of the way, the current aim for us is to show that $h_1$ sends some notion of homotopies to natural isomorphisms.

To understand what this means in a precise sense, we will devise a suitable interpretation of the categories $\textbf{SeSp}^{inj}$ and $\textbf{Cat}$ as $\textbf{sSet}$-enriched categories that $h_1$ respects, where $\textbf{SeSp}^{inj}$ is enriched via homotopies between maps and $\textbf{Cat}$ is enriched via natural isomorphisms between functors. The latter is classical, though we will need a modification of it suitable for our methods.

\begin{definition} \label{defn_h1_i_n}
    Define $I[n] \in \textbf{Cat}$ to be the unique groupoid with objects $\{0, \cdots, n\}$ and with a unique morphism between any two objects.
\end{definition}

\begin{definition}
    Define $\textbf{Iso} : \textbf{Cat} \rightarrow \textbf{Grpd}$ to be the functor sending $\mathscr{C}$ to the maximal subgroupoid $\textbf{Iso}(\mathscr{C})$ of $\mathscr{C}$.
\end{definition}

Note that $\textbf{Iso}$ is right adjoint to the natural inclusion $\iota : \textbf{Grpd} \rightarrow \textbf{Cat}$, which itself evidently preserves finite products and terminal objects. Thus, the functors $\textbf{Iso}, \iota$ and the right adjoint functor $\textbf{nerve} : \textbf{Cat} \rightarrow \textbf{sSet}$ are all monoidal under the monoidal structures on $\textbf{Cat}, \textbf{Grpd}$ and $\textbf{sSet}$ given by Cartesian products. This implies the functor
$$
    \textbf{nerve} \circ \iota \circ \textbf{Iso} : \textbf{Cat} \rightarrow \textbf{sSet}
$$
is monoidal under this structure on $\textbf{Cat}$ and $\textbf{sSet}$.

Note moreover that $\textbf{Cat}$ is naturally enriched over itself by considering categories $\textbf{Fun}(-, -)$ of functors and natural transformations. Thus, by \cite[Lemma 3.4.3]{riehlCategoricalHomotopyTheory2014}, applying $\textbf{nerve} \circ \iota \circ \textbf{Iso}$ to this enrichment allows us to obtain the following enrichment of $\textbf{Cat}$ by simplicial sets:

\begin{definition} \label{defn_h1_cat_s}
    Let $\textbf{Cat}_s$ be the $\textbf{sSet}$-enriched category of small categories with hom-spaces of the form
    $$
        \textbf{Hom}_{\textbf{Cat}_s}(\mathscr{C}, \mathscr{D}) := \textbf{nerve}(\textbf{Iso}(\textbf{Fun}(\mathscr{C}, \mathscr{D}))).
    $$
    The set of $n$-simplices for this simplicial set may be written as
    $$
        \textbf{Hom}_{\textbf{Cat}_s}(\mathscr{C}, \mathscr{D})_n \cong \textbf{Hom}_{\textbf{Cat}}(\mathscr{C} \times I[n], \mathscr{D}).
    $$
\end{definition}

There is a clear cosimplicial object $I[\bullet] : \Delta \rightarrow \textbf{Grpd}$ sending each $[n]$ to $I[n]$, defining the simplicial maps here. To see how the above description of natural isomorphisms applies, note that a natural isomorphism $F \cong G$ of functors $F, G : \mathscr{C} \rightarrow \mathscr{D}$ can be alternatively be written as a functor $\mathscr{C} \times I[1] \rightarrow \mathscr{D}$.

\subsection{Cosimplicial Resolutions}

The story for Segal spaces demands some more advanced technology. We could certainly state that $\textbf{SeSp}^{inj}$ is enriched in $\textbf{sSet}$ almost immediately, using $\textbf{Map}_1^0(-, -)$. This is however not the structure we seek. Our wish is for the $1$-simplices in our enrichment to represent homotopies of some description, which should be sent by $h_1$ to natural isomorphisms. More precisely, we wish to have a notion of homotopy between maps $f, g : X \rightarrow Y$ in $\textbf{SeSp}^{inj}$ of the form $H : X \times M \rightarrow Y$ for some fixed $M \in \textbf{SeSp}^{inj}$ such that $h_1(M) \cong I[1]$. If this condition is satisfied, then we may apply $h_1$ to obtain a natural isomorphism from $H$ between $h_1(f)$ and $h_1(g)$ of the form
$$
    h_1(X) \times I[1] \cong h_1(X) \times h_1(M) \cong h_1(X \times M) \xrightarrow{h_1(H)} h_1(Y).
$$
We would moreover like these maps $X \times M \rightarrow Y$ to be the $1$-simplices in our enrichment of $\textbf{SeSp}^{inj}$ and in turn $h_1$.

Let $X, Y \in \textbf{SeSp}^{inj}$. A $1$-simplex in $\textbf{Map}_1^0(X, Y)$ is a morphism $X \square_1^0 \Delta[1] \rightarrow Y$. It is at this point we notice a fundamental obstruction: the constant bisimplicial set $\iota_1^0\Delta[1]$ is not Reedy fibrant, so is not an object in $\textbf{SeSp}^{inj}$.

Thus, we must seek some alternative construction. One other such potential enrichment of $\textbf{SeSp}^{inj}$ we could turn to is given by mapping spaces from $X$ to $Y$ of the form
$$
    \textbf{Map}_F(X, Y) := ((Y^X)_\bullet)_0.
$$
The $n$-simplices are now of the form $f : X \times F_1^0(n) \rightarrow Y$. Notably, $F_1^0(n)$ is Reedy fibrant as it is levelwise discrete, as noted more generally in \cite[pg. 6]{rezkModelHomotopyTheory2000}. Moreover, it is evidently a Segal space; indeed, the Segal maps in this case are bijections. Thus, we may apply $h_1$ to such a $1$-simplex $f$ and obtain a functor
$$
    h_1(f) : h_1(X) \times h_1(F_1^0(1)) \cong h_1(X) \times [1] \rightarrow h_1(Y).
$$
This is precisely a natural transformation between functors. However, we are interested only in natural \emph{isomorphisms}, as all of the coherence $2$-morphisms we wish to construct for our homotopy bicategories must be as such. Hence, though one could likely proceed with the above enrichment, we it will be beneficial for us to instead pick an enrichment of $\textbf{SeSp}^{inj}$ whose $1$-simplices are guaranteed to be sent by $h_1$ to natural isomorphisms on the nose.

A sensible path forwards is to identify some deeper mechanism which gives rise to the appearance of $\Delta[1]$ in our theory and consider whether $\Delta[1]$ can be thus substituted for something more well-behaved, while continuing to reap all the benefits said mechanism provided. The correct such structure is that of a \emph{homotopy function complex} and the underlying \emph{cosimplicial resolution}.

\begin{definition}[{\cite[Not. 16.1.1]{hirschhornModelCategoriesTheir2009}} {\cite[Def. 16.1.2]{hirschhornModelCategoriesTheir2009}}] \label{defn_loc_cosimp_res}
    Let $\mathscr{M}$ be a model category. Let $X \in \mathscr{M}$. Define $\textbf{cc}_*(X) \in \mathscr{M}^\Delta$ to be the constant functor to $X$.

    A \emph{cosimplicial resolution of $X$} is a cofibrant approximation $\Tilde{X} \rightarrow \textbf{cc}_*(X)$ in the Reedy model category structure\footnote{Note that the source Reedy category is now $\Delta$ rather than $\Delta^{op}$.} on $\mathscr{M}^\Delta$.
\end{definition}

\begin{definition}[{\cite[Def. 17.1.1 and Notation 16.4.1]{hirschhornModelCategoriesTheir2009}}] \label{defn_loc_left_htpy_func_cplx}
    Let $\mathscr{M}$ be a model category and $X, Y \in \mathscr{M}$. A \emph{left homotopy function complex from $X$ to $Y$} is a triple
    $$
        (\Tilde{X}, \hat{Y}, \mathscr{M}(\Tilde{X}, \hat{Y}))
    $$
    where $\Tilde{X}$ is a cosimplicial resolution of $X$, $\hat{Y}$ is a fibrant approximation to $Y$ and $\mathscr{M}(\Tilde{X}, \hat{Y})$ is the simplicial set which at level $n$ is the set $\textbf{Hom}_{\mathscr{M}}(\Tilde{X}_n, \hat{Y})$.
\end{definition}

The reason to think about these is that cosimplicial resolutions are closely tied to notions of left homotopy, always exhibiting valid cylinder objects:

\begin{definition}[{\cite[Def. 4.2]{dwyerSpalinskiHomotopyTheoriesModel1995}}] \label{defn_h1_cylobj}
    Let $\mathscr{M}$ be a model category and $C$ an object in $\mathscr{M}$. A \emph{cylinder object of $C$} is a factorization
    $$
        C \sqcup C \xrightarrow{i} C \wedge I \xrightarrow{p} C
    $$
    of the map $C \sqcup C \rightarrow C$ such that $p$ is a weak equivalence. The cylinder object is a \emph{good cylinder object} if $i$ is a cofibration and a \emph{very good cylinder object} if it is good and $p$ is a trivial fibration.
\end{definition}

\begin{definition}[{\cite[pg. 19]{dwyerSpalinskiHomotopyTheoriesModel1995}}] \label{defn_h1_lefthtpy}
    Let $\mathscr{M}$ be a model category and $f, g : C \rightarrow D$ be morphisms in $\mathscr{M}$. Suppose $C \sqcup C \xrightarrow{i} C \wedge I \xrightarrow{p} C$ is a cylinder object for $C$. A \emph{left homotopy $H : C \wedge I \rightarrow D$ from $f$ to $g$} is a map $H$ such that $H \circ i$ is the map $f + g : C \sqcup C \rightarrow D$.
    
    The map $H$ is moreover a \emph{good left homotopy} if $C \wedge I$ is a good cylinder object and a \emph{very good left homotopy} if $C \wedge I$ is a very good cylinder object.
\end{definition}

\begin{proposition}[{\cite[Prop. 16.1.6]{hirschhornModelCategoriesTheir2009}}] \label{prop_h1_cosimp_res_cylobj}
    Let $\mathscr{M}$ be a model category. If $\Tilde{X}$ is a cosimplicial resolution of $X$ in $\mathscr{M}$, then $$\Tilde{X}_0 \sqcup \Tilde{X}_0 \xrightarrow{\Tilde{X}_{\langle 0 \rangle} \sqcup \Tilde{X}_{\langle 1 \rangle}} \Tilde{X}_1 \xrightarrow{{\Tilde{X}_{\langle 0, 0 \rangle}}} \Tilde{X}_0$$ is a good cylinder object of $\Tilde{X}_0$.
\end{proposition}

One may show that $\textbf{Map}_1^0(-, -)$ is a left homotopy function complex induced by a cosimplicial resolution whose cylinder objects are of the form $X \sqcup X \rightarrow X \square_1^0 \Delta[1] \rightarrow X$ \cite[Cor. 2.2.53]{romoTowardsAlgebraicNCategories}. It seems reasonable then that replacing this cosimplicial resolution is the correct way to replace $\Delta[1]$. We should choose the resolution so that the resulting cylinder objects are of the form $X \sqcup X \rightarrow X \times K \rightarrow X$ for some $K$ such that $h_1(K) \cong I[1]$. This will imply we obtain an enrichment in homotopy function complexes whose $1$-simplices are sent by $h_1$ to natural isomorphisms. We will find that this indeed leads to the enrichment we seek.

Another advantage to this approach to building an enrichment is that our mapping spaces will always be Kan complexes:

\begin{proposition}[{\cite[Prop. 17.1.3]{hirschhornModelCategoriesTheir2009}}] \label{prop_h1_funccomp_kan}
    Suppose $\mathscr{M}$ is a model category and $X$ and $Y$ are objects of $\mathscr{M}$. Then a left homotopy function complex is a Kan complex.
\end{proposition}

This will later let us compose our left homotopies vertically.

Our chosen cosimplicial resolution involves the \emph{classifying diagram} functor. Recall that $\textbf{Iso}(\mathscr{C})$ for a category $\mathscr{C}$ is its maximal subgroupoid. Moreover, recall that $\textbf{Iso} : \textbf{Cat} \rightarrow \textbf{Grpd}$ is right adjoint to the inclusion $\textbf{Grpd} \rightarrow \textbf{Cat}$.

\begin{definition}[{\cite[pg. 8]{rezkModelHomotopyTheory2000}}] \label{defn_h1_class_diag}
    The \emph{classifying diagram functor} $N : \textbf{Cat} \rightarrow \textbf{sSpace}$ is defined such that for any category $\mathscr{C}$ and $n \geq 0$,
    $$
    N\mathscr{C}_n := \textbf{nerve}(\textbf{Iso}(\mathscr{C}^{[n]}))
    $$
    with the evident simplicial maps and behavior on functors.
\end{definition}

\begin{proposition}[{\cite[Prop. 6.1]{rezkModelHomotopyTheory2000}}] \label{prop_h1_n_cssp}
    Let $\mathscr{C} \in \textbf{Cat}$. Then $N \mathscr{C}$ is a complete Segal space.
\end{proposition}

\begin{proposition}[{\cite[Thm. 3.7]{rezkModelHomotopyTheory2000}}]
    Let $\mathscr{C}, \mathscr{D} \in \textbf{Cat}$. Then $N(\mathscr{C} \times \mathscr{D}) \cong N(\mathscr{C}) \times N(\mathscr{D})$.
\end{proposition}

\begin{proposition}[{\cite[pg. 13]{rezkModelHomotopyTheory2000}}] \label{prop_h1_n_cancel}
    There is a canonical natural isomorphism $h_1 \circ N \cong 1_{\textbf{Cat}}$.
\end{proposition}

For a Segal space $X$, we take our cosimplicial resolution to be $\Tilde{X} := X \times (N \circ I[\bullet])$, so that $\Tilde{X}_n := X \times N(I[n])$.

\begin{proposition} \label{prop_h1_ni_cosimp_res}
    Suppose $X \in \textbf{SeSp}^{inj}$. Then the functor $X \times N(I[\bullet])$, sending $[n] \mapsto X \times N(I[n])$, defines a cosimplicial resolution of $X$ in $\textbf{sSpace}^{inj}$.
\end{proposition}

\begin{proof}
    We need to prove that the natural map $X \times N(I[\bullet]) \rightarrow \textbf{cc}_\ast(X)$ is a levelwise weak equivalence. This is in fact a trivial fibration levelwise; each map $X \times N(I[n]) \rightarrow X$ is a pullback map for the cospan $X \rightarrow \ast \leftarrow N(I[1])$. Note that since $N$ has its image in $\textbf{CSSP}^{inj}$ by Proposition \ref{prop_h1_n_cssp}, the map $N(I[1]) \rightarrow \ast$ is a fibration in $\textbf{sSpace}^{inj}$. It is moreover a weak equivalence since $I[1] \rightarrow \ast$ is an equivalence of categories and $N$ sends equivalences to levelwise weak equivalences of simplicial spaces \cite[Thm. 3.7]{rezkModelHomotopyTheory2000}.

    It now suffices to prove that $X \times N(I[\bullet])$ is Reedy cofibrant. In particular, we must prove that the latching maps for all $n \geq 0$ of the form
    $$
        \text{colim}_{k \in \partial \Delta[n]} \big( X \times N(I[k]) \big) \cong X \times \big( \text{colim}_{k \in \partial \Delta[n]} N(I[k]) \big) \rightarrow X \times N(I[n])
    $$
    are cofibrations in $\textbf{sSpace}^{inj}$, namely monomorphisms, for which it is sufficient to prove that the maps
    $$
        L(n) := \text{colim}_{k \in \partial \Delta[n]} N(I[k]) \rightarrow N(I[n])
    $$
    are monomorphisms. Thus, we are tasked with proving that the cosimplicial object $N(I[\bullet]) : \Delta \rightarrow SeSp^{inj}$ is Reedy cofibrant.

    Let $m \geq 0$. Note that the above maps are monomorphisms if and only if they are levelwise, namely if and only if each map
    $$
        L(n)_m \cong \text{colim}_{k \in \partial \Delta[n]} N(I[k])_m = L_n (N(I[\bullet])_m) \rightarrow N(I[n])_m
    $$
    is a cofibration in $\textbf{sSet}$. Hence, it suffices to prove that the cosimplicial simplicial sets $N(I[\bullet])_m : \Delta \rightarrow \textbf{sSet}$ are Reedy cofibrant.

    These cosimplicial simplicial sets are, for $n \geq 0$, of the form
    $$
        N(I[n])_m \cong \textbf{nerve}(\textbf{Iso}(I[n]^{[m]})) = \textbf{nerve}(I[n]^{[m]}).
    $$
    By \cite[Cor. 15.9.10]{hirschhornModelCategoriesTheir2009}, a cosimplicial simplicial set $X : \Delta \rightarrow \textbf{sSet}$ is Reedy cofibrant if and only if the equalizer of the map
    $$
        X(0) \sqcup X(0) \rightarrow X(1)
    $$
    is empty. However, this is clearly the case above; the two maps $\textbf{nerve}(I[0]^{m}) \rightarrow \textbf{nerve}(I[1]^{[m]})$ send a $k$-simplex $[k] \times [m] \rightarrow I[0]$ to a map $[k] \times [m] \rightarrow I[1]$ with constant image in either the object $0$ or $1$, respectively. Hence, each $N(I[\bullet])_m$ is a Reedy cofibrant cosimplicial simplicial set, implying $N(I[\bullet])$ is Reedy cofibrant in $(\textbf{sSpace}^{inj})^{\Delta}$, as needed.
\end{proof}

Hence, since Segal spaces are fibrant in the model structure $\textbf{sSpace}^{inj}$, we can build a Kan complex $\textbf{Map}_{\textbf{sS}}^I(X, Y)$ for any Segal spaces $X$ and $Y$, which at level $k$ is the set of maps $X \times N(I[k]) \rightarrow Y$, given by the left homotopy function complex induced by our cosimplicial resolution.

\begin{proposition} \label{prop_h1_enrich_kan}
    $\textbf{Map}_{\textbf{sS}}^I(-, -) : (\textbf{SeSp}^{inj})^{op} \times \textbf{SeSp}^{inj} \rightarrow \textbf{sSet}$ defines an enrichment of $\textbf{SeSp}^{inj}$ in Kan complexes.
\end{proposition}

\begin{proof}
    Suppose $X, Y$ and $Z$ are Segal spaces and let $k \geq 0$. Let $f : X \times N(I[k]) \rightarrow Y$ and $g : Y \times N(I[k]) \rightarrow Z$. Then we have a natural map
    $$
        X \times N(I[k]) \xrightarrow{1_X \times D_k} X \times N(I[k]) \times N(I[k]) \xrightarrow{f \times 1_{N(I[k])}} Y \times N(I[k]) \xrightarrow{g} Z
    $$
    where the map $D_k : N(I[k]) \rightarrow N(I[k]) \times N(I[k])$ is the diagonal map.
    We will use this as our composite of $g$ and $f$.

    The resulting operations are compatible with the simplicial structure on $\textbf{Map}^I_{\textbf{sS}}(-, -)$, giving an operation
    $$
        \circ : \textbf{Map}^I_{\textbf{sS}}(X, Y) \times \textbf{Map}^I_{\textbf{sS}}(Y, Z) \rightarrow \textbf{Map}^I_{\textbf{sS}}(X, Z).
    $$
    Identities are given by $X \times N(I[1]) \rightarrow X \times N(I[0]) \cong X \xrightarrow{1_X} X$. Associativity and identity laws are routine.
\end{proof}

We thus obtain our $\textbf{sSet}$-enrichment for $\textbf{SeSp}^{inj}$. Moreover, the enrichment of $h_1$ is now immediate: there is an evident simplicial map for Segal spaces $X, Y$ of the form
$$
    \textbf{Map}^I_{\textbf{sS}}(X, Y) \rightarrow \textbf{nerve}(\textbf{Iso}(\textbf{Fun}(h_1(X), h_1(Y))))
$$
sending $f : X \times N(I[k]) \rightarrow Y$ to 
\begin{align*}
    h_1(X) \times I[k] \cong h_1(X) \times h_1(N(I[k])) &\cong h_1(X \times N(I[k])) \\
    &\xrightarrow{h_1(f)} h_1(Y).
\end{align*}

We will in general write $h_1(H) : h_1(X) \times I[1] \rightarrow h_1(Y)$ for this induced functor.

\begin{proposition} \label{prop_h1_enrich_functor}
    By the above mapping, $h_1$ extends to an $\textbf{sSet}$-enriched functor.
\end{proposition}

\begin{proof}
    It is clear that identities are preserved, so we are left to show the same for composition. However, one notes that if a natural isomorphism $\alpha : F \Rightarrow G$ between functors $F, G : \mathscr{C} \rightarrow \mathscr{D}$ is written in the form $\mathscr{C} \times I[1] \rightarrow \mathscr{D}$, then horizontal composition is exactly as stated for $\textbf{Map}^I_{\textbf{sS}}$ at all levels of $\textbf{nerve}(\textbf{Iso}(\textbf{Fun}(\mathscr{C}, \mathscr{D})))$. Hence, composition is preserved as needed.
\end{proof}

We now have the desired simplicial set enrichment. We also have that the $1$-simplices in $\textbf{Map}_{\textbf{sS}}^I(X, Y)$ are valid left homotopies, as $X \times N(I[1])$ is a cylinder object for $X$ owing to it being part of a cosimplicial resolution. We thus have established how $h_1$ can transmit left homotopies to natural isomorphisms in an `$\infty$-groupoidal' manner.

A useful consequence of this particular choice of cylinder is that it is in fact a very good cylinder object:

\begin{proposition} \label{prop_h1_very_good_cyl_obj}
    Let $X$ be a simplicial space. Then $X \times N(I[1])$ is a very good cylinder object for $X$ in $\textbf{sSpace}^{inj}$.
\end{proposition}

\begin{proof}
    By Proposition \ref{prop_h1_cosimp_res_cylobj} and Proposition \ref{prop_h1_ni_cosimp_res}, we have a good cylinder object. Moreover, the object $N(I[1])$ is Reedy fibrant, so that the map $X \times N(I[1]) \rightarrow X$ is a fibration as needed.
\end{proof}

One final fact we will need about left homotopies is their capacity to be composed. That left homotopies can be composed in general is classical; by \cite[Def. 7.4.3]{hirschhornModelCategoriesTheir2009}, given two good cylinder objects $X \wedge I$ and $X \wedge I'$ of a cofibrant object $X$ in a model category $\mathscr{M}$, one may explicitly produce a `composite' of any two good left homotopies $H : X \wedge I \rightarrow Y$ and $H' : X \wedge I' \rightarrow Y$ in $\mathscr{M}$ by using the pushout
$$
    (X \wedge I) \sqcup_X (X \wedge I')
$$
as a new good cylinder object. We provide the details for our case, which concretely shows moreover how we may continue to use precisely the same kind of cylinder objects as we have before:

\begin{definition} \label{defn_h1_htpy_comp}
    Consider a lifting problem of Segal spaces
    \[\begin{tikzcd}
    	& B \\
    	C & D
    	\arrow[from=2-1, to=2-2]
    	\arrow["p", from=1-2, to=2-2]
    \end{tikzcd}\]
    where $p$ is a trivial fibration. Suppose $f, g, h : C \rightarrow B$ are solutions, with left homotopies $H, K : C \times N(I[1]) \rightarrow B$ over $D$, from $f$ to $g$ and from $g$ to $h$ respectively. Then a \emph{composite} of $H$ and $K$ is a left homotopy $KH : C \times N(I[1]) \rightarrow B$ over $D$ from $f$ to $h$ defined by the composition $Q \circ i$ in the diagram
    \[\begin{tikzcd}
    	& {C \sqcup C \sqcup C} \\
    	& {C \times (N(I[1]) \sqcup_\ast N(I[1]))} && B \\
    	{C \times N(I[1])} & {C \times N(I[2])} && D
    	\arrow["{f \sqcup g \sqcup h}", from=1-2, to=2-4]
    	\arrow["p", from=2-4, to=3-4]
    	\arrow[hook', from=1-2, to=2-2]
    	\arrow["j"', hook', from=2-2, to=3-2]
    	\arrow[from=3-2, to=3-4]
    	\arrow["{H \sqcup_{g} K}"{description}, dashed, from=2-2, to=2-4]
    	\arrow["Q"{description}, dashed, from=3-2, to=2-4]
    	\arrow["i", hook, from=3-1, to=3-2]
    \end{tikzcd}\]
    where $i$ is the image of $\langle 0, 2 \rangle : \Delta[1] \rightarrow \Delta[2]$ in the cosimplicial object $C \times N(I[\bullet])$, while $j$ on each component $C \times N(I[1])$ in its domain is the image of the maps $\langle 0, 1 \rangle$ and $\langle 1, 2 \rangle$ respectively.
\end{definition}

Note that $Q$ need not be unique in the above definition; in general, there may be many candidate composites for two given homotopies.

A trivial result is then immediate:

\begin{proposition} \label{prop_h1_htpy_comp_is_comp}
    In the situation of Definition \ref{defn_h1_htpy_comp}, the map $Q$ is a horn filler for the $2$-horn $\Lambda^2_1 \rightarrow \textbf{Map}_{\textbf{sS}}^I(C, B)$ defined by $H$ and $K$ on each of its respective $1$-simplices. The homotopy $KH$ represents the edge $\langle 0, 2 \rangle$ of $Q$.
\end{proposition}

Hence, it is reasonable to see $KH$ as a valid composite of $H$ and $K$ in the Kan complex $\textbf{Map}_{\textbf{sS}}^I(C, B)$. Of course, $KH$ is not at all unique. This will not matter when we come to need this result.

\subsection{Globular 2-Homotopies}

We are in need of some result that establishes when $h_1$ sends two left homotopies $K, H : X \times N(I[1]) \rightarrow Y$ between the same maps $f, g : X \rightarrow Y$ to the same natural isomorphism.  A suitable such criterion is the existence of a \emph{globular $2$-homotopy} between $K$ and $H$:

\begin{definition}
    Let $X, Y \in \textbf{SeSp}^{inj}$. Let $f, g : X \rightarrow Y$ and let $K, H : X \times N(I[1]) \rightarrow Y$ be two left homotopies from $f$ to $g$. Then a \emph{globular $2$-homotopy} $\alpha : K \Rightarrow H$ is a left homotopy from $K$ to $H$, namely a functor $X \times N(I[1]) \times N(I[1]) \rightarrow Y$, such that the restrictions to $X \times \{0\} \times N(I[1])$ and $X \times \{1\} \times N(I[1])$ are constantly $f$ and $g$ respectively.
\end{definition}

Globular $2$-homotopies are notably similar to the \emph{correspondences} of \cite[pg. 23, Def. 2]{quillenHomotopicalAlgebra1967}, for our particular model category and cylinder objects. The two are not precisely equal, though we believe they should be equivalent in some suitable sense.

Globular $2$-homotopies are, in spirit, homotopies between two homotopies of maps of Segal spaces. They moreover enjoy a place in our enrichment of $\textbf{SeSp}^{inj}$ by supposed Kan complexes of `higher homotopies', helping to justify this interpretation thereof. Consider the two maps
$$
    A, B : I[2] \rightarrow I[1] \times I[1]
$$
sending $0 \mapsto (0, 0)$ and $2 \mapsto (1, 1)$, where $A$ sends $1$ to $(0, 1)$ while $B$ sends $1$ to $(1, 0)$. Then for a globular $2$-homotopy $\alpha : X \times N(I[1]) \times N(I[1]) \rightarrow Y$, we obtain two maps
$$
    X \times N(I[2]) \rightarrow X \times N(I[1] \times I[1]) \cong X \times N(I[1]) \times N(I[1]) \rightarrow Y.
$$
In our enrichment of $\textbf{SeSp}^{inj}$ by $\textbf{Map}^I_{\textbf{sS}}(-, -)$, these correspond precisely to those $2$-simplices $\alpha \in \textbf{Map}^I_{\textbf{sS}}(X, Y)_2$ such that either the $1$-simplex of the form $\textbf{Map}^I_{\textbf{sS}}(X, Y)_{\langle 0, 1 \rangle}(\alpha)$ or the $1$-simplex of the form $\textbf{Map}^I_{\textbf{sS}}(X, Y)_{\langle 1, 2 \rangle}(\alpha)$ is degenerate. We may thus interpret globular $2$-homotopies as the $2$-morphisms in each $\infty$-groupoid $\textbf{Map}^I_{\textbf{sS}}(X, Y)$.

A particularly appealing aspect of our specific definition for globular $2$-homotopies is that, upon an application of $h_1$, a globular $2$-homotopy between two homotopies $H$ and $K$ will explicitly identify the natural isomorphisms $h_1(H)$ and $h_1(K)$:

\begin{proposition}
    Let $X, Y, f, g, K, H$ be as above. Suppose $\alpha : K \Rightarrow H$ is a globular $2$-homotopy from $K$ to $H$. Then the two natural isomorphisms $h_1(H), h_1(K) : h_1(X) \times I[1] \rightarrow h_1(Y)$ are equal.
\end{proposition}

\begin{proof}
    We find that the map, which we denote by $h_1(\alpha)$, of the form
    \begin{align*}
        h_1(X) \times I[1] \times I[1] &\cong h_1(X) \times h_1(N(I[1])) \times h_1(N(I[1])) \\
        &\cong h_1(X \times N(I[1]) \times N(I[1])) \xrightarrow{h_1(\alpha)} h_1(Y)
    \end{align*}
    reduces to $h_1(f)$ on $h_1(X) \times \{0\} \times I[1]$ and $h_1(g)$ on $h_1(X) \times \{1\} \times I[1]$. Moreover, it restricts to $h_1(K)$ and $h_1(H)$ on the other evident restrictions. Thus, $h_1(K) = h_1(H)$ as needed.
\end{proof}

An important consequence to this proposition for us lies in comparing solutions of lifting problems. Consider a lifting problem with two possible solutions
\[\begin{tikzcd}
	A && B \\
	C && D
	\arrow["f", from=1-1, to=1-3]
	\arrow["i"', from=1-1, to=2-1]
	\arrow["p", from=1-3, to=2-3]
	\arrow["\delta"{description}, curve={height=-6pt}, dashed, from=2-1, to=1-3]
	\arrow["\epsilon"{description}, curve={height=6pt}, dashed, from=2-1, to=1-3]
	\arrow["g"', from=2-1, to=2-3]
\end{tikzcd}\]
where $i$ is a cofibration and $p$ a trivial fibration. We have the following useful result:

\begin{proposition}[{\cite[Prop. 7.6.13]{hirschhornModelCategoriesTheir2009}}]\label{prop_h1_2lift_htpy}  Let $\mathscr{M}$ be a model category. Consider the above solid arrow diagram, where adding in either $\delta$ or $\epsilon$ alone still makes the diagram commute. Then there is a left homotopy from $\delta$ to $\epsilon$ in the model category $(A \downarrow \mathscr{M} \downarrow D)$.
\end{proposition}

It will be important for us to understand how this is proven.

\begin{proof}
    There is a model structure on $(A \downarrow \mathscr{M} \downarrow D)$ given in \cite[Thm. 7.6.5]{hirschhornModelCategoriesTheir2009}, where weak equivalences, fibrations and cofibrations are those which restrict to such maps in $\mathscr{M}$.

    Now, factor the map $C \sqcup_A C \rightarrow C$ as $C \sqcup_A C \overset{j}{\rightarrow} C \wedge I \overset{r}{\rightarrow} C$ as a good cylinder object, so $j$ is a cofibration and $r$ a weak equivalence. Consider the solid arrow diagram
    \[\begin{tikzcd}
    	{C \sqcup_A C} && B \\
    	{C \wedge I} & C & D
    	\arrow["{\delta \sqcup_f \epsilon}", from=1-1, to=1-3]
    	\arrow["j"', from=1-1, to=2-1]
    	\arrow["p", from=1-3, to=2-3]
    	\arrow["H"{description}, dashed, from=2-1, to=1-3]
    	\arrow["r"', from=2-1, to=2-2]
    	\arrow["g"', from=2-2, to=2-3]
    \end{tikzcd}\]
    where $j$ is a cofibration and $p$ is a trivial fibration. Thus, there is a dotted arrow $H$ making the diagram commute, which makes $H$ a left homotopy from $\delta$ to $\epsilon$ in $(A \downarrow \mathscr{M} \downarrow D)$ as needed.
\end{proof}

A side result should be noted with regards to these homotopies and taking fibers:

\begin{proposition} \label{prop_fib_of_htpy_is_htpy}
    Suppose $\mathscr{M} = \textbf{sSpace}^{inj}$ in the above lifting problem, $X$ is a $2$-fold Segal space and that $f, i, p, g, \delta$ and $\epsilon$ are all object-fibered over $(X_0)^{n+1}$. Then the induced homotopy $H$ will also be object-fibered, with respect to the vertex maps $C \wedge I \rightarrow C \rightarrow (X_0)^{n+1}$. Moreover, for any $x_0, \cdots, x_n \in (X_{0, 0})_0$, the map $H^{x_0, \cdots, x_n} : (C \wedge I)^{x_0, \cdots, x_n} \rightarrow B^{x_0, \cdots, x_n}$ is a left homotopy from $\delta^{x_0, \cdots, x_n}$ to $\epsilon^{x_0, \cdots, x_n}$.
\end{proposition}

\begin{proof}
    The map $j$ is clearly object-fibered, by the naturally induced map from the pushout $C \sqcup_A C \rightarrow (X_0)^{n+1}$. The map $\delta \sqcup_f \epsilon$ must also be object-fibered, while $p$ and $g$ are by assumption. Thus, by Proposition \ref{prop_cd_objfibered_liftprob}, the solution $H$ of the lifting problem must also be object-fibered such that commutativity of the diagram is preserved after taking a fiber.

    Note that $r$ is object-fibered. Moreover, $r^{x_0, \cdots, x_n}$ is a trivial fibration and thus a weak equivalence, so $(C \wedge I)^{x_0, \cdots, x_n}$ is a cylinder object, though perhaps not a good or very good one anymore. Hence, $H^{x_0, \cdots, x_n}$ is a valid left homotopy as claimed.
\end{proof}

Note that in the case $C \wedge I = C \times N(I[1])$, the homotopy will in fact stay a very good left homotopy.

If $A = \emptyset$ is the initial object, the cylinder object in question is also a cylinder object in $(\mathscr{M} \downarrow D)$, as well as in $\mathscr{M}$, yielding a lifting problem of the form
\[\begin{tikzcd}
	{C \sqcup C} && B \\
	{C \wedge I} && D
	\arrow["{\delta \sqcup \epsilon}", from=1-1, to=1-3]
	\arrow["g"', from=2-1, to=2-3]
	\arrow["p", from=1-3, to=2-3]
	\arrow["i"', from=1-1, to=2-1]
	\arrow["H", dashed, from=2-1, to=1-3]
\end{tikzcd}\]

We can then compare two such left homotopies $H, K : C \wedge I \rightarrow B$. As per the proof of Proposition \ref{prop_h1_2lift_htpy}, these are two solutions to a new lifting problem, letting us inductively apply the proposition to obtain a left homotopy $\Gamma$ from $H$ to $K$ in $(C \sqcup C \downarrow \mathscr{M} \downarrow D)$. Of course, the cylinder object this homotopy will be defined with respect to is not guaranteed to be such that, for instance, we precisely obtain a globular $2$-homotopy in the case of Segal spaces. We will show how to resolve this situation in the most general context possible, as the structure of a globular $2$-homotopy between homotopies of solutions to a lifting problem may be of interest in other model categories.

We should spend some time understanding precisely what such a left homotopy is in our use case. Consider, for $f : X \rightarrow Z$ a map of simplicial spaces and $p : Y \rightarrow Z$ a trivial fibration in $\textbf{sSpace}^{inj}$, two solutions $\alpha$ and $\beta$ of the induced lifting problem and two left homotopies $H, K : \alpha \Rightarrow \beta$ induced between these solutions as follows:

\[\begin{tikzcd}
	&& Y \\
	X && Z
	\arrow["f"', from=2-1, to=2-3]
	\arrow["p", two heads, from=1-3, to=2-3]
	\arrow["\alpha"{description}, curve={height=-6pt}, dashed, from=2-1, to=1-3]
	\arrow["\beta"{description}, curve={height=6pt}, dashed, from=2-1, to=1-3]
\end{tikzcd}
\hspace{2em}
\begin{tikzcd}
	{X \sqcup X} && Y \\
	{X \times N(I[1])} && Z
	\arrow["{f \circ (1_X \times !)}"', from=2-1, to=2-3]
	\arrow["p", two heads, from=1-3, to=2-3]
	\arrow["H"{description}, curve={height=-6pt}, dashed, from=2-1, to=1-3]
	\arrow["K"{description}, curve={height=6pt}, dashed, from=2-1, to=1-3]
	\arrow[hook', from=1-1, to=2-1]
	\arrow["{\alpha \sqcup \beta}", from=1-1, to=1-3]
\end{tikzcd}\]

Then there is an induced left homotopy $\Gamma$ in $(X \sqcup X \downarrow \textbf{sSpace}^{inj} \downarrow Z)$ from $H$ to $K$. We need a suitable cylinder object in this model category to be able to write down $\Gamma$ directly. One may produce a general such cylinder object. Before we do so, a technical lemma is needed for the construction:

\begin{lemma} \label{lemma_h1_v_cofib}
    Suppose $C \sqcup C \xrightarrow{i} C \wedge I \xrightarrow{p} C$ is a very good cylinder object for $C$ in a model category $\mathscr{M}$. Then there is a very good cylinder object
    $$
        (C \wedge I) \sqcup (C \wedge I) \xrightarrow{h} C \wedge I \wedge I \xrightarrow{r} C \wedge I
    $$
    for $C \wedge I$ and a cofibration $v$ that solves the lifting problem
    \[\begin{tikzcd}
    	{(C \sqcup C) \sqcup (C \sqcup C)} & {(C \wedge I) \sqcup (C \wedge I)} & {C \wedge I \wedge I} \\
    	{(C \sqcup C) \sqcup (C \sqcup C)} \\
    	{(C \wedge I) \sqcup (C \wedge I)} & {C \sqcup C} & {C \wedge I}
    	\arrow["{i \sqcup i}", from=1-1, to=1-2]
    	\arrow["\tau"', from=1-1, to=2-1]
    	\arrow["{i \sqcup i}"', from=2-1, to=3-1]
    	\arrow["h", from=1-2, to=1-3]
    	\arrow["r", from=1-3, to=3-3]
    	\arrow["v", dashed, from=3-1, to=1-3]
    	\arrow["{p \sqcup p}"', from=3-1, to=3-2]
    	\arrow["i"', from=3-2, to=3-3]
    \end{tikzcd}\]
    where $\tau : (C \sqcup C) \sqcup (C \sqcup C) \rightarrow (C \sqcup C) \sqcup (C \sqcup C)$ is the isomorphism permuting the $C$'s by the permutation $(1 3 2 4)$\footnote{Read this as the map $(a_1 a_2 a_3 a_4)$ sending $i \mapsto a_i$, rather than the cyclic permutation.}.
\end{lemma}

\begin{proof}
    Consider the pushout diagram
    \[\begin{tikzcd}
    	{(C \sqcup C) \sqcup (C \sqcup C)} & {(C \sqcup C) \sqcup (C \sqcup C)} & {(C \wedge I) \sqcup (C \wedge I)} \\
    	{(C \wedge I) \sqcup (C \wedge I)} && {C \wedge \square}
    	\arrow["\tau", from=1-1, to=1-2]
    	\arrow["{i \sqcup i}"', from=1-1, to=2-1]
    	\arrow["{i \sqcup i}", from=1-2, to=1-3]
    	\arrow["{v'}", from=1-3, to=2-3]
    	\arrow["{h'}"', from=2-1, to=2-3]
    	\arrow["\lrcorner"{anchor=center, pos=0.125, rotate=180}, draw=none, from=2-3, to=1-2]
    \end{tikzcd}\]
    We note that $h'$ and $v'$ are both cofibrations, by preservation along pushouts. Now, a universal map $f : C \wedge \square \rightarrow C \wedge I$ is obtained from the pushout by the diagram
    \[\begin{tikzcd}
    	{(C \sqcup C) \sqcup (C \sqcup C)} & {(C \sqcup C) \sqcup (C \sqcup C)} & {(C \wedge I) \sqcup (C \wedge I)} & {C \sqcup C} \\
    	{(C \wedge I) \sqcup (C \wedge I)} && {C \wedge \square} \\
    	&&& {C \wedge I}
    	\arrow["\tau", from=1-1, to=1-2]
    	\arrow["{i \sqcup i}"', from=1-1, to=2-1]
    	\arrow["{i \sqcup i}", from=1-2, to=1-3]
    	\arrow["{p \sqcup p}", from=1-3, to=1-4]
    	\arrow["{v'}", from=1-3, to=2-3]
    	\arrow["i", from=1-4, to=3-4]
    	\arrow["{h'}"', from=2-1, to=2-3]
    	\arrow[curve={height=12pt}, two heads, from=2-1, to=3-4]
    	\arrow["\lrcorner"{anchor=center, pos=0.125, rotate=180}, draw=none, from=2-3, to=1-2]
    	\arrow["f", dashed, from=2-3, to=3-4]
    \end{tikzcd}\]
    Factorize this into a cofibration followed by a trivial fibration $C \wedge \square \xrightarrow{\iota} C \wedge I \wedge I \xrightarrow{r} C \wedge I$. The map $h$ is then defined to be $\iota \circ h'$, concluding the construction of the very good cylinder object. The map $v$ is then set to be $\iota \circ v'$. This is clearly a cofibration and solves the required lifting problem by commutativity of the pushout diagram.
\end{proof}

Note that for us, given $\mathscr{M} = \textbf{sSpace}^{inj}$, some Segal space $X$ and using the cylinder object $X \times N(I[1])$, Lemma \ref{lemma_h1_v_cofib}'s induced cylinder object can be set to be $X \times N(I[1]) \times N(I[1])$, while $v$ is given by the evident inclusion of the `vertical' isomorphisms.

\begin{lemma} \label{lemma_h1_glob_cylobj}
    Suppose $C \sqcup C \xrightarrow{i} C \wedge I \xrightarrow{p} C$ is a very good cylinder object for a cofibrant object $C$ in a left proper model category $\mathscr{M}$. Let $h, r, C \wedge I \wedge I$ and $v$ be as in Lemma \ref{lemma_h1_v_cofib}.
    
    Then there is a cylinder object for $C \wedge I$ in $(C \sqcup C \downarrow \mathscr{M})$ of the form
    $$
        (C \wedge I) \sqcup_{C \sqcup C} (C \wedge I) \xrightarrow{b} C \wedge B \xrightarrow{z} C \wedge I
    $$
    where $C \wedge B$ is the pushout
    \[\begin{tikzcd}
    	{(C \wedge I) \sqcup (C \wedge I)} & {C \sqcup C} \\
    	{C \wedge I \wedge I} & {C \wedge B.}
    	\arrow["v"', from=1-1, to=2-1]
    	\arrow["{p \sqcup p}", from=1-1, to=1-2]
    	\arrow["k"', from=2-1, to=2-2]
    	\arrow["s", from=1-2, to=2-2]
    	\arrow["\ulcorner"{anchor=center, pos=0.125, rotate=180}, draw=none, from=2-2, to=1-1]
    \end{tikzcd}\]
\end{lemma}

\begin{proof}
    We need to define the two maps $b$ and $z$. It will then be sufficient to show that $z$ is a weak equivalence and that $zb$ is the appropriate projection map
    $$
        (C \wedge I) \sqcup_{C \sqcup C} (C \wedge I) \rightarrow C \wedge I
    $$
    induced by the pushout.

    The map $z$ is easiest, so we start here: it is simply given by the universal map from the pushout, induced by the maps $i : C \sqcup C \rightarrow C \wedge I$ and $r : C \wedge I \wedge I \rightarrow C \wedge I$. That these commute is due to the lifting problem $v$ is a solution for.

    As an intermediate step, note that $p \sqcup p$ is a weak equivalence. Indeed, the two maps $C \rightarrow C \sqcup C \xrightarrow{i} C \wedge I$ are weak equivalences by $2$-out-of-$3$. As $C$ is cofibrant, these are moreover trivial cofibrations. Then, note that the map
    $$
        C \sqcup C \hookrightarrow (C \sqcup C) \sqcup (C \sqcup C) \xrightarrow{i \sqcup i} (C \wedge I) \sqcup (C \wedge I) \xrightarrow{p \sqcup p} C \sqcup C
    $$
    is simply the identity. The composite of the first two maps is a coproduct of trivial cofibrations so is itself a trivial cofibration \cite[Prop. 7.2.5 (2)]{hirschhornModelCategoriesTheir2009}. Hence, $p \sqcup p$ is a weak equivalence by $2$-out-of-$3$.
    
    Note then that $k$ is a weak equivalence, as $\mathscr{M}$ is left proper, $v$ is a cofibration and $p \sqcup p$ is a weak equivalence. This implies by 2-out-of-3 with $r$ that $z$ is a weak equivalence, as needed.

    Now we turn to $b$. Consider the two morphisms $i_1, i_2 : C \sqcup C \rightarrow (C \wedge I) \sqcup (C \wedge I)$. It will suffice to show that $khi_1 = khi_2$, as this will induce a pushout map to $C \wedge B$. Extending the pushout diagram defining $C \wedge B$ reveals two larger commutative diagrams, for $j \in \{1, 2\}$, of the form
    \[\begin{tikzcd}
    	{C \sqcup C} & {(C \wedge I) \sqcup (C \wedge I)} & {C \sqcup C} \\
    	{(C \wedge I) \sqcup (C \wedge I)} & {C \wedge I \wedge I} & {C \wedge B}
    	\arrow["{i_j}"', from=1-1, to=2-1]
    	\arrow["h"', from=2-1, to=2-2]
    	\arrow["k"', from=2-2, to=2-3]
    	\arrow["{p \sqcup p}", from=1-2, to=1-3]
    	\arrow["s", from=1-3, to=2-3]
    	\arrow["v"', from=1-2, to=2-2]
    	\arrow["\ulcorner"{anchor=center, pos=0.125, rotate=180}, draw=none, from=2-3, to=1-2]
    	\arrow["{\alpha_j}", from=1-1, to=1-2]
    \end{tikzcd}\]
    where $\alpha_1$ and $\alpha_2$ are induced by the lifting problem defining $v$. Commutativity is given by this same lifting problem. It thus suffices to show that $s(p \sqcup p)\alpha_1 = s(p \sqcup p)\alpha_2$. However, $(p \sqcup p)\alpha_1$ and $(p \sqcup p)\alpha_2$ give the identity on $C \sqcup C$, so they are equal as needed, in turn defining $b$.

    We finally need to show that $b$ and $z$ satisfy the requirements of a cylinder object, namely that $z \circ b$ is the projection map from the pushout. This is a matter of a quick diagram chase and an application of the same property of $h$ and $r$. Indeed, if $q : (C \wedge I) \sqcup (C \wedge I) \rightarrow (C \wedge I) \sqcup_{C \sqcup C} (C \wedge I)$ is the natural map, then $zbq = zkh = rh$ by inspection, which is the required projection.
\end{proof}

Note that we did not prove this to be a very good cylinder object, as this is not needed for our purposes. We have not even proven it to be a good cylinder object. Regardless, though this cylinder object is a useful intermediate construction, we will never need to work with it directly:

\begin{lemma} \label{lemma_h1_cb_iff_glob}
    Suppose, in the situation of Lemma \ref{lemma_h1_glob_cylobj}, we have $D \in \mathscr{M}$ and two homotopies $H, K : C \wedge I \rightarrow D$, both between maps $f, g : C \rightarrow D$.
    
    Then there is a left homotopy from $H$ to $K$ in $(C \sqcup C \downarrow \mathscr{M})$ if and only if there is a left homotopy
    $$
    \Gamma : C \wedge I \wedge I \rightarrow D
    $$
    from $H$ to $K$ in $\mathscr{M}$ such that $\Gamma \circ v = (f \circ p) + (g \circ p)$.
\end{lemma}

\begin{proof}
    Suppose the map $\Gamma$ exists. The property it satisfies corresponds exactly to having a map to $D$ from the span which $C \wedge B$ is the pushout of. Thus, a left homotopy $\Gamma' : C \wedge B \rightarrow D$ is induced by universality of the pushout.

    Now, suppose $\kappa : C \wedge B' \rightarrow D$ is a left homotopy in $(C \sqcup C \downarrow \mathscr{M})$, for some good cylinder object
    $$
        (C \wedge I) \sqcup_{C \sqcup C} (C \wedge I) \xrightarrow{b'} C \wedge B' \xrightarrow{z'} C \wedge I.
    $$
    Some work is needed to obtain a left homotopy using the cylinder object $C \wedge B$ itself, as we have not proven it to be a good cylinder object. For now, take any factorization of $b$
    $$
        (C \wedge I) \sqcup_{C \sqcup C} (C \wedge I) \xrightarrow{\gamma} C \wedge \beta \xrightarrow{\omega} C \wedge B
    $$
    into a cofibration $\gamma$ followed by a trivial fibration $\omega$. $C \wedge \beta$ will serve just fine as a good cylinder object. Hence, we can assume $C \wedge B' = C \wedge \beta$.

    Now, take a lift
    \[\begin{tikzcd}
    	{(C \wedge I) \sqcup (C \wedge I)} & {C \sqcup C} & {(C \wedge I) \sqcup_{C \sqcup C} (C \wedge I)} & {C \wedge \beta} \\
    	{C \wedge I \wedge I} &&& {C \wedge B}
    	\arrow["k"', from=2-1, to=2-4]
    	\arrow["\omega", from=1-4, to=2-4]
    	\arrow["v"', from=1-1, to=2-1]
    	\arrow["\gamma", from=1-3, to=1-4]
    	\arrow["\kappa"{description}, dashed, from=2-1, to=1-4]
    	\arrow["{p \sqcup p}", from=1-1, to=1-2]
    	\arrow[from=1-2, to=1-3]
    \end{tikzcd}\]
    The map $\kappa$, together with the natural map $C \sqcup C \rightarrow C \wedge \beta$ given by $\gamma$, induces a map by universality of the pushout $e : C \wedge B \rightarrow C \wedge \beta$. Since $k$ and $\omega$ are weak equivalences, $\kappa$ is a weak equivalence. Thus, since $\kappa$ and $k$ are weak equivalences, $e$ must be a weak equivalence. Note that $e$ commutes with the maps $C \sqcup C \rightarrow C \wedge \beta$ and $s : C \sqcup C \rightarrow C \wedge B$, since $e$ is induced by the pushout's universal property. Precomposing with $e$ thus gives the left homotopy $C \wedge B \rightarrow D$ in $(C \sqcup C \downarrow \mathscr{M})$ desired, which when precomposed with $k$ gives $\Gamma$.
\end{proof}

In particular, we find the following:

\begin{corollary} \label{corr_h1_lefthtpys_2lhtpy}
    Consider the lifting problem induced by the span $C \xrightarrow{g} D \xleftarrow{p} B$ where $C$ is cofibrant and $p$ is a trivial fibration, with solutions $\delta, \epsilon : C \rightarrow B$ in $(\mathscr{M} \downarrow D)$, along with the data given in Lemma \ref{lemma_h1_glob_cylobj}. Any two left homotopies $H, K : C \wedge I \rightarrow B$ over $D$ between $\delta$ and $\epsilon$ induced by this lifting problem will be related by a left homotopy over $D$
    $$
        \Gamma : C \wedge I \wedge I \rightarrow B
    $$
    such that $\Gamma \circ v = (\delta \circ p) + (\epsilon \circ p)$.
\end{corollary}

\begin{proof}
    $H$ and $K$ are solutions of the same lifting problem. Thus, by Proposition \ref{prop_h1_2lift_htpy}, there is a left homotopy between them in $(C \sqcup C \downarrow \mathscr{M} \downarrow D) = (C \sqcup C \downarrow (\mathscr{M} \downarrow D))$, implying the result by Lemma \ref{lemma_h1_cb_iff_glob}.
\end{proof}

Note then that $\Gamma$, in our use case, is exactly the definition of a globular $2$-homotopy. Thus, we have the following result:

\begin{theorem} \label{thm_h1_liftprobs_htpic_2globhptic}
    Consider a lifting problem
    \[\begin{tikzcd}
    	& B \\
    	C & D
    	\arrow["g"', from=2-1, to=2-2]
    	\arrow["p", from=1-2, to=2-2]
    \end{tikzcd}\]
    in $\textbf{sSpace}^{inj}$, where $C, B$ and $D$ are Segal spaces and $p$ is a trivial fibration. Any two solutions will be related by a left homotopy in $(\textbf{SeSp}^{inj} \downarrow D)$ using the cylinder object $C \times N(I[1])$, while any two such left homotopies will be related by a globular $2$-homotopy.
\end{theorem}

\begin{corollary} \label{corr_h1_lift_mapssi_h1c_ident_htpies}
    Let $f, g : C \rightarrow B$ be two solutions to the lifting problem above. Let $H, K : C \times N(I[1]) \rightarrow B$ be two left homotopies over $D$ between them. Then the enrichment of $h_1$ in Kan complexes sends $H$ and $K$ to the same natural isomorphism $h_1(C) \times I[1] \rightarrow h_1(D)$.
\end{corollary}

As an aside, a consequence of our approach to left homotopies that now deserves mentioning is that $h_1$ will send a section of a trivial fibration to its inverse:

\begin{proposition} \label{prop_h1_section_inv}
    Consider a lifting problem of Segal spaces
    \[\begin{tikzcd}
    	& A \\
    	B & B
    	\arrow["id"', from=2-1, to=2-2]
    	\arrow["f", from=1-2, to=2-2]
    	\arrow["g", dashed, from=2-1, to=1-2]
    \end{tikzcd}\]
    with solution $g$ and a trivial fibration $f$. Then $h_1(g)$ is a categorical inverse to $h_1(f)$.
\end{proposition}

\begin{proof}
    We have that $h_1(f) \circ h_1(g) = 1_{h_1(B)}$. Moreover, by the 2-out-of-3 property of weak equivalences, we have that $g$ is a Dwyer-Kan equivalence and thus $h_1(g)$ is an equivalence.

    The proof will be complete if we produce a left homotopy $H : A \times N([I]) \rightarrow A$ from $g \circ f$ to $1_A$, as $h_1^c$ will send this to the necessary natural isomorphism. Such a left homotopy can be produced by noting that $g \circ f$ and $1_A$ are both solutions to the lifting problem
    \[\begin{tikzcd}
    	& A \\
    	A & B
    	\arrow["f"', from=2-1, to=2-2]
    	\arrow["f", from=1-2, to=2-2]
    	\arrow["{g \circ f}", curve={height=-12pt}, dashed, from=2-1, to=1-2]
    	\arrow["id"{description}, dashed, from=2-1, to=1-2]
    \end{tikzcd}\]
    completing the proof.
\end{proof}

Our results also make it possible to prove directly that composition of left homotopies and natural isomorphisms coincide, without needing the simplicial enrichment of $h_1$. Recall our convention of identifying $h_1(X \times N(I[1]))$ with $h_1(X) \times I[1]$ for some Segal space $X$.

\begin{proposition} \label{prop_h1_comp_htpies_is_natiso_comp}
    In the situation of Definition \ref{defn_h1_htpy_comp}, the vertical composite of natural isomorphisms $h_1(K)h_1(H) : h_1(C) \times I[1] \rightarrow h_1(B)$ is equal to $h_1(KH) : h_1(C) \times I[1] \rightarrow h_1(B)$ for any composite $KH : C \times N(I[1]) \rightarrow B$.
\end{proposition}

\begin{proof}
    Let $Q$ be as in Proposition \ref{defn_h1_htpy_comp}. The composite of $h_1(K)$ and $h_1(H)$ amounts to considering the induced functor
    $$
        F : h_1(C) \times I[2] \rightarrow h_1(B)
    $$
    which restricts to $h_1(H)$ and $h_1(K)$. There is precisely one such functor. However, $h_1(Q) : h_1(C) \times I[2] \cong h_1(C \times N(I[2])) \rightarrow h_1(B)$ also shares this property, so $h_1(Q) = F$. Hence, restriction yields $h_1(KH)$ for any $KH$ as needed.
\end{proof}
\section{Homotopy Unbiased Bicategories}

Let $X \in \textbf{SeSp}_2^{inj}$ be a $2$-fold Segal space. We wish to define an unbiased bicategory $h_2(X)$. As we will soon find, it will not be reasonable to expect such a bicategory to exist without additional data. This will take the form of a choice of horizontal compositions.

We set $\textbf{ob}(h_2(X)) := (X_{0, 0})_0$, for consistency with $h_1$. We also define
$$
    \textbf{Hom}_{h_2(X)}(x, y) := h_1(X(x, y)) = h_1((X_{1, \bullet})^{x, y}).
$$
We then take a choice of horizontal compositions $(\mu_n)_{n \geq 0}$ for $X$ as in Definition \ref{defn_cd_hor_comps}. Composition functors, for $n \geq 0$ and $x_0, \cdots, x_n \in \textbf{ob}(h_2(X))$, will then be the maps
$$
    \bullet^{x_0, \cdots, x_n} : \prod_{i = 1}^n \textbf{Hom}_{h_2(X)}(x_{i-1}, x_i) \cong h_1(\prod_{i = 1}^n X(x_{i-1}, x_i)) \xrightarrow{h_1(\circ^{x_0, \cdots, x_n})} \textbf{Hom}_{h_2(X)}(x_0, x_n).
$$
In general, for a simplicial composition diagram $K$ of arity $n$, write
$$
    \bullet^{x_0, \cdots, x_n}_K : \prod_{i = 1}^n \textbf{Hom}_{h_2(X)}(x_{i-1}, x_i) \cong h_1(\prod_{i = 1}^n X(x_{i-1}, x_i)) \xrightarrow{h_1(\circ^{x_0, \cdots, x_n}_K)} \textbf{Hom}_{h_2(X)}(x_0, x_n).
$$
Note the domain is $\ast \cong \{x_0\}$ if $n = 0$, as with $\circ^{x_0}$.

Our next order of business is to provide the associators and unitors. We will in fact provide much more general machinery here, which will be needed to work with functors as well. For associators, let $n \in \mathbb{Z}_{> 0}$ and $k_1, \cdots, k_n \in \mathbb{Z}_{\geq 0}$, with $r = \sum_{i = 1}^n k_i$. Consider any $n$-tuple of tuples of elements of $(X_{0, 0})_0$
$$
    Y := ((x_0^1, x_1^1, \cdots, x_{k_1}^1), \cdots, (x_0^n, \cdots, x_{k_n}^n))
$$
such that $x_{k_i}^i = x_0^{i+1}$ for every $i < n$. Let $(x_0, \cdots, x_r)$ be the flattened tuple after removing each $x_{k_i}^i$ for $i < n$.

\begin{proposition} \label{prop_h2_diag_is_nested_comp}
    Let $Y$ be as above. Let $Sp(r) \xrightarrow{\iota} K \xleftarrow{\tau} \Delta[1]$ be a simplicial composition diagram where $K = (K_1 \sqcup_{\Delta[0]} \cdots \sqcup_{\Delta[0]} K_n) \sqcup_{Sp(n)} K_0$, with each diagram $K_i$ of arity $k_i$ where $k_0 = n$. Then
    $$
        \bullet^{x_0, \cdots, x_r}_K = \bullet^{x^1_0, \cdots, x^n_0, x^n_{k_n}}_{K_0} \circ (\bullet^{x^1_0, \cdots, x^1_{k_1}}_{K_1} \times \cdots \times \bullet^{x^n_0, \cdots, x^n_{k_n}}_{K_n}).
    $$
\end{proposition}

\begin{proof}
    We have by Theorem \ref{thm_cd_diag_is_comp_of_circs} that the diagram
    \[\begin{tikzcd}
    	{\prod_{i = 1}^r \textbf{Hom}_{h_2(X)}(x_{i-1}, x_i)} &&& {h_1(\prod_{i = 1}^r X(x_{i-1}, x_i))} \\
    	{\prod_{i = 1}^n h_1(\prod_{j = 1}^{k_i} X(x^i_{j-1}, x^i_j))} \\
    	{\prod_{i = 1}^n \textbf{Hom}_{h_2(X)}(x^i_0, x^i_{k_i})} \\
    	{h_1(\prod_{i = 1}^n X(x^i_0, x^i_{k_i}))} &&& {\textbf{Hom}_{h_2(X)}(x_0, x_r)}
    	\arrow["{h_1(\circ_K^{x_0, \cdots, x_r})}", from=1-4, to=4-4]
    	\arrow["\cong", from=1-1, to=1-4]
    	\arrow["\cong"', from=1-1, to=2-1]
    	\arrow["{\prod_{i = 1}^n h_1(\circ_{K_i}^{x^i_0, \cdots, x^i_{k_i}})}"', from=2-1, to=3-1]
    	\arrow["\cong"', from=3-1, to=4-1]
    	\arrow["{h_1(\circ^{x^1_0, \cdots, x^n_0, x^n_{k_n}}_{K_0})}"', from=4-1, to=4-4]
    	\arrow["\cong"{description}, from=1-4, to=2-1]
    	\arrow["{h_1(\circ^{x^1_0, \cdots, x^1_{k_1}}_{K_1} \times \cdots \times \circ^{x^n_0, \cdots, x^n_{k_n}}_{K_n})}"{description}, from=1-4, to=4-1]
    \end{tikzcd}\]
    commutes. The result follows.
\end{proof}

Let $Sp(n) \xrightarrow{\iota_{K_i}} K_i \xleftarrow{\tau_{K_i}} \Delta[1]$, for $i \in \{1, 2\}$, be simplicial composition diagrams of arity $n$. Consider a map $f : K_1 \rightarrow K_2$ in $\textbf{SCD}_n$, as in Definition \ref{defn_cd_map_scds}. We have as in Proposition \ref{prop_cd_objfib_liftprob} the map $\mu_{K_2}' := f^\ast \circ \mu_{K_2}$. The two maps $\mu_{K_1}$ and $\mu_{K_2}'$ both solve the lifting problem
\[\begin{tikzcd}
	&& {X_{K_1}} & {X_1} \\
	\\
	{X_{Sp(n)}} && {X_{Sp(n)}}
	\arrow["{\tau_{K_1}^\ast}", from=1-3, to=1-4]
	\arrow[from=1-3, to=3-3]
	\arrow["{\mu_{K_2}'}"{description}, curve={height=-12pt}, dashed, from=3-1, to=1-3]
	\arrow["{\mu_{K_1}}"{description}, curve={height=12pt}, dashed, from=3-1, to=1-3]
	\arrow[from=3-1, to=3-3]
\end{tikzcd}\]
and thus admit an induced left homotopy $\beta_f : X_{Sp(n)} \times N(I[1]) \rightarrow X_{K_1}$ between them. This left homotopy is object-fibered and is constant on postcomposition with $\iota^\ast_{K_1} : X_{K_1} \rightarrow X_{Sp(n)}$. We will write $\beta_f$ for a left homotopy induced by a map $f$ in $\textbf{SCD}_n$ in this manner. We will in general speak of an `induced left homotopy' when a left homotopy is produced from two maps being solutions to the same lifting problem.

Recall that even though $\beta_f$ is not in general unique, any two induced left homotopies $\beta_f$ and $\beta_f'$ by a map $f$ will have a globular $2$-homotopy $\alpha : X_{Sp(n)} \times N(I[1]) \times N(I[1]) \rightarrow X_{K_1}$ between them, by Theorem \ref{thm_h1_liftprobs_htpic_2globhptic}. It is this property that asserts functoriality of the following construction:

\begin{definition}
    Let $X$ be a $2$-fold Segal space with a choice of horizontal compositions $(\mu_n)_{n \geq 0}$. Let $x, y \in (X_{0, 0})_0$ and set $n \geq 0$. Define the functor
    $$
        \omega_{(X, (\mu_n)_{n \geq 0})}^{x, y} : \textbf{SCD}_n \rightarrow \textbf{Fun}\Big( h_1 (X_{Sp(n)}^{x, y}), h_1(X(x, y)) \Big)
    $$
    to send $K \mapsto h_1((\tau^*_K)^{x, y} \circ \mu^{x, y}_K)$ and $f : K_1 \rightarrow K_2$ to the natural isomorphism
    $$
        h_1((\tau^*_{K_1})^{x, y} \circ \beta_f^{x, y}) : h_1(X_{Sp(n)}^{x, y}) \times I[1] \xrightarrow{h_1(\beta_f^{x, y})} h_1(X_{K_1}^{x, y}) \xrightarrow{h_1((\tau^\ast_{K_1})^{x, y})} h_1(X(x, y))
    $$
    for an induced left homotopy $\beta_f : X_{Sp(n)} \times N(I[1]) \rightarrow X_{K_1}$ from $\mu_{K_1}$ to $\mu_{K_2}'$.
\end{definition}

\begin{proposition}
    $\omega_{(X, (\mu_n)_{n \geq 0})}^{x, y}$ is a functor.
\end{proposition}

\begin{proof}
    By the above discussion, $\omega_{(X, (\mu_n)_{n \geq 0})}^{x, y}(f)$ is well-defined for any morphism $f$ in $\textbf{SCD}_n$.
    
    Suppose $K_1 \xrightarrow{f} K_2 \xrightarrow{g} K_3$ is a diagram in $\textbf{SCD}_n$ for simplicial composition diagrams $\Delta[n] \xrightarrow{\iota_i} K_i \xleftarrow{\tau_i} \Delta[1]$. We wish to show that
    $$
        \omega_{(X, (\mu_n)_{n \geq 0})}^{x, y}(g) \omega_{(X, (\mu_n)_{n \geq 0})}^{x, y}(f) = \omega_{(X, (\mu_n)_{n \geq 0})}^{x, y}(g \circ f).
    $$
    These three natural isomorphisms can be rewritten as

    \begin{align*}
        \omega_{(X, (\mu_n)_{n \geq 0})}^{x, y}(f) &= (\tau_1^*)^{x, y} \circ \beta_f^{x, y} \\
        \omega_{(X, (\mu_n)_{n \geq 0})}^{x, y}(g) &= (\tau_1^*)^{x, y} \circ (f^* \circ \beta_g)^{x, y} \\
        \omega_{(X, (\mu_n)_{n \geq 0})}^{x, y}(g \circ f) &= (\tau_1^*)^{x, y} \circ \beta_{g \circ f}^{x, y} \\
    \end{align*}
    for left homotopies $\beta_f^{x, y}$, $(f^* \circ \beta_g)^{x, y}$ and $\beta_{g \circ f}^{x, y}$ induced between $\mu_{K_1}^{x, y}$ and $(f^* \circ \mu_{K_2})^{x, y}$, $(f^* \circ \mu_{K_2})^{x, y}$ and $(f^* \circ g^* \circ \mu_{K_3})^{x, y}$ and $\mu_{K_1}^{x, y}$ and $(f^* \circ g^* \circ \mu_{K_3})^{x, y}$ respectively, all induced by the maps $\mu_{K_1}, (f^\ast \circ \mu_{K_2})$ and $(f^\ast \circ g^\ast \circ \mu_{K_3})$ being solutions to the lifting problem of the form
    \[\begin{tikzcd}
    	&&& {X_{K_1}} \\
    	\\
    	\\
    	{X_{Sp(n)}} &&& {X_{Sp(n)}}
    	\arrow["{\iota_{K_1}^\ast}", from=1-4, to=4-4]
    	\arrow["{(f^\ast \circ g^\ast \circ \mu_{K_3})}"{description}, curve={height=-24pt}, dashed, from=4-1, to=1-4]
    	\arrow["{\mu_{K_1}}"{description}, curve={height=24pt}, dashed, from=4-1, to=1-4]
    	\arrow["{(f^\ast \circ \mu_{K_2})}"{description}, dashed, from=4-1, to=1-4]
    	\arrow[from=4-1, to=4-4]
    \end{tikzcd}\]
    
    We may therefore take a composite $C = ((f^* \circ \beta_g)^{x, y})(\beta_f^{x, y})$ as in Definition \ref{defn_h1_htpy_comp}, which is a left homotopy from $\mu_{K_1}^{x, y}$ to $(f^* \circ g^* \circ \mu_{K_3})^{x, y}$. By Proposition \ref{prop_h1_comp_htpies_is_natiso_comp}, we have that $h_1(C)$ is the composite of $h_1((f^* \circ \beta_g)^{x, y})$ and $h_1(\beta_f^{x, y})$.

    Note however that $C$ and $\beta_{g \circ f}^{x, y}$ are both left homotopies induced between the same solutions of a lifting problem. Hence, by Corollary \ref{corr_h1_lift_mapssi_h1c_ident_htpies}, we have that $h_1(\beta_{g \circ f}^{x, y}) = h_1(C)$ is the composite of the natural isomorphisms $h_1((f^* \circ \beta_g)^{x, y})$ and $h_1(\beta_f^{x, y})$, as needed.

    Now, suppose $K$ is a simplicial composition diagram of arity $n$. We wish to show that
    $$
        \omega_{(X, (\mu_n)_{n \geq 0})}^{x, y}(1_K) = 1_{h_1((\tau_K^*)^{x, y} \circ \mu_K^{x, y})}.
    $$
    The left homotopy $\beta_{1_K}$ can be set to be a constant left homotopy, as its source and target are equal. Since any choice of $\beta_{1_K}$ will yield the same natural isomorphism, the result is evident.
\end{proof}

\begin{definition}
    Let $X$ be a $2$-fold Segal space with a choice of horizontal compositions $(\mu_n)_{n \geq 0}$. Let $x_0, \cdots, x_n \in (X_{0, 0})_0$. Then define the map 
    $$
    \phi^{x_0, \cdots, x_n}_{(X, (\mu_n)_{n \geq 0})} : \textbf{Fun} \Big(h_1(X_{Sp(n)}^{x_0, x_n}), h_1(X(x_0, x_n))\Big) \rightarrow \textbf{Fun} \Big(\prod_{i = 1}^n h_1(X(x_{i-1}, x_i)), h_1(X(x_0, x_n)) \Big)
    $$
    to be the functor given by precomposition with the inclusion
    $$
        F_{x_0, \cdots, x_n} : \prod_{i = 1}^n h_1(X(x_{i-1}, x_i)) \cong h_1(\prod_{i = 1}^n X(x_{i-1}, x_i)) \hookrightarrow h_1(X_{Sp(n)}^{x_0, x_n}).
    $$
    Write $\Phi^{x_0, \cdots, x_n}_{(X, (\mu_n)_{n \geq 0})} := \phi^{x_0, \cdots, x_n}_{(X, (\mu_n)_{n \geq 0})} \circ \omega^{x_0, x_n}_{(X, (\mu_n)_{n \geq 0})}$.
\end{definition}

For our given $Y$, let $K = (\Delta[k_1] \sqcup_{\Delta[0]} \cdots \sqcup_{\Delta[0]} \Delta[k_n]) \sqcup_{Sp(n)} \Delta[n]$ and take the map $f_{k_1, \cdots, k_n} : K \rightarrow \Delta[r]$ in $\textbf{SCD}_r$ to be the map restricting to the morphisms
\begin{align*}
    \langle \sum_{j = 1}^{i-1} k_j, \cdots, \sum_{j = 1}^{i-1} k_j + k_i \rangle &: \Delta[k_i] \rightarrow \Delta[r] \\
    \langle 0, k_1, k_1 + k_2, \cdots, \sum_{j = 1}^n k_j \rangle &: \Delta[n] \rightarrow \Delta[r].
\end{align*}
Our associators for $h_2(X)$ will be the natural isomorphisms
$$
     \Phi^{x_0, \cdots, x_r}_{(X, (\mu_n)_{n \geq 0})}(f_{k_1, \cdots, k_n}) : \bullet^{x^1_0, \cdots, x^n_0, x^n_{k_n}} \circ (\bullet^{x^1_0, \cdots, x^1_{k_1}} \times \cdots \times \bullet^{x^n_0, \cdots, x^n_{k_n}}) \Rightarrow \bullet^{x_0, \cdots, x_r}.
$$

Unitors for us will be trivial; we choose to set them to be identities. We are thus now ready to declare the definition of our unbiased bicategory in full:

\begin{definition} \label{defn_h2_h2}
    Let $X$ be a $2$-fold Segal space. Let $(\mu_n)_{n \geq 0}$ be a choice of horizontal compositions.
    
    Then the \emph{unbiased homotopy bicategory $h_2(X, (\mu_n)_{n \geq 0})$ of $X$}, written as $h_2(X)$ if the $\mu_n$ are known, is the unbiased bicategory defined such that:

    \begin{enumerate}
        \item $\textbf{ob}(h_2(X)) := (X_{0, 0})_0$;
        
        \item $\textbf{Hom}_{h_2(X)}(x, y) := h_1(X(x, y))$ for all $x, y \in (X_{0, 0})_0$;
        
        \item For each $n > 0$ and $x_0, \cdots, x_n \in (X_{0, 0})_0$, composition is given by the functor $\bullet^{x_0, \cdots, x_n}$;

        \item For each $x \in (X_{0, 0})_0$, identities are given by the functor $\bullet^x$;

        \item Let $n > 0$ and $k_1, \cdots, k_n \in \mathbb{Z}_{\geq 0}$ with $r = \sum_i k_i$. Consider any $n$-tuple of tuples of elements of $(X_{0, 0})_0$
        $$
            Y := ((x_0^1, x_1^1, \cdots, x_{k_1}^1), \cdots, (x_0^n, \cdots, x_{k_n}^n))
        $$
        such that $x_{k_i}^i = x_0^{i+1}$ for every $i < n$. Let $(x_0, \cdots, x_r)$ be the flattened tuple after removing each $x_{k_i}^i$ for $i < n$. Then the natural isomorphism $\gamma_Y$ is defined to be
        $$
            \Phi^{x_0, \cdots, x_r}_{(X, (\mu_n)_{n \geq 0})}(f_{k_1, \cdots, k_n});
        $$
        
        \item For each $x, y \in (X_{0, 0})_0$, the natural isomorphism
        $$
            \iota_{x, y} : 1_{\textbf{Hom}_{h_2(X)}(x, x)} \Rightarrow 1_{\textbf{Hom}_{h_2(X)}(x, x)} = \bullet^{x, x}
        $$
        is an identity.
    \end{enumerate}
\end{definition}

\subsection{Coherence Conditions}

We now come to proving the coherence conditions on our proposed data for an unbiased bicategory. We will go further than this; our final proof of coherence will be to establish an `operadic' action of $\textbf{SCD}_\bullet$ via $\Phi^{(-)}_{(X, (\mu_n)_{n \geq 0})}$ on our hom-categories and functors between them. This construction will rely on identifying and composing left homotopies in a suitably `operadic' manner. A technical lemma is first required:

\begin{lemma} \label{lemma_h2_seq_of_lifts}
    Consider a diagram in a model category $\mathscr{M}$
    \[\begin{tikzcd}
    	&& E & F & Z \\
    	A & B & G & H \\
    	C & D
    	\arrow["u", from=2-2, to=2-3]
    	\arrow["h", from=2-1, to=2-2]
    	\arrow["f"', from=2-1, to=3-1]
    	\arrow["g", from=2-2, to=3-2]
    	\arrow["k"', from=3-1, to=3-2]
    	\arrow["q"', from=1-3, to=2-3]
    	\arrow["t"', from=2-3, to=2-4]
    	\arrow["r", from=1-4, to=2-4]
    	\arrow["s", from=1-3, to=1-4]
    	\arrow["\delta"{description}, dashed, from=3-1, to=2-2]
    	\arrow["\epsilon"{description}, dashed, from=2-3, to=1-4]
    	\arrow["w", from=1-4, to=1-5]
    \end{tikzcd}\]
    where the two squares define lifting problems with solutions $\delta$ and $\epsilon$, such that $g$ and $r$ are trivial fibrations. Then the square in the diagram
    \[\begin{tikzcd}
    	A && {B \times_H F} & F & Z \\
    	&& B \\
    	C && D
    	\arrow["f"', from=1-1, to=3-1]
    	\arrow["k"', from=3-1, to=3-3]
    	\arrow["{h \times_{tuh} \epsilon uh}", from=1-1, to=1-3]
    	\arrow["{\delta \times_{t u \delta} \epsilon u \delta}"{description}, dashed, from=3-1, to=1-3]
    	\arrow[two heads, from=1-3, to=2-3]
    	\arrow["g", from=2-3, to=3-3]
    	\arrow["w", from=1-4, to=1-5]
    	\arrow[two heads, from=1-3, to=1-4]
    \end{tikzcd}\]
    is also a lifting problem in $\mathscr{M}$ with the given solution. Moreover, the map $C \rightarrow B \times_H F \rightarrow F \rightarrow Z$ in this diagram is equal to $w \epsilon u \delta$.
\end{lemma}

\begin{proof}
    It is clear that the rightmost vertical map will be a trivial fibration if $g$ and $r$ are as such. Thus, the diagram is a valid lifting problem. The last equality is easily checked by the definition of the chosen lift.
\end{proof}

The following lemma is what allows us to compose left homotopies horizontally in an `operadic' manner, which forms the foundations of our proof of coherence conditions:

\begin{lemma} \label{lemma_h2_operad_of_htpies}
    Let $X$ be a Segal space and set $n \geq 1$ and $k_1, \cdots, k_n \geq 0$. Take Segal spaces $A^j_i$ for $1 \leq i \leq n$ and $1 \leq j \leq k_i$, Segal spaces $B_i$ and $B_i'$ for $1 \leq i \leq n$ and some Segal space $C$, all fibered over $X \times X$. Take between these a set of maps in $SeSp^{inj}$ for $1 \leq i \leq n$ of the form
    \begin{align*}
        f_i, g_i &: \prod_X^{1 \leq j \leq k_i} A^j_i \rightarrow B_i \\
        t_i &: B_i \rightarrow B_i' \\
        \phi, \psi &: \prod_X^{1 \leq i \leq n} B_i' \rightarrow C.
    \end{align*}
    Consider a finite collection of left homotopies in $SeSp^{inj}$ emerging from lifting problems in commutative diagrams fibered over $X \times X$
    \[\begin{tikzcd}
    	{\prod^{1 \leq j \leq k_i}_X A^j_i \sqcup \prod^{1 \leq j \leq k_i}_X A^j_i} && {B_i} & {B_i'} \\
    	{\prod^{1 \leq j \leq k_i}_X A^j_i \times N(I[1])} && {\prod^{1 \leq j \leq k_i}_X A^j_i}
    	\arrow["{H_i}"{description}, dashed, from=2-1, to=1-3]
    	\arrow["{f_i \sqcup g_i}", from=1-1, to=1-3]
    	\arrow[from=1-1, to=2-1]
    	\arrow[from=2-1, to=2-3]
    	\arrow[from=1-3, to=2-3]
    	\arrow["{t_i}", from=1-3, to=1-4]
    \end{tikzcd}\]
    for $1 \leq i \leq n$ where the right vertical map is a trivial fibration and
    \[\begin{tikzcd}
    	{\prod_X^{1 \leq i \leq n} B_i' \sqcup \prod_X^{1 \leq i \leq n} B_i'} && C \\
    	{\prod_X^{1 \leq i \leq n} B_i' \times N(I[1])} && {\prod_X^{1 \leq i \leq n} B_i'}
    	\arrow["{\psi \sqcup \phi}", from=1-1, to=1-3]
    	\arrow[from=1-3, to=2-3]
    	\arrow[from=1-1, to=2-1]
    	\arrow[from=2-1, to=2-3]
    	\arrow["K"{description}, dashed, from=2-1, to=1-3]
    \end{tikzcd}\]
    where again the right vertical map is a trivial fibration. For all $k \geq 1$, set $D_k : \prod^{i, j}_X A^j_i \times N(I[1]) \rightarrow \prod^{i, j}_X A^j_i \times N(I[1])^k$ to be the diagonal map.
    
    Then the left homotopy
    $$
        Q := K \circ \Big( \Big( (\prod^{1 \leq i \leq n}_X t_i \circ H_i) \circ D_n \Big) \times 1_{N(I[1])} \Big) \circ D_2
    $$
    is equal to $p \circ F$, where $p$ and $F$ are fibered naturally over $X \times X$ and fit into a lifting problem in $SeSp^{inj}$
    \[\begin{tikzcd}
    	{\prod^{i, j}_X A^j_i \sqcup \prod_X^{i, j} A^j_i} & {\prod^i_X B_i \times_{\prod^i_X B_i'} C} & C \\
    	& {\prod^i_X B_i} \\
    	{\prod^{i, j}_X A^j_i \times N(I[1])} & {\prod^{i, j}_X A^j_i}
    	\arrow["\iota", from=1-1, to=1-2]
    	\arrow[from=1-2, to=2-2]
    	\arrow[from=2-2, to=3-2]
    	\arrow[from=3-1, to=3-2]
    	\arrow["F"{description}, dashed, from=3-1, to=1-2]
    	\arrow["p", two heads, from=1-2, to=1-3]
    	\arrow[from=1-1, to=3-1]
    \end{tikzcd}\]
    such that $p$ is the pullback projection and $\iota$ is induced by the maps
    $$
        \prod^{i, j}_X A_i^j \sqcup \prod^{i,j}_X A_i^j \xrightarrow{\prod_X^i f_i \sqcup \prod_X^i g_i} \prod^i_X B_i
    $$
    and
    $$
        \prod_X^{i, j} A_i^j \sqcup \prod_X^{i, j} A_i^j \xrightarrow{\prod_X^i f_i + \prod_X^i g_i} \prod_X^i B_i \sqcup \prod_X^i B_i \xrightarrow{\prod_X^i t_i + \prod_X^i t_i} \prod_X^i B_i' \sqcup \prod_X^i B_i' \xrightarrow{\psi \sqcup \phi} C.
    $$
\end{lemma}

\begin{proof}
    Define
    $$
        q : \prod_X^{i, j} A_i^j \sqcup \prod_X^{i, j} A_i^j \hookrightarrow \prod_X^{1 \leq i \leq n} (\prod_X^j A_i^j \sqcup \prod_X^j A_i^j) \cong \bigsqcup_{k = 1}^{2^n} \prod_X^{i, j} A_i^j
    $$
    to be the pullback map induced by the $n$ maps $\rho_i \sqcup \rho_i$ for $1 \leq i \leq n$, where $\rho_i : \prod_X^{i, j} A_i^j \rightarrow \prod_X^j A_i^j$ is the projection to the $i^{th}$ coordinate. This is implicitly an inclusion.

    We have that $(\prod_X^i H_i) \circ D_n$ is a solution of the lifting problem defined by the outermost square of the diagram
    \[\begin{tikzcd}
    	{\prod_X^{i, j} A_i^j \sqcup \prod_X^{i, j} A_i^j} \\
    	{\prod_X^{1 \leq i \leq n} (\prod_X^j A_i^j \sqcup \prod_X^j A_i^j)} &&& {\prod_X^{1 \leq i \leq n} B_i} \\
    	\\
    	{\prod_{X}^{1 \leq i \leq n} (\prod_X^j A_i^j \times N(I[1]))} &&& {\prod_X^{i, j} A_i^j} \\
    	{\prod_X^{i, j} A_i^j \times N(I[1])}
    	\arrow["{\prod_X^{1 \leq i \leq n} H_i}"{description}, dashed, from=4-1, to=2-4]
    	\arrow[from=2-4, to=4-4]
    	\arrow[hook', from=2-1, to=4-1]
    	\arrow["{\prod_{X}^{1 \leq i \leq n} (f_i \sqcup g_i)}", from=2-1, to=2-4]
    	\arrow[from=4-1, to=4-4]
    	\arrow[from=5-1, to=4-4]
    	\arrow[two heads, from=4-1, to=5-1]
    	\arrow["q"', hook', from=1-1, to=2-1]
    	\arrow["{\prod_X^{1 \leq i \leq n} f_i \sqcup \prod_X^{1 \leq i \leq n} g_i}", curve={height=-12pt}, from=1-1, to=2-4]
    \end{tikzcd}\]
    which demonstrates $(\prod_X^i H_i) \circ D_n$ is an induced left homotopy from $\prod_X^i f_i$ to $\prod_X^i g_i$.
    
    We now have a commutative diagram of the form
    \[\begin{tikzcd}
    	&& {\prod_X^{1 \leq i \leq n} B_i' \sqcup \prod_X^{1 \leq i \leq n} B_i'} && C \\
    	&& {\prod_X^{1 \leq i \leq n} B_i' \times N(I[1])} && {\prod_X^{1 \leq i \leq n} B_i'} \\
    	{(\prod_X^{i, j} A_i^j \sqcup \prod_X^{i, j} A_i^j) \times N(I[1])} && {\prod_X^{1 \leq i \leq n} B_i \times N(I[1])} \\
    	{\prod_X^{i, j} A_i^j \times N(I[1])^{n+1}} && {\prod_X^{i, j} A_i^j \times N(I[1])} \\
    	{\prod_X^{i, j} A_i^j \times N(I[1])^2}
    	\arrow["{\psi \sqcup \phi}", from=1-3, to=1-5]
    	\arrow[from=1-5, to=2-5]
    	\arrow[from=1-3, to=2-3]
    	\arrow[two heads, from=2-3, to=2-5]
    	\arrow["K"{description}, dashed, from=2-3, to=1-5]
    	\arrow["{\prod_X^i t_i \times 1_{N(I[1])}}"', from=3-3, to=2-3]
    	\arrow[from=3-1, to=4-1]
    	\arrow["{(\prod_X^i f_i \sqcup \prod_X^i g_i) \times 1_{N(I[1])}}", shift left=2, from=3-1, to=3-3]
    	\arrow["{(\prod_X^i H_i) \times 1_{N(I[1])}}"{description}, dashed, from=4-1, to=3-3]
    	\arrow[two heads, from=4-1, to=5-1]
    	\arrow[from=3-3, to=4-3]
    	\arrow[two heads, from=4-1, to=4-3]
    	\arrow[two heads, from=5-1, to=4-3]
    \end{tikzcd}\]
    which, by Lemma \ref{lemma_h2_seq_of_lifts}, produces a single lifting problem with solution $F'$ within the diagram
    \[\begin{tikzcd}
    	{\prod^{i, j}_X A^j_i \sqcup \prod_X^{i, j} A^j_i} && C \\
    	{(\prod^{i, j}_X A^j_i \sqcup \prod_X^{i, j} A^j_i) \times N(I[1])} & {\Big( \prod^i_X B_i \times N(I[1]) \Big) \times_{\prod^i_X B_i'} C} & {\prod^i_X B_i \times_{\prod^i_X B_i'} C} \\
    	{\prod^{i, j}_X A^j_i \times N(I[1])^2} & {\prod^{i, j}_X A^j_i \times N(I[1])} \\
    	{\prod^{i, j}_X A^j_i \times N(I[1])} && {\prod^{i, j}_X A^j_i}
    	\arrow["{\iota'}", from=1-1, to=2-2]
    	\arrow[from=2-2, to=3-2]
    	\arrow[two heads, from=3-2, to=4-3]
    	\arrow["{F'}"{description}, dashed, from=3-1, to=2-2]
    	\arrow["{(id, \langle 0 \rangle) + (id, \langle 1 \rangle)}"', from=1-1, to=2-1]
    	\arrow[from=2-1, to=3-1]
    	\arrow[from=2-1, to=2-2]
    	\arrow[two heads, from=3-1, to=4-1]
    	\arrow[two heads, from=3-1, to=3-2]
    	\arrow[two heads, from=4-1, to=4-3]
    	\arrow["R", two heads, from=2-2, to=2-3]
    	\arrow["p", from=2-3, to=1-3]
    	\arrow[from=2-3, to=4-3]
    \end{tikzcd}\]
    where the maps $\langle 0 \rangle, \langle 1 \rangle : \ast \rightarrow N(I[1])$ identify the two objects of $I[1]$ and $R$ is induced by the projection forgetting $N(I[1])$. The outermost square of this diagram is another lifting problem, solved by $F := R \circ F' \circ D_2$. It is clear then that $p \circ F = Q$ by design. Moreover, since the diagram
    \[\begin{tikzcd}
    	{\prod^{i, j}_X A_i^j \sqcup \prod^{i,j}_X A_i^j} && {\prod_X^i B_i \sqcup \prod_X^i B_i} \\
    	{\Big( \prod^{i, j}_X A_i^j \sqcup \prod^{i,j}_X A_i^j \Big) \times N(I[1])} && {\prod_X^i B_i' \sqcup \prod_X^i B_i'} \\
    	{\Big( \prod^i_X B_i \Big) \times N(I[1])} \\
    	{\Big( \prod_X^i B_i' \Big) \times N(I[1])} && C
    	\arrow["{\prod_X^i f_i + \prod_X^i g_i}", from=1-1, to=1-3]
    	\arrow["{\prod_X^i t_i + \prod_X^i t_i}", from=1-3, to=2-3]
    	\arrow["{\psi \sqcup \phi}", from=2-3, to=4-3]
    	\arrow["K"', from=4-1, to=4-3]
    	\arrow["{(\prod_X^i t_i) \times 1_{N(I[1])}}"', from=3-1, to=4-1]
    	\arrow["{(id, \langle 0 \rangle) + (id, \langle 1 \rangle)}"', from=1-1, to=2-1]
    	\arrow["{(\prod_X^i f_i \sqcup \prod_X^i g_i) \times 1_{N(I[1])}}"', from=2-1, to=3-1]
    \end{tikzcd}\]
    commutes, setting $\iota := R \circ \iota'$ completes the proof.
\end{proof}

A crucial and nontrivial lemma is now needed, which carries a remarkably operadic flavor to it. We will not explore any subtle appearances of operads in this paper; such concerns will be left to future work. One could argue this lemma is the reason we have coherence:

\begin{lemma} \label{lemma_h2_psi_operadic}
    Let $X$ and the list $Y$ be as above, with some choice of horizontal compositions $(\mu_n)_{n \geq 0}$. Consider the functor $(-) \cdot (-)$ of the form
    \[\begin{tikzcd}
    	{\textbf{Fun} \Big(\prod\limits_{i = 1}^n h_1(X(x^i_0, x^i_{k_i})), h_1(X(x_0, x_r)) \Big) \times \textbf{Fun} \Big(\prod\limits_{i = 1}^r h_1(X(x_{i-1}, x_i)), \prod\limits_{i = 1}^n h_1(X(x^i_0, x^i_{k_i})) \Big)} \\
    	{\textbf{Fun} \Big(\prod\limits_{i = 1}^r h_1(X(x_{i-1}, x_i)), h_1(X(x_0, x_r)) \Big)}
    	\arrow["{(-) \cdot (-)}"', from=1-1, to=2-1]
    \end{tikzcd}\]
    defined by horizontal composition. Then
    $$
        \Phi^{x^1_0, \cdots, x^n_0, x^n_{k_n}}_{(X, (\mu_n)_{n \geq 0})} \cdot \Big( \prod_{i = 1}^n \Phi^{x^i_0, \cdots, x^i_{k_i}}_{(X, (\mu_n)_{n \geq 0})} \Big) = \Phi^{x_0, \cdots, x_r}_{(X, (\mu_n)_{n \geq 0})} \circ \mathcal{G}^n_{k_1, \cdots, k_n}.
    $$
\end{lemma}

\begin{proof}
    That this equivalence holds on objects is by construction and Proposition \ref{prop_h2_diag_is_nested_comp}. For morphisms, consider for $0 \leq i \leq n$ a collection of maps $f_i : K_i \rightarrow K_i'$ in $\textbf{SCD}_{k_i}$, where $Sp(k_i) \xrightarrow{\iota_i} K_i \xleftarrow{\tau_i} \Delta[1]$ and $Sp(k_i) \xrightarrow{\iota_i'} K_i' \xleftarrow{\tau_i'} \Delta[1]$ are simplicial composition diagrams and $k_0 = n$. Then we have an induced map $f : K \rightarrow K'$, where
    \begin{align*}
        K &:= (K_1 \sqcup_{\Delta[0]} \cdots \sqcup_{\Delta[0]} K_n) \sqcup_{Sp(n)} K_0 \\
        K' &:= (K_1' \sqcup_{\Delta[0]} \cdots \sqcup_{\Delta[0]} K_n') \sqcup_{Sp(n)} K_0'
    \end{align*}
    with maps $Sp(r) \xrightarrow{\iota} K \xleftarrow{\tau} \Delta[1]$ and $Sp(r) \xrightarrow{\iota'} K' \xleftarrow{\tau'} \Delta[1]$. We wish to show that
    $$
        \Phi^{x^1_0, \cdots, x^n_0, x^n_{k_n}}_{(X, (\mu_n)_{n \geq 0})}(f_0) \circ \prod_{i = 1}^n \Phi^{x^i_0, \cdots, x^i_{k_i}}_{(X, (\mu_n)_{n \geq 0})}(f_i) = \Phi^{x_0, \cdots, x_r}_{(X, (\mu_n)_{n \geq 0})}(f).
    $$
    The left hand side of this equality can be expanded out as
    $$
        h_1((\tau_0^*)^{x^1_0, x^n_{k_n}}) \circ h_1(\beta_{f_0}^{x^1_0, x^n_{k_n}}) \circ F_{x^1_0, \cdots, x^n_0, x^n_{k_n}} \circ \Big( \prod_{i = 1}^n h_1((\tau_i^*)^{x^i_0, x^i_{k_i}}) \circ h_1(\beta_{f_i}^{x^i_0, x^i_{k_i}}) \circ F_{x^i_0, \cdots, x^i_{k_i}} \Big)
    $$
    where each $\beta_{f_i}$ is a left homotopy induced between $\mu_{K_i}$ and $\mu_{K_i'}'$. The right hand side can then be expanded out as
    $$
        h_1((\tau^*)^{x_0, x_r}) \circ h_1(\beta_f^{x_0, x_r}) \circ F_{x_0, \cdots, x_r}.
    $$
    If we can show that the left hand side is also of the form $h_1((\tau^*)^{x^1_0, x^n_{k_n}}) \circ h_1(H^{x_0, x_r}) \circ F_{x_0, \cdots, x_r}$ for some left homotopy $H$ induced between $\mu_K$ and $\mu_{K'}'$, the proof will be complete, as these will then have to be identified by Corollary \ref{corr_h1_lift_mapssi_h1c_ident_htpies}.

    We should note that, if we are to treat these natural isomorphisms as functors of the form $\mathscr{C} \times I[1] \rightarrow \mathscr{D}$, we will have to rewrite the left-hand side as
    $$
        h_1((\tau^*_0)^{x^1_0, x^n_{k_n}} \circ \beta_{f_0}^{x^1_0, x^n_{k_n}}) \circ F_{x^1_0, \cdots, x^n_0, x^n_{k_n}}' \circ \Bigg( \Big( \prod_{i = 1}^n h_1((\tau_i^*)^{x^i_0, x^i_{k_i}} \circ \beta_{f_i}^{x^i_0, x^i_{k_i}}) \circ F_{x^i_0, \cdots, x^i_{k_i}}' \circ D_n' \Big) \times 1_{I[1]} \Bigg) \circ D_2'
    $$
    where we now have the inclusion
    $$
        F_{a_0, \cdots, a_k}' : \prod_{i = 1}^k h_1(X(a_{i-1}, a_i)) \times I[1] \hookrightarrow h_1(X_{Sp(k)}^{a_0, a_k} \times N(I[1]))
    $$
    and $D_k' : \prod_{i = 1}^r \textbf{Hom}_{h_2(X)}(x_{i - 1}, x_i) \times I[1] \rightarrow \prod_{i = 1}^r \textbf{Hom}_{h_2(X)}(x_{i - 1}, x_i) \times I[1]^k$ the diagonal map for all $k > 0$.

    Using naturality of the isomorphism $h_1(-) \times h_1(-) \cong h_1(- \times -)$, we can convert this into a map $h_1((\tau^*_0)^{x^1_0, x^n_{k_n}} \circ P)$, where
    $$
        P := \beta_{f_0}^{x^1_0, x^n_{k_n}} \circ \mathcal{F}_{x^1_0, \cdots, x^n_0, x^n_{k_n}} \circ \Bigg( \Big( \prod_{i = 1}^n (\tau_i^*)^{x^i_0, x^i_{k_i}} \circ \beta_{f_i}^{x^i_0, x^i_{k_i}} \circ \mathcal{F}_{x^i_0, \cdots, x^i_{k_i}} \circ D_n^{x_0, \cdots, x_r} \Big) \times 1_{N(I[1])} \Bigg) \circ D_2^{x_0, \cdots, x_r}
    $$
    with the inclusions
    $$
        \mathcal{F}_{a_0, \cdots, a_k} : \prod_{i = 1}^k X(a_{i-1}, a_i) \rightarrow X_{Sp(k)}^{a_0, a_k}
    $$
    and the diagonal maps $D_k : X_{Sp(r)} \times N(I[1]) \rightarrow X_{Sp(r)} \times N([1])^k$ on $N(I[1])$.

    Note that the maps $\mathcal{F}_{a_0, \cdots, a_k}$ form part of a more general natural transformation $(-)^{a_0, \cdots, a_k} \Rightarrow (-)^{a_0, a_k}$. Using this naturality, we can prove that $P$ is in fact equal to a map $Q^{x_0, x_r} \circ \mathcal{F}_{x_0, \cdots, x_r}$ such that
    $$
        Q := \beta_{f_0} \circ \Bigg( \Big( ( \prod_{X_0}^{1 \leq i \leq n} \tau^*_i \circ \beta_{f_i} ) \circ D_n \Big) \times 1_{N(I[1])} \Bigg) \circ D_2.
    $$ 
    This precisely places us in the situation of Lemma \ref{lemma_h2_operad_of_htpies}. Note however that the resulting induced left homotopy is of the form
    $$
        H : X_{Sp(r)} \times N(I[1]) \rightarrow X_K
    $$
    and maps from $\mu_K$ to $\mu_{K'}'$ by Proposition \ref{prop_cd_decomp_scd}. Moreover, given the inclusion $b : K_0 \hookrightarrow K$, it is such that $b^* \circ H = Q$. Since $\tau_0^* \circ b^* = \tau^*$, we have our left homotopy $H$ such that our original natural isomorphism is of the form
    $$
        h_1(\tau^{x_0, x_r}) \circ h_1(H^{x_0, x_r}) \circ F_{x_0, \cdots, x_r}.
    $$
    This confirms by Corollary \ref{corr_h1_lift_mapssi_h1c_ident_htpies} that the equality holds.
\end{proof}

\begin{theorem}
    Let $X$ be a $2$-fold Segal space and $(\mu_n)_{n \geq 0}$ a choice of horizontal compositions. Then $h_2(X, (\mu_n)_{n \geq 0})$ is an unbiased bicategory.
\end{theorem}

\begin{proof}
    We begin by dealing with associativity conditions. Let $n, m_1, \cdots, m_n \in \mathbb{Z}_{> 0}$ and $k_1^1, \cdots, k_n^{m_n} \in \mathbb{Z}_{\geq 0}$. Take a thrice-nested tuple of elements of $(X_{0, 0})_0$ of the form
    $$
        L = (L_p)^n_{p = 1} = ((L_{p, q})^{m_p}_{q = 1})^n_{p = 1} = (((x_{p, q, s})_{s = 0}^{k_p^q})_{q = 1}^{m_p})_{p = 1}^n
    $$
    where $x_{p, q, k_p^q} = x_{p, q+1, 0}$ if $q < m_p$ and $x_{p, m_p, k_p^{m_p}} = x_{p+1, 1, 0}$ if $p < n$. Let $t_p = \sum_{q = 1}^{m_p} k_p^q$ and $r = \sum_{p = 1}^n t_p$. Set $(x_0, \cdots, x_r)$ to be the flattened version of $L$ after removing the elements $x_{p, q, k^q_p}$, not including $x_{n, m_n, k^{m_n}_n}$. Let $(x^p_0, \cdots, x^p_{t_p})$ be the flattened version of $L_p$ after removing each $x_{p, q, k^q_p}$ not including $x_{p, m_p, k_p^{m_p}}$. 

    We take an interest in the four composition operations $\bullet^{x_0, \cdots, x_r}_{Q_i}$ for $i \in \{1, 2, 3, 4\}$, where the four simplicial composition diagrams $Q_i$ are as follows:
    \begin{align*}
        Q_1 := &\Big( Q_{11} \sqcup_{\Delta[0]} \cdots \sqcup_{\Delta[0]} Q_{1n} \Big) \sqcup_{Sp(n)} \Delta[n] \\
        Q_2 := &\Big( \Delta[t_1] \sqcup_{\Delta[0]} \cdots \sqcup_{\Delta[0]} \Delta[t_n] \Big) \sqcup_{Sp(n)} \Delta[n] \\
        Q_3 := &\Big( \Delta[k_1^1] \sqcup_{\Delta[0]} \cdots \sqcup_{\Delta[0]} \Delta[k_n^{m_n}] \Big) \sqcup_{Sp(\sum_{i = 1}^n m_i)} \Delta[\sum_{i = 1}^n m_i] \\
        Q_4 := &\Delta[r]
    \end{align*}
    where, for $1 \leq p \leq n$,
    $$
        Q_{1p} := (\Delta[k_p^1] \sqcup_{\Delta[0]} \cdots \sqcup_{\Delta[0]} \Delta[k_p^{m_p}]) \sqcup_{Sp(m_p)} \Delta[m_p].
    $$
    There is a clear commutative diagram in $\textbf{SCD}_r$
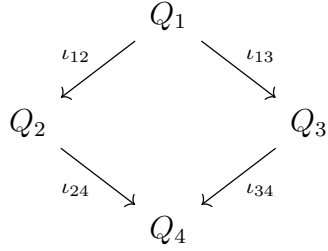
    \[\begin{tikzcd}
    	& {Q_1} \\
    	{Q_2} && {Q_3} \\
    	& {Q_4}
    	\arrow["{\iota_{12}}"', from=1-2, to=2-1]
    	\arrow["{\iota_{24}}"', from=2-1, to=3-2]
    	\arrow["{\iota_{13}}", from=1-2, to=2-3]
    	\arrow["{\iota_{34}}", from=2-3, to=3-2]
    \end{tikzcd}\]
    We wish to show that this induces the diagram of natural isomorphisms
    
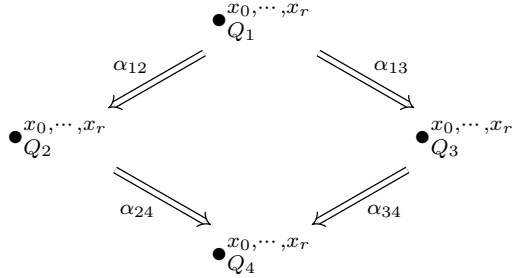
\[\begin{tikzcd}
    	& {\bullet^{x_0, \cdots, x_r}_{Q_1}} \\
    	{\bullet^{x_0, \cdots, x_r}_{Q_2}} && {\bullet^{x_0, \cdots, x_r}_{Q_3}} \\
    	& {\bullet^{x_0, \cdots, x_r}_{Q_4}}
    	\arrow["{\alpha_{12}}"', Rightarrow, from=1-2, to=2-1]
    	\arrow["{\alpha_{24}}"', Rightarrow, from=2-1, to=3-2]
    	\arrow["{\alpha_{13}}", Rightarrow, from=1-2, to=2-3]
    	\arrow["{\alpha_{34}}", Rightarrow, from=2-3, to=3-2]
    \end{tikzcd}\]
    that we seek to prove commutes, where
    \begin{align*}
        \alpha_{12} &:= 1_{\bullet^{x_{1, 1, 0}, \cdots, x_{n, 1, 0}, x_{n, m_n, k^{m_n}_n}}} \circ \prod_{p = 1}^n \Phi^{x^p_0, \cdots, x^p_{t_p}}_{(X, (\mu_n)_{n \geq 0})}(f_{k^p_1, \cdots, k^p_{m_p}}) \\
        \alpha_{13} &:= \Phi^{x_{1, 1, 0}, \cdots, x_{n, 1, 0}, x_{n, m_n, k^{m_n}_n}}_{(X, (\mu_n)_{n \geq 0})}(f_{m_1, \cdots, m_n}) \circ \prod_{p = 1}^n 1_{\Big(\bullet^{x^p_0, \cdots, x^p_{t_p}}_{Q_{1p}}\Big)} \\
        \alpha_{24} &:= \Phi^{x_0, \cdots, x_r}_{(X, (\mu_n)_{n \geq 0})}(f_{t_1, \cdots, t_n}) \\
        \alpha_{34} &:= \Phi^{x_0, \cdots, x_r}_{(X, (\mu_n)_{n \geq 0})}(f_{k^1_1, \cdots, k^{m_n}_n}).
    \end{align*}
    Using Lemma \ref{lemma_h2_psi_operadic}, it is straightforward to identify each $\alpha_{ij}$ with the corresponding $\Phi_{(X, (\mu_n)_{n \geq 0})}^{x_0, \cdots, x_r}(\iota_{ij})$. Thus, by functoriality, the diagram commutes as needed.
    
    The conditions on unitors is then trivial, since all involved natural isomorphisms are identities. Indeed, the relevant associators are $\Phi_{(X, (\mu_n)_{n \geq 0})}^{x_0, \cdots, x_n}(f_{1, \cdots, 1})$, which are constant as $f_{1, \cdots, 1} = 1_{\Delta[n]}$.
\end{proof}
\section{Homotopy Unbiased Pseudofunctors}

In order for $h_2$ to be a genuine functor, we need to demonstrate how to convert a map $f : (X, (\mu_n)_{n \geq 0}) \rightarrow (Y, (\nu_n)_{n \geq 0})$ in $\textbf{SeSp}_2^{comp}$ into a pseudofunctor
$$
    h_2(f) : h_2(X, (\mu_n)_{n \geq 0}) \rightarrow (Y, (\nu_n)_{n \geq 0}).
$$
Before we can do so, a great deal of the technology we developed for $h_2$'s behavior on objects now needs to be extended to handle maps $X \rightarrow Y$ of $2$-fold Segal spaces.

We will first need to extend the notion of being object-fibered to a map between $2$-fold Segal spaces:

\begin{proposition}
    Let $F : X \rightarrow Y$ be a map in $\textbf{sSpace}_k$. Consider two commutative diagrams in $\textbf{sSpace}_{k-1}$
    \[\begin{tikzcd}
    	A && K & A && K \\
    	& {(X_0)^n} & {(Y_0)^n} & {(X_0)^n} & {(Y_0)^n} \\
    	B && L & B && L
    	\arrow["{(F_0)^n}", from=2-2, to=2-3]
    	\arrow["{F_B}"', from=3-1, to=3-3]
    	\arrow[from=3-1, to=2-2]
    	\arrow[from=3-3, to=2-3]
    	\arrow[from=1-3, to=2-3]
    	\arrow[from=1-1, to=2-2]
    	\arrow["{F_A}", from=1-1, to=1-3]
    	\arrow["f"', from=1-1, to=3-1]
    	\arrow["{F_A}", from=1-4, to=1-6]
    	\arrow["{(F_0)^n}", from=2-4, to=2-5]
    	\arrow[from=1-4, to=2-4]
    	\arrow[from=1-6, to=2-5]
    	\arrow["g", from=1-6, to=3-6]
    	\arrow[from=3-6, to=2-5]
    	\arrow[from=3-4, to=2-4]
    	\arrow["{F_B}"', from=3-4, to=3-6]
    \end{tikzcd}\]
    for some chosen maps $F_A$ and $F_B$, so that $f$ and $g$ are object-fibered. Let $x_0, \cdots, x_n \in (X_{0, \cdots, 0})_0$.

    Then, defining $F_B^{x_0, \cdots, x_n}$ to be the map
    \[\begin{tikzcd}
    	{B^{x_0, \cdots, x_n}} & {L^{F(x_0), \cdots, F(x_n)}} \\
    	{B^{F(x_0), \cdots, F(x_n)}}
    	\arrow["{F_B^{x_0, \cdots, x_n}}", from=1-1, to=1-2]
    	\arrow[hook', from=1-1, to=2-1]
    	\arrow["{F_B^{F(x_0), \cdots, F(x_n)}}"', from=2-1, to=1-2]
    \end{tikzcd}\]
    and similarly for $F_A^{x_0, \cdots, x_n}$, we have that
    $$
        F_B^{x_0, \cdots, x_n} \circ f^{x_0, \cdots, x_n} = (F_B \circ f)^{f(x_0), \cdots, f(x_n)} \circ \mathcal{I}^{F, A}_{x_0, \cdots, x_n},
    $$
    where $\mathcal{I}^{F, A}_{x_0, \cdots, x_n} : A^{x_0, \cdots, x_n} \hookrightarrow A^{f(x_0), \cdots, f(x_n)}$ is the inclusion. Similarly, we have
    $$
        g^{f(x_0), \cdots, f(x_n)} \circ F_A^{x_0, \cdots, x_n} = (g \circ F_A)^{f(x_0), \cdots, f(x_n)} \circ \mathcal{I}^{F, A}_{x_0, \cdots, x_n}.
    $$
\end{proposition}

\begin{proof}
    Diagram chases show both statements to be true.
\end{proof}

The central idea controlling our composition, coherence isomorphisms and coherence conditions was the collections of functors $\Phi^{x_0, \cdots, x_n}_{(X, (\mu_n)_{n \geq 0})}$ for all $x_0, \cdots, x_n \in (X_{0, 0})_0$ and $n \geq 0$. The flavor of results we will need to prove for pseudofunctors are similar, so we seek an augmented version of this construct.

In essence, what $\Phi^{x_0, \cdots, x_n}_{(X, (\mu_n)_{n \geq 0})}$ accomplished was to convert diagrams in $\textbf{SCD}_n$ into commutative diagrams of coherence isomorphisms. A reasonable extension of this category for the purposes of pseudofunctors is $\textbf{SCD}_n \times [1]$, with the morphisms $(1_K, 0 < 1)$ representing a homotopy from composition in the codomain to the domain:

\begin{definition}
    Let $f : (X, (\mu_n)_{n \geq 0}) \rightarrow (Y, (\nu_n)_{n \geq 0})$ be a map in $\textbf{SeSp}_2^{comp}$. Let $x, y \in (X_{0, 0})_0$. Then define the functor
    $$
        \omega^{x, y}_f : \textbf{SCD}_n \times [1] \rightarrow \textbf{Fun} \Big(h_1(X_{Sp(n)}^{x, y}), h_1(Y(x, y))\Big)
    $$
    to be the functor sending
    \begin{align*}
        (K, 0) &\mapsto h_1((\tau_K^*)^{f(x), f(y)} \circ \nu_K^{f(x), f(y)} \circ f^{x, y}_{Sp(n)}) \\
        (K, 1) &\mapsto h_1((\tau_K^*)^{f(x), f(y)} \circ f^{x, y}_K \circ \mu_K^{x, y})
    \end{align*}
    and sending a morphism to the map generated by the following cases:

    \begin{enumerate}
        \item $(g, 1_0) \mapsto \omega_{(Y, (\nu_n)_{n \geq 0})}^{f(x), f(y)}(g) \circ h_1(f^{x, y}_{Sp(n)})$, where $g : K_1 \rightarrow K_2$;
        
        \item $(g, 1_1) \mapsto h_1(f^{x, y}_1) \circ \omega_{(X, (\mu_n)_{n \geq 0})}^{x, y}(g)$, where $g$ is as above;
        
        \item $(1_K, 0 < 1) \mapsto h_1((\tau_K^*)^{f(x), f(y)}) \circ h_1(\eta_K^{f(x), f(y)}) \circ h_1(\mathcal{I}^{f, X_{Sp(n)}}_{x, y})$, where $\eta_K$ is an induced left homotopy from $\nu_K \circ f_{Sp(n)}$ to $f_K \circ \mu_K$ by the lifting problem
        \[\begin{tikzcd}
        	& {X_K} & {Y_K} \\
        	\\
        	{X_{Sp(n)}} && {Y_{Sp(n)}}
        	\arrow["{f_K}", dashed, from=1-2, to=1-3]
        	\arrow[from=1-3, to=3-3]
        	\arrow["{\mu_K}", curve={height=-6pt}, dashed, from=3-1, to=1-2]
        	\arrow["{\nu_K \circ f_{Sp(n)}}"{description}, curve={height=6pt}, dashed, from=3-1, to=1-3]
        	\arrow[from=3-1, to=3-3]
        \end{tikzcd}\]
    \end{enumerate}
\end{definition}

\begin{proposition}
    $\omega^{x, y}_f$ is a functor.
\end{proposition}

\begin{proof}
    Functoriality on the subcategory $\textbf{SCD}_n \times \{0, 1\}$ is clear. We need only prove that, for a map $g : K_1 \rightarrow K_2$,
    $$
        \omega^{x, y}_f(g, 1_1) \omega^{x, y}_f(1_{K_1}, 0 < 1) = \omega^{x, y}_f(1_{K_2}, 0 < 1) \omega^{x, y}_f(g, 1_0).
    $$
    Expanding out terms, we have that
    \begin{align*}
        \omega^{x, y}_f(g, 1_0) &= h_1 \Big((\tau_{K_1}^*)^{f(x), f(y)} \circ \alpha_g^{f(x), f(y)} \circ (f_{Sp(n)} \times 1_{N(I[1])})^{f(x), f(y)} \circ \mathcal{I}^{f, X_{Sp(n)} \times N(I[1])}_{x, y} \Big) \\
        \omega^{x, y}_f(1_{K_1}, 0 < 1) &= h_1 \Big((\tau_{K_1}^*)^{f(x), f(y)} \circ \eta_{K_1}^{f(x), f(y)} \circ \mathcal{I}^{f, X_{Sp(n)} \times N(I[1])}_{x, y} \Big)\\
        \omega^{x, y}_f(1_{K_2}, 0 < 1) &= h_1 \Big((\tau_{K_2}^*)^{f(x), f(y)} \circ \eta_{K_2}^{f(x), f(y)} \circ \mathcal{I}^{f, X_{Sp(n)} \times N(I[1])}_{x, y} \Big)\\
        \omega^{x, y}_f(g, 1_1) &= h_1 \Big(f_1^{f(x), f(y)}  \circ (\tau_{K_1}^*)^{f(x), f(y)} \circ \beta_g^{f(x), f(y)} \circ \mathcal{I}^{f, X_{Sp(n)} \times N(I[1])}_{x, y} \Big)
    \end{align*}
    for induced left homotopies $\alpha_g$ from $\nu_{K_1}$ to $\nu_{K_2}'$ and $\beta_g$ from $\mu_{K_1}$ to $\mu_{K_2}'$.

    We are able to produce the two compositions $(f^{f(x), f(y)}_{K_1} \circ \beta^{f(x), f(y)}_g) (\eta_{K_1}^{f(x), f(y)})$ and $((g^*)^{f(x), f(y)} \circ \eta_{K_2}^{f(x), f(y)}) \Big(\alpha^{f(x), f(y)}_g \circ (f^{f(x), f(y)}_{Sp(n)} \times 1_{N(I[1])})\Big)$, which yield the same natural isomorphisms as before. As these compositions are left homotopies induced between the same solutions to the same lifting problem, the equality holds.
\end{proof}

\begin{definition}
    Let $f : (X, (\mu_n)_{n \geq 0}) \rightarrow (Y, (\nu_n)_{n \geq 0})$ be a map in $\textbf{SeSp}_2^{comp}$. Let $x_0, \cdots, x_n \in (X_{0, 0})_0$. Then define the map
    $$
        \phi^{x_0, \cdots, x_n}_f : \textbf{Fun}\Big( h_1(X_{Sp(n)}^{x_0, x_n}), h_1(Y(x_0, x_n)) \Big) \rightarrow \textbf{Fun} \Big( \prod_{i = 1}^n h_1(X(x_{i-1}, x_i)), h_1(Y(x_0, x_n)) \Big)
    $$
    to be precomposition with the inclusion
    $$
        F_{x_0, \cdots, x_n} : \prod_{i = 1}^n h_1(X(x_{i-1}, x_i)) \cong h_1(\prod_{i = 1}^n X(x_{i-1}, x_i)) \hookrightarrow h_1(X_{Sp(n)}^{x_0, x_n}).
    $$
    Write $\Phi^{x_0, \cdots, x_n}_f := \phi^{x_0, \cdots, x_n}_f \circ \omega^{x_0, x_n}_f$.
\end{definition}

Henceforth, for a map $f : (X, (\mu_n)_{n \geq 0}) \rightarrow (Y, (\nu_n)_{n \geq 0})$, write $\bullet^{x_0, \cdots, x_n}_A$ or $\bullet^{x_0, \cdots, x_n}_{K, A}$ for $A \in \{X, Y\}$ to distinguish the composition operations in each of these spaces.

The `operadic' flavor of $\Phi^{x_0, \cdots, x_n}_{(X, (\mu_n)_{n \geq 0})}$ is admittedly somewhat diluted in the following lemma concerning $\Phi^{x_0, \cdots, x_n}_f$. Future work will include seeking out a perhaps more natural category than $\textbf{SCD}_n \times [1]$, which appears too `coarse' in the following results to lend itself to such an interpretation. Nonetheless, some operadic color still shines through, at least enough for us to prove what we need:

\begin{lemma} \label{lemma_h2_operadic_funct}
    Let $f : (X, (\mu_n)_{n \geq 0}) \rightarrow (Y, (\nu_n)_{n \geq 0})$ be a map in $\textbf{SeSp}_2^{comp}$. Let $n > 0$ and choose integers $k_1, \cdots, k_n \in \mathbb{Z}_{\geq 0}$ and a nested sequence of elements $((x^i_j)_{j = 0}^{k_i})_{i = 1}^n$ of $(X_{0, 0})_0$ such that $x^i_{k_i} = x^{i+1}_0$ for $i < n$. Let $(x_0, \cdots, x_r)$ be the flattened version of the sequence with all $x^i_{k_i}$ removed for $i < n$. Set $K$ to be the simplicial composition diagram
    $$
        K = (K_1 \sqcup_{\Delta[0]} \cdots \sqcup_{\Delta[0]} K_n) \sqcup_{Sp(n)} K_0
    $$
    such that $K_i$ has arity $k_i$ and $k_0 = n$. Then the diagram
    \[\begin{tikzcd}
    	{\bullet_{K, Y}^{f(x_0), \cdots, f(x_r)} \circ h_1(f^{x_0, \cdots, x_r}_{Sp(r)})} \\
    	\\
    	& {\bullet_{K_0, Y}^{f(x^1_0), \cdots, f(x^n_0), f(x^n_{k_n})} \circ h_1(f^{x^1_0, \cdots, x^n_0, x^n_{k_n}}_{Sp(n)}) \circ \prod\limits_{i = 1}^n \bullet_{K_0, X}^{x^i_0, \cdots, x^i_{k_i}}} \\
    	\\
    	{h_1(f^{x_0, x_r}_1) \circ \bullet_{K, X}^{x_0, \cdots, x_r}}
    	\arrow["{\Phi^{x^1_0, \cdots, x^n_0, x^n_{k_n}}_f(1_{K_0}, 0 < 1) \circ \prod\limits_{i = 1}^n 1_{\bullet_{K_i, X}^{x^i_0, \cdots, x^i_{k_i}}}}", from=3-2, to=5-1]
    	\arrow["{1_{\bullet_{K_0, Y}^{f(x^1_0), \cdots, f(x^n_0), f(x^n_{k_n})}} \circ \prod\limits_{i = 1}^n \Phi^{x^i_0, \cdots, x^i_{k_i}}_f(1_{K_i}, 0 < 1)}", from=1-1, to=3-2]
    	\arrow["{\Phi^{x_0, \cdots, x_r}_f(1_K, 0 < 1)}"', from=1-1, to=5-1]
    \end{tikzcd}\]
    commutes.
\end{lemma}

\begin{proof}
    The proof is similar to that of Lemma \ref{lemma_h2_psi_operadic}. Expanding out all the terms, we obtain the expressions
    \begin{align*}
        R &:= h_1((\tau_K^*)^{f(x_0), f(x_r)}) \circ h_1(\eta_K^{f(x_0), f(x_r)}) \circ h_1(\mathcal{I}^{f, X}_{x_0, x_r}) \circ F_{x_0, \cdots, x_r} \\
        S &:= \bullet_{K_0, Y}^{f(x^1_0), \cdots, f(x^n_0), f(x^n_{k_n})} \circ \prod_{i = 1}^n h_1((\tau_{K_i}^*)^{f(x^i_0), f(x^i_{k_i})}) \circ h_1(\eta_{K_i}^{f(x^i_0), f(x^i_{k_i})}) \circ h_1(\mathcal{I}^{f, X}_{x^i_0, x^i_{k_i}}) \circ F_{x^i_0, \cdots, x^i_{k_i}} \\
        T &:= h_1((\tau_{K_0}^*)^{f(x_0), f(x_r)}) \circ h_1(\eta_K^{f(x_0), f(x_r)})  \circ h_1(\mathcal{I}^{f, X}_{x_0, x_r}) \circ F_{x^1_0, \cdots, x^n_0, x^n_{k_n}} \circ \prod_{i = 1}^n \bullet_{K_0, Y}^{f(x^i_0), \cdots, f(x^i_{k_i})}
    \end{align*}
    If we can show both $S$ and $T$ to be of the form 
    $$
        h_1((\tau_K^*)^{f(x_0), f(x_r)}) \circ h_1(Q_A) \circ h_1(\mathcal{I}^{f, X}_{x_0, x_r}) \circ F_{x_0, \cdots, x_r}
    $$
    for some suitable induced left homotopies $Q_A$ where $A \in \{S, T\}$, then we will be done.

    Set $D_k : X_{Sp(r)} \times N(I[1]) \rightarrow X_{Sp(r)} \times N(I[1])^k$ to the diagonal. For $A \in \{X, Y\}$, define the map
    $$
        \mathcal{F}_{a_0, \cdots, a_k}^A : \prod_{i = 1}^k A(a_{i-1}, a_i) \rightarrow A_{Sp(k)}^{a_0, a_k}
    $$
    to be the evident inclusion. Define $I_{K, X}$ to be the identity homotopy from $\mu_K$ to itself and $I_{K, Y}$ to be as such for $\nu_K$.

    Similarly to the proof of Lemma \ref{lemma_h2_psi_operadic}, we have that $S = h_1((\tau_{K_0}^*)^{f(x_0), f(x_r)}) \circ h_1(\mathcal{S})$, where $\mathcal{S}$ is the map
    $$
        (I_{K_0, Y}^{f(x_0), f(x_r)}) \circ \mathcal{F}^Y_{f(x^1_0), \cdots, f(x^n_0), f(x^n_{k_n})} \circ \Bigg( \Big( \prod_{i = 1}^n Q_i \circ D_n^{x_0, \cdots, x_r} \Big) \times 1_{N(I[1])} \Bigg) \circ D_2^{x_0, \cdots, x_r}
    $$
    where
    $$
        Q_i := (\tau_{K_i}^* \circ \eta_{K_i})^{f(x^i_0), f(x^i_{k_i})} \circ \mathcal{I}^{f, X_{Sp(k_i)} \times N(I[1])}_{x^i_0, x^i_{k_i}} \circ \mathcal{F}^X_{x^i_0, \cdots, x^i_{k_i}}
    $$
    and that $T = h_1((\tau_{K_0}^*)^{f(x_0), f(x_r)}) \circ h_1(\mathcal{T})$, where $\mathcal{T}$ is the map
    $$
        \eta_K^{f(x_0), f(x_r)} \circ \mathcal{I}_{x_0, x_r}^{f, X_{Sp(n)} \times N(I[1])} \circ \mathcal{F}_{x^1_0, \cdots, x^n_0, x^n_{k_n}}^X \circ \Bigg( \Big( \prod_{i = 1}^n W_i \circ D_n^{x_0, \cdots, x_r} \Big) \times 1_{N(I[1])} \Bigg) \circ D_2^{x_0, \cdots, x_r}
    $$
    where
    $$
        W_i := (\tau_{K_i}^*)^{x^i_0, x^i_{k_i}} \circ I_{K_i, X}^{x^i_0, x^i_{k_i}} \circ \mathcal{F}^X_{x^i_0, \cdots, x^i_{k_i}}.
    $$
    By similar manipulations to those in Lemma \ref{lemma_h2_psi_operadic}, we can reduce each of these further such that $\mathcal{S} = Q_S^{f(x_0), f(x_r)} \circ \mathcal{I}^{f, X}_{x_0, x_r} \circ \mathcal{F}^X_{x_0, \cdots, x_r}$ where
    $$
        Q_S := I_{K_0, Y} \circ \Bigg( \Big( \prod_{i = 1}^n (\tau_{K_i}^* \circ \eta_{K_i}) \circ D_n \Big) \times 1_{N(I[1])} \Bigg) \circ D_2
    $$
    and such that $\mathcal{T} = Q_T^{f(x_0), f(x_r)} \circ \mathcal{I}^{f, X}_{x_0, x_r} \circ \mathcal{F}^X_{x_0, \cdots, x_r}$ where
    $$
        Q_T := \eta_K \circ \Bigg( \Big( \prod_{i = 1}^n (\tau_{K_i}^* \circ I_{K_i, X}) \circ D_n \Big) \times 1_{N(I[1])} \Bigg) \circ D_2.
    $$
    Both of these are of suitable form for application of Lemma \ref{lemma_h2_operad_of_htpies}, the resulting homotopies from which are both of the form $X_{Sp(r)} \times N(I[1]) \rightarrow Y_K$ and can be composed to complete the proof.
\end{proof}

Armed with these results, we are able to finally declare our definition of $h_2(f)$ and prove its correctness.

\begin{definition} \label{defn_h2_h2(f)}
    Let $f : (X, (\mu_n)_{n \geq 0}) \rightarrow (Y, (\nu_n)_{n \geq 0})$ be a morphism in $\textbf{SeSp}_2^{comp}$. Then define
    $$
        h_2(f) : h_2(X, (\mu_n)_{n \geq 0}) \rightarrow h_2(Y, (\nu_n)_{n \geq 0})
    $$
    to be the pseudofunctor defined such that:
    \begin{enumerate}
        \item For every $x \in (X_{0, 0})_0$, $h_2(f)(x) = (f_{0, 0})_0(x)$;

        \item For every $x, y \in (X_{0, 0})_0$, $h_2(f)_{x, y} = h_1(f_1^{x, y}) : h_1(X(x, y)) \rightarrow h_1(Y(f(x), f(y)))$;

        \item For every $n \in \mathbb{Z}_{> 0}$ and $x_0, \cdots, x_n \in (X_{0, 0})_0$, the natural isomorphism
        $$
            \pi_{x_0, \cdots, x_n} := \Phi^{x_0, \cdots, x_n}_f(1_{\Delta[n]}, 0 < 1);
        $$
        \item For every $x \in (X_{0, 0})_0$, the natural isomorphism $\pi_x$ is the identity.
    \end{enumerate}
\end{definition}

\begin{theorem}
    Let $f : (X, (\mu_n)_{n \geq 0}) \rightarrow (Y, (\nu_n)_{n \geq 0})$ be a morphism in $\textbf{SeSp}_2^{comp}$. Then $h_2(f)$ is a pseudofunctor.
\end{theorem}

\begin{proof}
    Let $n \in \mathbb{Z}_{> 0}$ and $k_1, \cdots, k_n \in \mathbb{Z}_{\geq 0}$, with $r = \sum_{i = 1}^n k_i$. Consider any $n$-tuple of tuples of elements of $(X_{0, 0})_0$
    $$
        Y := ((x_0^1, x_1^1, \cdots, x_{k_1}^1), \cdots, (x_0^n, \cdots, x_{k_n}^n))
    $$
    such that $x_{k_i}^i = x_0^{i+1}$ for every $i < n$. Let $(x_0, \cdots, x_r)$ be the flattened tuple after removing each $x_{k_i}^i$ for $i < n$.
    
    Consider moreover the composition diagram
    $$
        K := (\Delta[k_1] \sqcup_{\Delta[0]} \cdots \sqcup_{\Delta[0]} \Delta[k_n]) \sqcup_{Sp(n)} \Delta[n].
    $$
    We are tasked with showing that the diagram
    \[\begin{tikzcd}
    	{\bullet^{f(x_0), \cdots, f(x_r)}_{K, Y} \circ h_1(f_{Sp(r)}^{x_0, \cdots, x_r})} & {\bullet_Y^{x_0, \cdots, x_r} \circ h_1(f_{Sp(r)}^{x_0, \cdots, x_r})} \\
    	\\
    	{\bullet_Y^{f(x^1_0), \cdots, f(x^n_0), f(x^n_{k_n})} \circ h_1(f_{Sp(n)}^{x^1_0, \cdots, x^n_0, x^n_{k_n}}) \circ \prod_{i = 1}^n \bullet_X^{x^i_0, \cdots, x^i_{k_i}}} \\
    	\\
    	{h_1(f_1^{x_0,x_r}) \circ \bullet^{x_0, \cdots, x_r}_{K, X}} & {h_1(f_1^{x_0, x_r}) \circ \bullet_X^{x_0, \cdots, x_r}}
    	\arrow["{\Phi_f^{x_0, \cdots, x_r}(f_{k_1, \cdots, k_n}, 1_0)}", Rightarrow, from=1-1, to=1-2]
    	\arrow["{1_{\bullet_Y^{f(x^1_0), \cdots, f(x^n_0), f(x^n_{k_n})}} \circ \prod_{i = 1}^n \Phi^{x^i_0, \cdots, x^i_{k_i}}_f(1_{\Delta[k_i]}, 0 < 1)}", Rightarrow, from=1-1, to=3-1]
    	\arrow["{\Phi^{x^1_0, \cdots, x^n_0, x^n_{k_n}}_f(1_{\Delta[n]}, 0 < 1) \circ \prod_{i = 1}^n 1_{\bullet_{X}^{x^i_0, \cdots, x^i_{k_i}}}}", Rightarrow, from=3-1, to=5-1]
    	\arrow["{\Phi_f^{x_0, \cdots, x_r}(f_{k_1, \cdots, k_n}, 1_1)}"', Rightarrow, from=5-1, to=5-2]
    	\arrow["{\Phi_f^{x_0, \cdots, x_r}(1_{\Delta[r]}, 0 < 1)}"{description}, Rightarrow, from=1-2, to=5-2]
    \end{tikzcd}\]
    commutes. This, however, must hold, by Lemma \ref{lemma_h2_operadic_funct} and functoriality of $\Phi^{x_0, \cdots, x_r}_f$. The other commuting diagram is trivial, as all involved natural isomorphisms are identities.
\end{proof}

We now need to show that $h_2$, given the behavior on objects from Definition \ref{defn_h2_h2} and morphisms from Definition \ref{defn_h2_h2(f)}, defines a valid functor into $\textbf{UBicat}$:

\begin{theorem}
    Let
    $$
        (X, (\mu_n)_{n \geq 0}) \xrightarrow{f} (Y, (\nu_n)_{n \geq 0}) \xrightarrow{g} (Z, (\omega_n)_{n \geq 0})
    $$
    be a chain of morphisms in $\textbf{SeSp}^{comp}_2$. Then $h_2(g \circ f) = h_2(g) \circ h_2(f)$.
\end{theorem}

\begin{proof}
    The equivalence is evident on the object mapping and behavior on hom-categories. For the compositors, let the isomorphisms for $g$ be labelled as $\theta$ and for $f$ be $\pi$. The ones $\psi$ for $g \circ f$ must be proven to be such that
    $$
        \psi_{x_0, \cdots, x_n} = \big( h_1(g_1^{f(x_0), f(x_n)}) \circ (\pi_{x_0, \cdots, x_n}) \big) \big( \theta_{f(x_0), \cdots, f(x_n)} \circ h_1(f_{Sp(n)}^{x_0, \cdots, x_n}) \big).
    $$
    Expanding out all of these terms gives us that
    \begin{align*}
        \pi_{x_0, \cdots, x_n} &= h_1(Y_{\langle 0, n \rangle}^{f(x_0), f(x_n)}) \circ h_1(\eta_n^{f(x_0), f(x_n)}) \circ h_1(\mathcal{I}^{f, X_{Sp(n)}}_{x_0, x_n}) \circ F^X_{x_0, \cdots, x_n} \\
        \theta_{f(x_0), \cdots, f(x_n)} &= h_1(Z_{\langle 0, n \rangle}^{gf(x_0), gf(x_n)}) \circ h_1(\kappa_n^{gf(x_0), gf(x_n)}) \circ h_1(\mathcal{I}^{g, Y_{Sp(n)}}_{f(x_0), f(x_n)}) \circ F^Y_{f(x_0), \cdots, f(x_n)} \\
        \psi_{x_0, \cdots, x_n} &= h_1(Z_{\langle 0, n \rangle}^{gf(x_0), gf(x_n)}) \circ h_1(\zeta_n^{gf(x_0), gf(x_n)}) \circ h_1(\mathcal{I}^{gf, X_{Sp(n)}}_{x_0, x_n}) \circ F^X_{x_0, \cdots, x_n}
    \end{align*}
    where $F^A_{a_0, \cdots, a_n}$ is the map $F_{a_0, \cdots, a_n}$ adjusted in the obvious way for $A \in \{X, Y\}$ and the left homotopies $\eta_n, \kappa_n$ and $\zeta_n$ are respectively from $\nu_n \circ f_{Sp(n)}$ to $f_n \circ \mu_n$, from $\omega_n \circ g_{Sp(n)}$ to $g_n \circ \nu_n$ and from $\omega_n \circ g_{Sp(n)} \circ f_{Sp(n)}$ to $g_n \circ f_n \circ \mu_n$.

    Note that $\theta_{f(x_0), \cdots, f(x_n)} \circ h_1(f_{Sp(n)}^{x_0, \cdots, x_n})$ is equal to
    $$
        h_1(Z_{\langle 0, n \rangle}^{gf(x_0), gf(x_n)}) \circ h_1((\kappa_n \circ (f_{Sp(n)} \times 1_{N(I[1])}))^{gf(x_0), gf(x_n)}) \circ h_1(\mathcal{I}^{gf, X_{Sp(n)}}_{x_0, x_n}) \circ F^X_{x_0, \cdots, x_n}
    $$
    while similarly $h_1(g_1^{f(x_0), f(x_n)}) \circ (\pi_{x_0, \cdots, x_n})$ is equal to
    $$
        h_1(g^{f(x_0), f(x_n)}_1) \circ h_1(Y_{\langle 0, n \rangle}^{f(x_0), f(x_n)}) \circ h_1(\eta_n^{f(x_0), f(x_n)}) \circ h_1(\mathcal{I}^{f, X_{Sp(n)}}_{x_0, x_n}) \circ F^X_{x_0, \cdots, x_n}
    $$
    which is then equal to
    $$
         h_1(Z_{\langle 0, n \rangle}^{gf(x_0), gf(x_n)}) \circ h_1((g_n \circ \eta_n)^{f(x_0), f(x_n)}) \circ h_1(\mathcal{I}^{f, X_{Sp(n)}}_{x_0, x_n}) \circ F^X_{x_0, \cdots, x_n}.
    $$
    All of these three cases are now of the form $h_1(Z_{\langle 0, n \rangle}^{gf(x_0), gf(x_n)}) \circ h_1(\Gamma^{f(x_0), f(x_n)}) \circ h_1(\mathcal{I}^{f, X_{Sp(n)}}_{x_0, x_n}) \circ F^X_{x_0, \cdots, x_n}$ for induced left homotopies $\Gamma$. We find that the left homotopies $g_n \circ \eta_n$ and $\kappa_n \circ (f_{Sp(n)} \times 1_{N(I[1])})$ are now readily composable, so by Proposition \ref{prop_h1_comp_htpies_is_natiso_comp} and Corollary \ref{corr_h1_lift_mapssi_h1c_ident_htpies} the equality holds.
\end{proof}

\begin{theorem}
    Let $(X, (\mu_n)_{n \geq 0})$ be an object in $\textbf{SeSp}^{comp}_2$. Then $h_2(1_X) = 1_{h_2(X, (\mu_n)_{n \geq 0})}.$
\end{theorem}

\begin{proof}
    The behavior on objects and hom-categories is self-evident. Moreover,
    $$
        \Phi^{x_0, \cdots, x_n}_{1_X}(1_{\Delta[n]}, 0 < 1)
    $$
    is built of a left homotopy with the same domain and codomain, so can be assumed to be constant, making the natural isomorphisms trivial.
\end{proof}

We are finally able to state what we consider to be the central definition of this thesis:

\begin{definition}
    Define
    $$
        h_2 : \textbf{SeSp}^{comp}_2 \rightarrow \textbf{UBicat}
    $$
    to be the functor that sends $(X, (\mu_n)_{n \geq 0})$ to $h_2(X, (\mu_n)_{n \geq 0})$ and a morphism $f$ to $h_2(f)$.
\end{definition}

A useful consequence of functoriality is that we may finally show how all choices of horizontal compositions made in a $2$-fold Segal space $X$ ultimately have no real consequence in the homotopy bicategory:

\begin{corollary} \label{cor_h2_diffcomps_ident_fctr}
    Let $X$ be a $2$-fold Segal space. Let $(\mu_n)_{n \geq 0}$ and $(\nu_n)_{n \geq 0}$ be two choices of horizontal compositions for $X$. Then there is a pseudofunctor $h_2(X, (\mu_n)_{n \geq 0}) \rightarrow h_2(X, (\nu_n)_{n \geq 0})$ which is an isomorphism in $\textbf{UBicat}$ and acts as the identity on objects and hom-categories.
\end{corollary}
\section{Fundamental Bigroupoids}

To take stock of our progress in defining homotopy bicategories, we now consider the effect of $h_2$ on the case of $S := \textbf{Sing}_{\textbf{ssS}}(U)$ for some given $U \in \textbf{Top}$. We will find this yields a sensible notion of \emph{fundamental bigroupoid} of a topological space, similar in nature to \cite{hardieKampsKieboomHomotopyBigroupoidTopological2001} though with more flexible horizontal composition operations.

To begin, we have a set of objects $\textbf{ob}(h_2(S))$ equal to the set of points in $U$, along with hom-categories of the form
$$
    \textbf{Hom}_{h_2(S)}(x, y) \cong \Pi_1(\{(x, y)\} \times_{U \times U} U^{\Delta_t[1]}),
$$
namely those groupoids whose objects are paths in $U$ from $x$ to $y$ and whose morphisms are homotopy classes of homotopies between such paths that are constant on endpoints. This is thus far precisely what we should expect.

For composition, we observe that $\textbf{Sing}_{\textbf{sS}}$ commutes with limits, since it is a right adjoint. Hence,
$$
    \prod_{S_0}^{1 \leq i \leq n} S_1 \cong \textbf{Sing}_{\textbf{sS}}(U^{Sp_t(n)})
$$
where $Sp_t(n) := \Delta_t[1] \sqcup_{\Delta_t[0]} \cdots \sqcup_{\Delta_t[0]} \Delta_t[1]$. Thus, a choice of horizontal compositions for $S$ actually reduces to solving lifting problems of the form
\[\begin{tikzcd}
	& {\textbf{Sing}_{\textbf{sS}}(U^{\Delta_t[n]})} \\
	{\textbf{Sing}_{\textbf{sS}}(U^{Sp_t(n)})} & {\textbf{Sing}_{\textbf{sS}}(U^{Sp_t(n)})}
	\arrow["id"', from=2-1, to=2-2]
	\arrow["{\textbf{Sing}_{\textbf{sS}}(U^{\lvert g_n \rvert})}", from=1-2, to=2-2]
	\arrow["{\mu_n}"{pos=0.3}, dashed, from=2-1, to=1-2]
\end{tikzcd}\]
where $\lvert g_n \rvert : Sp_t(n) \hookrightarrow \Delta_t[n]$ is the topological spine inclusion.

Note that the map $\lvert g_n \rvert$ is a trivial Hurewicz cofibration; it is the inclusion of a subcomplex and is moreover a homotopy equivalence. This in turn implies that $U^{\Delta_t[n]} \rightarrow U^{Sp_t(n)}$ is a trivial Hurewicz fibration. Hence, solutions to this lifting problem can be obtained from any of the sections of $U^{\lvert g_n \rvert}$, which necessarily exist. We can obtain a simple example of such a section by finding a retract $R_n : \Delta_t[n] \twoheadrightarrow Sp_t[n]$ of $\lvert g_n \rvert$: if we understand that $Sp_t(n) \subseteq \Delta_t[n]$ and define $e_i \in \Delta_t[n]$ such that $e_i$'s $i^{th}$ entry is $1$, we choose to map $e_i \mapsto (\frac{n-i}{n}, 0, \cdots, 0, \frac{i}{n})$, inducing linearly a retraction $r_n$ to the edge from $e_0$ to $e_n$ in $\Delta_t[n]$ that sends
$$
    (x_0, \cdots, x_n) \mapsto (\sum_{i = 0}^n x_i \frac{n - i}{n}, 0, \cdots, 0, \sum_{i = 0}^n x_i \frac{i}{n}).
$$
Composing this with a map $s_n$ to $Sp_t(n)$ that sends
$$
    (1 - x_n, 0, \cdots, 0, x_n) \mapsto ((i + 1) - n x_n) e_i + (n x_n - i) e_{i+1}, \;\;\;\; \frac{i}{n} \leq x_n \leq \frac{i+1}{n}
$$
gives us a final retraction $R_n = s_n \circ r_n : \Delta_t[n] \rightarrow Sp_t(n)$, sending
$$
    (x_0, \cdots, x_n) \mapsto \big((i + 1) - \sum_{j = 0}^n j x_j\big) e_i + \big(\sum_{j = 0}^n j x_j - i\big) e_{i + 1}, \;\;\;\; i \leq \sum_{j = 0}^n j x_j \leq i + 1
$$
We should show that this is the identity on $Sp_t(n)$. Indeed, consider some $(x_0, \cdots, x_n) =  (1 - t)e_i + t e_{i+1} = (0, \cdots, 0, 1-t, t, 0, \cdots, 0)$ for $0 \leq t \leq 1$. Then
$$
    \sum_{j=0}^n j x_j = i(1-t) + (i + 1)t = i - it + it + t = i + t
$$
which ranges from $i$ to $i+1$ with $t$.

Applying $R_n$ then gives us
$$
    R_n(0, \cdots, 0, 1 - t, t, 0, \cdots, 0) = ((i + 1) - (i + t))e_i + (i + t - i)e_{i+1} = (1-t)e_i + te_{i+1}
$$
which confirms that the inclusion $\lvert g_n \rvert : Sp_t(n) \subseteq \Delta_t[n]$ is a section of $R_n$, as needed.

We now have our composition map, which by using the identification
$$
    \textbf{Hom}_{h_2(S)}(x_0, x_1) \times \cdots \times \textbf{Hom}_{h_2(S)}(x_{n-1}, x_n) \cong \Pi_1(\{x_0, \cdots, x_n\} \times_{U^{n+1}} U^{Sp_t(n)})
$$
can be phrased in the form
\[\begin{tikzcd}
	{\Pi_1(\{x_0, \cdots, x_n\} \times_{U^{n+1}} U^{Sp_t(n)})} \\
	{\Pi_1(\{x_0, x_n\} \times_{U \times U} U^{Sp_t(n)})} && {\Pi_1(\{x_0, x_n\} \times_{U \times U} U^{\Delta_t[n]})} \\
	&& {\Pi_1(\{x_0, x_n\} \times_{U \times U} U^{\Delta_t[1]})}
	\arrow[from=1-1, to=2-1]
	\arrow["{\Pi_1(1_{\{x_0, x_n\}} \times_{1_{U^2}} U^{R_n})}", from=2-1, to=2-3]
	\arrow[from=2-3, to=3-3]
\end{tikzcd}\]
which, in the end, amounts to composing a sequence of $n$ paths in a topological space into one path. This composition, by the way $R_n$ was defined, is `unbiased' - the composite path $[0, 1] \rightarrow U$ sends the interval $[\frac{i}{n}, \frac{i+1}{n}]$ to the $i^{th}$ path in the chain by the map $x \mapsto nx - i$.

We should note that this is but one possible choice of composition operation. In our approach, we established that choice of composition could be built upon a choice of retract for the inclusion $Sp_t(n) \hookrightarrow \Delta_t[n]$. There are in fact an uncountably infinite number of such retracts. For instance, if $n = 2$, choosing any $t \in (0, 1)$ induces a distinct retract $f_t : \Delta_t[2] \rightarrow Sp_t(2)$ sending
$$
    (a, b, c) \mapsto \begin{cases}
        (1 - \frac{bt + c}{t}, \frac{bt + c}{t}, 0) & bt + c < t \\
        (0, 1 - \frac{bt + c - t}{1 - t}, \frac{bt + c - t}{1 - t}) & bt + c \geq t.
    \end{cases}
$$
This results in a composition of paths which, as before, concatenates the paths in an affine manner, but differs in the parameterization of the concatenation. More precisely, the parameter $t$ specifies the point in $[0, 1]$ where the first path ends and the second begins in the composite. We can generalize this phenomenon to all $\Delta_t[n]$ in a straightforward way.

Of course, the space of all retracts $\Delta_t[n] \rightarrow Sp_t[n]$ of $\lvert g_n \rvert$ need not contain only piecewise linear elements or even only smooth ones. Moreover, our particular construction of the homotopy bicategory of $S$ may have used these retracts, but there is certainly no need to rely on these to obtain the necessary lifts at all. Any section of $U^{\lvert g_n \rvert} : U^{\Delta_t[n]} \rightarrow U^{Sp_t(n)}$ will do. One could go further and say the section need not result from a map on these underlying topological spaces - it need only be defined on the resulting Segal spaces, so behavior levelwise could potentially vary drastically.

We now turn to associators in $h_2(S)$, as unitors are of course trivial. In order to work within $\textbf{Top}$ efficiently, we need to show that one may convert from homotopies defined using the interval $[0, 1]$ to ones defined using $N(I[1])$.

\begin{proposition} \label{prop_h2_nervecyl_to_interval}
    There is a weak equivalence of Segal spaces $\tau : N(I[1]) \rightarrow \textbf{Sing}_{\textbf{sS}}([0, 1])$ such that the diagram
    \[\begin{tikzcd}
    	{\ast \sqcup \ast} & {\textbf{Sing}_{\textbf{sS}}([0, 1])} \\
    	{N(I[1])} & \ast
    	\arrow[from=1-1, to=1-2]
    	\arrow[from=1-2, to=2-2]
    	\arrow[from=1-1, to=2-1]
    	\arrow[from=2-1, to=2-2]
    	\arrow["\tau"{description}, dashed, from=2-1, to=1-2]
    \end{tikzcd}\]
    commutes.
\end{proposition}

\begin{proof}
    Note that this is a valid lifting problem in $\textbf{sSpace}^{inj}$, since $\textbf{Sing}_{\textbf{sS}}(-)$ preserves weak equivalences levelwise and has its image in $\textbf{SeSp}^{inj}$, so that the rightmost vertical map is a trivial fibration in $\textbf{sSpace}^{inj}$. Thus, the map $\tau$ is simply a solution of a lifting problem, so will exist. Moreover, it must be a weak equivalence by 2-out-of-3.
\end{proof}

Further discussion of a concrete implementation of $\tau$, together with some relevant references, may be found in the proof of \cite[Prop. 5.1.1]{romoTowardsAlgebraicNCategories} and the material immediately following it.

This means, given we specify our homotopies classically in $\textbf{Top}$, we can transmit them by $\textbf{Sing}_{\textbf{sS}}$ and precomposition with $S$ to the format of left homotopy we have built $h_2$ upon.

Suppose then we have chosen sections $p_n : U^{Sp_t(n)} \rightarrow U^{\Delta_t[n]}$ of $U^{\lvert g_n \rvert}$. It will suffice, for each $n > 0$ and $k_1, \cdots, k_n \geq 0$ with $r := \sum_i k_i$, to find a homotopy
$$
    U^{Sp_t(r)} \times [0, 1] \rightarrow U^{(\Delta_t[k_1] \sqcup_{\Delta_t[0]} \cdots \sqcup_{\Delta_t[0]} \Delta_t[k_n]) \sqcup_{Sp_t(n)} \Delta_t[n]}
$$
from $U^{\lvert f_{k_1, \cdots, k_n} \rvert} \circ p_r$ to
$$
    \big((1_{U^{\Delta_t[k_1]}} \times_{1_U} \cdots \times_{1_U} 1_{U^{\Delta_t[k_n]}}) \times_{1_{U^{Sp_t(n)}}} p_n \big) \circ \big(p_{k_1} \times_{1_U} \cdots \times_{1_U} p_{k_n}\big)
$$
that is constant on postcomposition with the natural map to $U^{Sp_t(r)}$. Both of these maps are again sections of a trivial Hurewicz fibration, so there will necessarily be a homotopy between them.

In the simple case that $p_n := U^{R_n}$, our challenge shrinks to the problem of finding a homotopy from 
$U^{\lvert f_{k_1, \cdots, k_n} \rvert} \circ U^{R_r}$ to $U^{(id \sqcup_{id} \cdots \sqcup_{id} id) \sqcup_{id} R_n} \circ U^{(R_{k_1} \sqcup_{id} \cdots \sqcup_{id} R_{k_n})}$. It then suffices to construct a homotopy of the form
$$
    \Big( (\Delta_t[k_1] \sqcup_{\Delta_t[0]} \cdots \sqcup_{\Delta_t[0]} \Delta_t[k_n]) \sqcup_{Sp_t(n)} \Delta_t[n] \Big) \times [0, 1] \rightarrow Sp_t(r)
$$
from $R_r \circ \lvert f_{k_1, \cdots, k_n} \rvert$ to $(R_{k_1} \sqcup_{id} \cdots \sqcup_{id} R_{k_n}) \circ (id \sqcup_{id} R_n)$. We choose to identify $Sp_t(r)$ with $[0, r]$ in the evident way, sending $e_i \mapsto i$, and set our homotopy to be linear interpolation.

To understand more concretely what our homotopy does, consider the case $n = 2$ and $k_1 = 1, k_2 = 2$. The homotopy is from the composition map $((- \circ -) \circ -)$ to $(- \circ - \circ -)$. We have reduced this to a homotopy of retracts
$$
    \Big( (\Delta_t[1] \sqcup_{\Delta_t[0]} \Delta_t[2]) \sqcup_{Sp_t(2)} \Delta_t[2] \Big) \times [0, 1] \rightarrow [0, 3] \cong Sp_t(3).
$$
The important behavior of the homotopy is on the unit square $\Delta_t[1] \times [0, 1]$ in the domain, identified by the inclusion
$$
    \Delta_t[1] \xrightarrow{\lvert \langle 0, 2 \rangle \rvert} \Delta_t[2] \cong (\emptyset \sqcup_{\emptyset} \emptyset) \sqcup_{\emptyset} \Delta_t[2] \hookrightarrow \Big((\Delta_t[1] \sqcup_{\Delta_t[0]} \Delta_t[2]) \sqcup_{Sp_t(2)} \Delta_t[2] \Big)
$$
as the image of this path is the end result of the composition operation. We find that $R_3 \circ \lvert f_{1, 2} \rvert$ acts on this interval as the morphism $[0, 1] \rightarrow [0, 3]$ sending $x \mapsto 3x$, while $(R_1 \sqcup_{id} R_2) \circ (id \sqcup_{id} R_2)$ is the piecewise linear map sending $0 \mapsto 0$, $\frac{1}{2} \mapsto 1$, $\frac{3}{4} \mapsto 2$ and $1 \mapsto 3$. Our associator is a piecewise linear interpolation between these two maps.

As we might expect, there are many possible choices of homotopy to exhibit the associators in $h_2(S)$. Linear interpolation is but one option; a suitable homotopy may be only polynomial, smooth or merely continuous. However, all of these will produce identical associators, as the induced natural isomorphisms will levelwise be homotopic paths. In order to generate truly distinct associators, they must be given as natural isomorphisms that are levelwise paths which are not homotopic, which is not possible by the given approach nor indeed by any method within the confines of our constructions of homotopy bicategories.

Now we consider the production of pseudofunctors. We add the following definition:

\begin{definition}
    Let $\textbf{Top}^{comp}$ be the category of pairs $(T, (p_n)_{n \geq 0})$ where $T \in \textbf{Top}$ and $p_n : T^{Sp_t(n)} \rightarrow T^{\Delta_t[n]}$ is a section of the map $T^{\lvert g_n \rvert}$, together with maps $(T, p_n) \rightarrow (U, q_n)$ given by continuous maps $T \rightarrow U$ in $\textbf{Top}$.
\end{definition}

We now have a functor
$$
    \textbf{Sing}_{\textbf{ssS}}^{comp} : \textbf{Top}^{comp} \rightarrow \textbf{SeSp}_2^{comp}
$$
that extends $\textbf{Sing}_{\textbf{ssS}}$ by inducing horizontal compositions from the maps $p_n$.

Consider a continuous map $f : (T, p_n) \rightarrow (U, q_n)$ in $\textbf{Top}$. Let $S_T := \textbf{Sing}_{\textbf{ssS}}(T)$ and $S_U := \textbf{Sing}_{\textbf{ssS}}(U)$. We already have specified horizontal compositions in $S_U$ and $S_T$. Write $\mu_n$ for the horizontal compositions induced by $p_n$ and likewise $\nu_n$ for those from $q_n$. We thus obtain a map
$$
    F := \textbf{Sing}_{\textbf{ssS}}(f) : (S_T, \mu_n) \rightarrow (S_U, \nu_n).
$$
We will write $F$ for the underlying map $S_T \rightarrow S_U$ as well.

The pseudofunctor $h_2(F) : h_2(S_T, (\mu_n)_{n \geq 0}) \rightarrow h_2(S_U, (\nu_n)_{n \geq 0})$ evidently sends $x \mapsto f(x)$ for objects $x \in T$. On hom-categories, for $x, y \in T$, we have the induced functor
$$
    \Pi_1(\{(x, y)\} \times f^{\Delta_t[1]}) : \Pi_1(\{(x, y)\} \times T^{\Delta_t[1]}) \rightarrow \Pi_1(\{(f(x), f(y))\} \times U^{\Delta_t[1]})
$$
which sends a $1$-morphism $p : [0, 1] \rightarrow T$ from $x$ to $y$ to the path $f \circ p$. Moreover, it sends a homotopy class of paths $[H]$ for $H : [0, 1] \times [0, 1] \rightarrow T$ to $[f \circ H]$.

Vertical composition is respected on the nose. For horizontal composition, we must understand the compositors $\pi_{x_0, \cdots, x_n}$. It suffices to find a homotopy
$$
    T^{Sp_t(n)} \times [0, 1] \rightarrow U^{\Delta_t[n]}
$$
from $f^{\Delta_t[n]} \circ p_n$ to $q_n \circ f^{Sp_t(n)}$ that is constant on the vertices. Such a homotopy will necessarily exist, as both are solutions to the same lifting problem
\[\begin{tikzcd}
	&& {U^{\Delta_t[n]}} \\
	\\
	{T^{Sp_t(n)}} && {U^{Sp_t(n)}}
	\arrow[from=1-3, to=3-3]
	\arrow["{f^{\Delta_t[n]} \circ p_n}"{description}, curve={height=-12pt}, dashed, from=3-1, to=1-3]
	\arrow["{q_n \circ f^{Sp_t(n)}}"{description}, curve={height=12pt}, dashed, from=3-1, to=1-3]
	\arrow[from=3-1, to=3-3]
\end{tikzcd}\]
in the \emph{Hurewicz model structure} on $\textbf{Top}$, as defined for instance in \cite[Sec. 17.1]{mayPontoMoreConciseAlgebraic2012}. Any two such homotopies will by Corollary \ref{cor_h2_diffcomps_ident_fctr} induce the same natural isomorphism, so the pseudofunctor $h_2(F)$ is specified entirely.

In the special case that $p_n = T^{R_n}$ and $q_n = U^{R_n}$, this is much simpler. In fact, the two maps being homotoped between are equal, so the natural isomorphisms $\pi_{x_0, \cdots, x_n}$ are all identities.

Now, suppose one introduced a different set of deformation retracts $R_n' : \Delta_t[n] \rightarrow Sp_t(n)$. These induce a different choice of horizontal compositions $(\eta_n)_{n \geq 0}$ on $S_T$. We thus have by Corollary \ref{cor_h2_diffcomps_ident_fctr} that there is a canonical pseudofunctor $P : h_2(S_T, (\mu_n)_{n \geq 0}) \rightarrow h_2(S_T, (\eta_n)_{n \geq 0})$ that is the identity on objects and hom-categories. Indeed, $P$ is induced by the identity morphism $1_T : T \rightarrow T$.

The only interesting aspect of this pseudofunctor's structure is in the natural isomorphisms $\pi_{x_0, \cdots, x_n}$. These will be induced by homotopies
$$
    T^{Sp_t(n)} \times [0, 1] \rightarrow T^{\Delta_t[n]}
$$
from $U^{R_n}$ to $U^{R_n'}$. For such purposes, it suffices to find a left homotopy
$$
    \Delta_t[n] \times [0, 1] \rightarrow Sp_t(n)
$$
from $R_n$ to $R_n'$ that preserves $Sp_t(n)$. We may choose our left homotopy to be linear interpolation pointwise, choosing to interpret $Sp_t(n) \cong [0, n]$. We find thus that, for a sequence of paths $x_0 \xrightarrow{f_1} x_1 \cdots \xrightarrow{f_n} x_n$, the $2$-morphism $\pi_{(f_1, \cdots, f_n)}$ is the homotopy class $[H]$ of the map $H : [0, 1] \rightarrow U^{[0, 1]}$ such that $0 \mapsto (f_0 \sqcup_* \cdots \sqcup_* f_n) \circ R_n$ and $1 \mapsto (f_0 \sqcup_* \cdots \sqcup_* f_n) \circ R_n'$, with all other points defined by the linear interpolation between $R_n$ and $R_n'$. Any other homotopy will yield the same natural isomorphism, so this description covers all possibilities.

In the end, composing $\textbf{Sing}_{\textbf{ssS}}^{comp}$ and $h_2$ yields a `fundamental bigroupoid' functor
$$
    \Pi_2^{comp} : \textbf{Top}^{comp} \rightarrow \textbf{UBicat}.
$$
Using some coherent choice of maps $p_n$, for instance those induced by $R_n$, allows us to make the domain $\textbf{Top}$. This then yields a functor
$$
    \Pi_2 : \textbf{Top} \rightarrow \textbf{UBicat}.
$$

{\textbf{Data Access Statement:} No data was collected or used in this project. All commutative diagrams were generated with \url{https://q.uiver.app/}.}

\printbibliography

\end{document}